\documentclass[12 pt]{amsart}
\usepackage{amssymb, amsmath, amsfonts, amsthm, graphics,mathrsfs, mathtools}
\usepackage[usenames, dvipsnames]{xcolor}
\usepackage[hmargin = 1 in, vmargin = 1 in]{geometry}
\usepackage{tikz}

\usetikzlibrary{matrix, calc, arrows, positioning, backgrounds} 
\usepackage{hyperref}
\usepackage[capitalise,noabbrev]{cleveref}
\crefformat{equation}{(#2#1#3)}
\Crefformat{equation}{(#2#1#3)}

\usepackage{thmtools}
\usepackage{thm-restate}

\usepackage{ytableau}

\definecolor{darkblue}{rgb}{0.0,0,0.7}

\newcommand{\newword}[1]{\textcolor{darkblue}{\textbf{\emph{#1}}}}

\usepackage[colorinlistoftodos]{todonotes}

\newcommand{\Flags}{\mathsf{Flags}}

\newcommand{\id}{\mathrm{id}}

\newcommand{\redword}{{\sf red}}
\newcommand{\heckeword}{{\sf hecke}}

\newcommand{\NE}{{\sf NE}}

\newcommand{\rk}{\mathrm{rk}}

\newcommand{\dom}{\mathrm{dom}}

\newcommand{\des}{\mathrm{des}}

\newcommand{\paths}{{\sf paths}}
\newcommand{\word}{{\sf word}}

\newcommand{\addable}{\alpha}

\newcommand{\partitionsubset}{{\mathbb Y}}

\newtheorem{Theorem}{Theorem}[section]

\newtheorem{remark}[Theorem]{Remark}

\newtheorem{proposition}[Theorem]{Proposition}

\newtheorem{theorem}[Theorem]{Theorem}

\newtheorem{lemma}[Theorem]{Lemma}

\numberwithin{equation}{section}

\theoremstyle{remark}
\newtheorem{examplex}[Theorem]{Example}
\newenvironment{example}
{\pushQED{\qed}\examplex}
{\popQED\endexamplex}

\newcommand{\pipes}{{\sf Pipes}}
\newcommand{\kpipes}{\overline{\sf Pipes}}
\newcommand{\bpd}{{\sf BPD}}
\newcommand{\kbpd}{\overline{\sf BPD}}

\newcommand{\cobpd}{{\sf coBPD}}
\newcommand{\kcobpd}{\overline{\sf coBPD}}

\newcommand{\copipes}{{\sf coPipes}}
\newcommand{\kcopipes}{\overline{\sf coPipes}}

\newcommand{\demprod}{\delta}

\newcommand{\rothe}{{\rm D}}
\newcommand{\mute}{{\rm mute}}
\newcommand{\up}{{\rm U}}
\newcommand{\wt}{{\tt wt}}
\newcommand{\kwt}{{\tt wt}_K}

\newcommand{\divdiff}{N}
\newcommand{\kdivdiff}{\overline{N}}

\newcommand{\cotransitionperms}{\Phi}
\newcommand{\cotransitionindices}{\phi}
\newcommand{\kcotransitionperms}{\overline{\Phi}}

\newcommand{\maybeaddspace}{%
	\futurelet\nexttoken\checknexttoken
}

\newcommand{\checknexttoken}{%
	\ifx\nexttoken.\else
	\ifx\nexttoken,\else
	\ifx\nexttoken;\else
	\ifx\nexttoken:\else
	\ifx\nexttoken\}\else
	\;\fi
	\fi
	\fi
	\fi
	\fi
}

\newcommand{\drtile}{\begin{tikzpicture}[x = .75em,y = .75em]
		\draw[step = 1,gray, very thin] (0,0) grid (1, -1);
		\draw[color = black, thick] (0,0) rectangle (1, -1);
		\draw[thick,rounded corners,color = blue](1/2, -1)--(1/2, -1/2)--(1, -1/2);
	\end{tikzpicture}\maybeaddspace}

\newcommand{\ultile}{\begin{tikzpicture}[x = .75em,y = .75em]
		\draw[step = 1,gray, very thin] (0,0) grid (1, -1);
		\draw[color = black, thick] (0,0) rectangle (1, -1);
		\draw[thick,rounded corners,color = blue](1/2, 0)--(1/2, -1/2)--(0, -1/2);
	\end{tikzpicture}\maybeaddspace}

\newcommand{\htile}{\begin{tikzpicture}[x = .75em,y = .75em]
		\draw[step = 1,gray, very thin] (0,0) grid (1, -1);
		\draw[color = black, thick] (0,0) rectangle (1, -1);
		\draw[thick,rounded corners,color = blue](0, -1/2)--(1, -1/2);
	\end{tikzpicture}\maybeaddspace}

\newcommand{\vtile}{\begin{tikzpicture}[x = .75em,y = .75em]
		\draw[step = 1,gray, very thin] (0,0) grid (1, -1);
		\draw[color = black, thick] (0,0) rectangle (1, -1);
		\draw[thick,rounded corners,color = blue](1/2, 0)--(1/2, -1);
	\end{tikzpicture}\maybeaddspace}

\newcommand{\ctile}{\begin{tikzpicture}[x = .75em,y = .75em]
		\draw[step = 1,gray, very thin] (0,0) grid (1, -1);
		\draw[color = black, thick] (0,0) rectangle (1, -1);
		\draw[thick,rounded corners,color = blue](0, -1/2)--(1, -1/2);
		\draw[thick,rounded corners,color = blue](1/2, 0)--(1/2, -1);
	\end{tikzpicture}\maybeaddspace}

\newcommand{\btile}{\begin{tikzpicture}[x = .75em,y = .75em]
		\draw[step = 1,gray, very thin] (0,0) grid (1, -1);
		\draw[color = black, thick] (0,0) rectangle (1, -1);
	\end{tikzpicture}\maybeaddspace}

\newcommand{\bumptile}{\begin{tikzpicture}[x = .75em,y = .75em]
		\draw[step = 1,gray, very thin] (0,0) grid (1, -1);
		\draw[color = black, thick] (0,0) rectangle (1, -1);
		\draw[thick,rounded corners,color = blue](1/2, -1)--(1/2, -1/2)--(1, -1/2);
		\draw[thick,rounded corners,color = blue](1/2, 0)--(1/2, -1/2)--(0, -1/2);
	\end{tikzpicture}\maybeaddspace}

\newcommand{\cohtile}{\begin{tikzpicture}[x = .75em,y = .75em]
		\draw[step = 1,gray, very thin] (0,0) grid (1, -1);
		\draw[color = black, thick] (0,0) rectangle (1, -1);
		\draw[thick,rounded corners,color = ForestGreen](0, -1/2)--(1, -1/2);
	\end{tikzpicture}\maybeaddspace}

\newcommand{\covtile}{\begin{tikzpicture}[x = .75em,y = .75em]
		\draw[step = 1,gray, very thin] (0,0) grid (1, -1);
		\draw[color = black, thick] (0,0) rectangle (1, -1);
		\draw[thick,rounded corners,color = ForestGreen](1/2, 0)--(1/2, -1);
	\end{tikzpicture}\maybeaddspace}

\newcommand{\courtile}{\begin{tikzpicture}[x = .75em,y = .75em]
		\draw[step = 1,gray, very thin] (0,0) grid (1, -1);
		\draw[color = black, thick] (0,0) rectangle (1, -1);
		\draw[thick,rounded corners,color = ForestGreen](1/2, 0)--(1/2, -1/2)--(1, -1/2);
	\end{tikzpicture}\maybeaddspace}

\newcommand{\coultile}{\begin{tikzpicture}[x = .75em,y = .75em]
		\draw[step = 1,gray, very thin] (0,0) grid (1, -1);
		\draw[color = black, thick] (0,0) rectangle (1, -1);
		\draw[thick,rounded corners,color = ForestGreen](1/2, 0)--(1/2, -1/2)--(0, -1/2);
	\end{tikzpicture}\maybeaddspace}

\newcommand{\pipeultile}{
\begin{tikzpicture}[x = .75em,y = .75em]
	\draw[step = 1,gray, very thin] (0,0) grid (1, -1);
	\draw[color = black, thick] (0,0) rectangle (1, -1);
	\draw[thick,rounded corners,color = Mulberry](1/2, 0)--(1/2, -1/2)--(0, -1/2);
\end{tikzpicture}\maybeaddspace}

\newcommand{\pipecross}{\begin{tikzpicture}[x = .75em,y = .75em]
		\draw[step = 1,gray, very thin] (0,0) grid (1, -1);
		\draw[color = black, thick] (0,0) rectangle (1, -1);
		\draw[thick,rounded corners,color = Mulberry](0, -1/2)--(1, -1/2);
		\draw[thick,rounded corners,color = Mulberry](1/2, 0)--(1/2, -1);
	\end{tikzpicture}\maybeaddspace}

\newcommand{\pipebump}{\begin{tikzpicture}[x = .75em,y = .75em]
		\draw[step = 1,gray, very thin] (0,0) grid (1, -1);
		\draw[color = black, thick] (0,0) rectangle (1, -1);
		\draw[thick,rounded corners,color = Mulberry](1/2, -1)--(1/2, -1/2)--(1, -1/2);
		\draw[thick,rounded corners,color = Mulberry](1/2, 0)--(1/2, -1/2)--(0, -1/2);
	\end{tikzpicture}\maybeaddspace}

\newcommand{\dltile}{\begin{tikzpicture}[x = .75em,y = .75em]
		\draw[step = 1,gray, very thin] (0,0) grid (1, -1);
		\draw[color = black, thick] (0,0) rectangle (1, -1);
		\draw[thick,rounded corners,color = Rhodamine](1/2, -1)--(1/2, -1/2)--(0, -1/2);
	\end{tikzpicture}\maybeaddspace}

\newcommand{\copipecross}{\begin{tikzpicture}[x = .75em,y = .75em]
	\draw[step = 1,gray, very thin] (0,0) grid (1, -1);
	\draw[color = black, thick] (0,0) rectangle (1, -1);
	\draw[thick,rounded corners,color = Rhodamine](0, -1/2)--(1, -1/2);
	\draw[thick,rounded corners,color = Rhodamine](1/2, 0)--(1/2, -1);
\end{tikzpicture}\maybeaddspace}

\newcommand{\copipebump}{\begin{tikzpicture}[x = .75em,y = .75em]
		\draw[step = 1,gray, very thin] (0,0) grid (1, -1);
		\draw[color = black, thick] (0,0) rectangle (1, -1);
		\draw[thick,rounded corners,color = Rhodamine](1/2, 0)--(1/2, -1/2)--(1, -1/2);
		\draw[thick,rounded corners,color = Rhodamine](0, -1/2)--(1/2, -1/2)--(1/2, -1);
	\end{tikzpicture}\maybeaddspace}

\title{Changing Bases with Pipe Dream Combinatorics}

\author{Anna Weigandt}
\address{School of Mathematics, University of Minnesota, Minneapolis MN 55455}
\email{weigandt@umn.edu}

\date{\today}

\makeatletter
\@namedef{subjclassname@2020}{%
	\textup{2020} Mathematics Subject Classification}
\makeatother


\begin{document}

	\begin{abstract}
		Lascoux and Sch\"utzenberger introduced Schubert and Grothendieck polynomials to study the cohomology and K-theory of the complete flag variety. 
		We present explicit combinatorial rules for expressing Grothendieck polynomials in the basis of Schubert polynomials, and vice versa, using the bumpless pipe dreams (BPDs) of Lam, Lee, and Shimozono. 
		A key advantage of BPDs is that they are naturally back stable, which allows us to give a combinatorial formula for expanding back stable Grothendieck polynomials in terms of back stable Schubert polynomials.
		We also provide pipe dream interpretations for the rules originally given by Lenart (Grothendieck to Schubert) and Lascoux (Schubert to Grothendieck), which were previously formulated in terms of binary triangular arrays. 
		We give new proofs of these results, relying on Knutson's co-transition recurrences. 
		As a consequence, we obtain a formula for expanding Grothendieck polynomials into Schubert polynomials using chains in Bruhat order. 
		The key connection between the pipe dream and BPD change of basis formulas is the canonical bijection of Gao and Huang. 
		We show that co-permutations are preserved by this map.
	\end{abstract}

	\maketitle 
	
	\section{Introduction}
	
	The \newword{complete flag variety} $\Flags(n)$ is the space of nested sequences of vector subspaces of $\mathbb{C}^n$ of the form
	\[ V_1 \subset V_2 \subset \cdots \subset V_n=\mathbb{C}^n,\] 
	where $\dim(V_i)=i$ for all $i$. 
	The flag variety has distinguished subvarieties called \emph{Schubert varieties}, which are indexed by permutations in the \emph{symmetric group} $S_n$. 
	Each Schubert variety determines a class $\sigma_w\in \mathrm{H}^*(\Flags(n))$ in the cohomology ring of $\Flags(n)$. 
	These \emph{Schubert classes} form a linear basis for $\mathrm{H}^*(\Flags(n))$. 
	A central problem in Schubert calculus is to find a combinatorial rule for the structure constants $c_{u,v}^w$ in the product 
	\[\sigma_u \cdot \sigma_v=\sum_{w\in S_n}c_{u,v}^w \sigma_w.\] 
	
	The Borel isomorphism identifies $\mathrm{H}^*(\Flags(n))$ with $\mathbb{Z}[x_1, \ldots, x_n]/I$, where $I$ is the ideal generated by nonconstant elementary symmetric polynomials. 
	Lascoux and Sch\"utzenberger \cite{Lascoux.Schutzenberger:Schubert} introduced \emph{Schubert polynomials} $\mathfrak{S}_w$, which are representatives for the Schubert classes. 
	There is an analogous story in K-theory. 
	Here, \emph{Grothendieck polynomials} $\mathfrak{G}_w$ serve as representatives for the classes of structure sheaves in the K-theory of $\Flags(n)$ \cite{Lascoux.Schutzenberger}. 
	
	The purpose of this article is to present combinatorial formulas for expanding Grothendieck polynomials in the basis of Schubert polynomials, and vice versa. We give formulas for changing bases between Schubert and Grothendieck polynomials as sums over the \emph{bumpless pipe dreams (BPDs)} of Lam, Lee, and Shimozono \cite{Lam.Lee.Shimozono}. We also translate the change of basis formulas of Lenart \cite{Lenart} and Lascoux \cite{Lascoux.03}, originally stated in terms of binary triangular arrays, into sums over \emph{pipe dreams}.
	Pipe dreams and BPDs are certain tilings of the $n \times n$ grid, both of which are used in formulas for computing the monomial expansions of Schubert and Grothendieck polynomials \cite{BJS93, Bergeron.Billey, Fomin.Stanley, Fomin.Kirillov, Fomin.Kirillov.96, Lam.Lee.Shimozono, Weigandt}. 
	Pipe dreams have greatly contributed to our combinatorial understanding of Schubert calculus, see, for instance, \cite{Knutson.Miller, Kogan.Miller, Assaf.Searles, Knutson:cotransition}.
	BPDs were originally developed to study back stable Schubert calculus \cite{Lam.Lee.Shimozono,Lam.Lee.Shimozono-K-theory}. 
	Since then, BPDs have become an important tool in Schubert calculus see, e.g., \cite{Buciumas.Scrimshaw, Huang.2023, Huang, Huang.Striker, LOTRZ}.

	Certain properties of Schubert and Grothendieck polynomials appear more transparently in terms of pipe dreams or bumpless pipe dreams. 
	For instance, the transition recurrence on Schubert and Grothendieck polynomials has a simple bijective explanation in terms of BPDs (see \cite{Lascoux, Weigandt}), whereas the co-transition recurrence of Knutson \cite{Knutson:cotransition} is compatible with pipe dreams. 
	Similarly, pipe dreams index components in antidiagonal Gr\"obner degenerations of \emph{matrix Schubert varieties} \cite{Knutson.Miller}, while BPDs govern certain diagonal degenerations \cite{Knutson.Miller.Yong, Hamaker.Pechenik.Weigandt, Klein, Klein.Weigandt}.
	We show that the BPD and pipe dream change of basis formulas are closely connected via the canonical bijection from pipe dreams to bumpless pipe dreams of Gao and Huang \cite{Gao.Huang}.

	\subsection{Description of main results}

	We now summarize our main results.
	We give only a brief overview here and refer the reader to \cref{section:planar} and \cref{section:pipes} for precise definitions.
	Given a BPD $\mathcal{B}$, there is a natural way to obtain a permutation $\demprod(\mathcal{B})$ by reading along the pipes, ignoring crossings whenever two pipes have previously crossed. 
	Our new ingredient is the following: to each BPD $\mathcal{B}$, we associate a co-BPD $\check{\mathcal{B}}$ by making the tile-by-tile replacements pictured below.
		
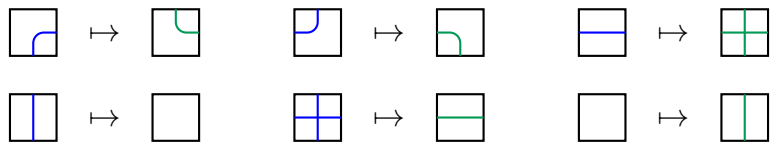
\[\begin{tikzpicture}[x = 1.5em,y = 1.5em]
		\draw[step = 1,gray, very thin] (0,0) grid (1, -1);
		\draw[color = black, thick] (0,0) rectangle (1, -1);
		\draw[thick,rounded corners,color = blue](1/2, -1)--(1/2, -1/2)--(1, -1/2);
	\end{tikzpicture} 
	\quad 
	\raisebox{.5em}{$\mapsto$}
	\quad
	\begin{tikzpicture}[x = 1.5em,y = 1.5em]
		\draw[step = 1,gray, very thin] (0,0) grid (1, -1);
		\draw[color = black, thick] (0,0) rectangle (1, -1);
		\draw[thick,rounded corners,color = ForestGreen](1/2, 0)--(1/2, -1/2)--(1, -1/2);
	\end{tikzpicture} 
	\hspace{3em}
	\begin{tikzpicture}[x = 1.5em,y = 1.5em]
		\draw[step = 1,gray, very thin] (0,0) grid (1, -1);
		\draw[color = black, thick] (0,0) rectangle (1, -1);
		\draw[thick,rounded corners,color = blue](1/2, 0)--(1/2, -1/2)--(0, -1/2);
	\end{tikzpicture}
	\quad 
	\raisebox{.5em}{$\mapsto$}
	\quad
	\begin{tikzpicture}[x = 1.5em,y = 1.5em]
		\draw[step = 1,gray, very thin] (0,0) grid (1, -1);
		\draw[color = black, thick] (0,0) rectangle (1, -1);
		\draw[thick,rounded corners,color = ForestGreen](1/2, -1)--(1/2, -1/2)--(0, -1/2);
	\end{tikzpicture}
	\hspace{3em}
	\begin{tikzpicture}[x = 1.5em,y = 1.5em]
		\draw[step = 1,gray, very thin] (0,0) grid (1, -1);
		\draw[color = black, thick] (0,0) rectangle (1, -1);
		\draw[thick,rounded corners,color = blue](0, -1/2)--(1, -1/2);
	\end{tikzpicture}
	\quad 
	\raisebox{.5em}{$\mapsto$}
	\quad
	\begin{tikzpicture}[x = 1.5em,y = 1.5em]
		\draw[step = 1,gray, very thin] (0,0) grid (1, -1);
		\draw[color = black, thick] (0,0) rectangle (1, -1);
		\draw[thick,rounded corners,color = ForestGreen](0, -1/2)--(1, -1/2);
		\draw[thick,rounded corners,color = ForestGreen](1/2, 0)--(1/2, -1);
	\end{tikzpicture}
	\]\[
	\begin{tikzpicture}[x = 1.5em,y = 1.5em]
		\draw[step = 1,gray, very thin] (0,0) grid (1, -1);
		\draw[color = black, thick] (0,0) rectangle (1, -1);
		\draw[thick,rounded corners,color = blue](1/2, 0)--(1/2, -1);
	\end{tikzpicture}
	\quad 
	\raisebox{.5em}{$\mapsto$}
	\quad
	\begin{tikzpicture}[x = 1.5em,y = 1.5em]
		\draw[step = 1,gray, very thin] (0,0) grid (1, -1);
		\draw[color = black, thick] (0,0) rectangle (1, -1);
	\end{tikzpicture}
	\hspace{3em}
	\begin{tikzpicture}[x = 1.5em,y = 1.5em]
		\draw[step = 1,gray, very thin] (0,0) grid (1, -1);
		\draw[color = black, thick] (0,0) rectangle (1, -1);
		\draw[thick,rounded corners,color = blue](0, -1/2)--(1, -1/2);
		\draw[thick,rounded corners,color = blue](1/2, 0)--(1/2, -1);
	\end{tikzpicture}
	\quad 
	\raisebox{.5em}{$\mapsto$}
	\quad
	\begin{tikzpicture}[x = 1.5em,y = 1.5em]
		\draw[step = 1,gray, very thin] (0,0) grid (1, -1);
		\draw[color = black, thick] (0,0) rectangle (1, -1);
		\draw[thick,rounded corners,color = ForestGreen](0, -1/2)--(1, -1/2);
	\end{tikzpicture}
	\hspace{3em}
	\begin{tikzpicture}[x = 1.5em,y = 1.5em]
		\draw[step = 1,gray, very thin] (0,0) grid (1, -1);
		\draw[color = black, thick] (0,0) rectangle (1, -1);
	\end{tikzpicture}
	\quad 
	\raisebox{.5em}{$\mapsto$}
	\quad
	\begin{tikzpicture}[x = 1.5em,y = 1.5em]
		\draw[step = 1,gray, very thin] (0,0) grid (1, -1);
		\draw[color = black, thick] (0,0) rectangle (1, -1);
		\draw[thick,rounded corners,color = ForestGreen](1/2, 0)--(1/2, -1);
	\end{tikzpicture}\] 
	We then assign a permutation to each co-BPD (again by following the pipes in a certain way), and write $\demprod(\check{\mathcal{B}})$ for this permutation. 
	We call a BPD or co-BPD \newword{reduced} if each pair of pipes crosses at most one time. 
	For each permutation $w$, we denote the set of associated BPDs by $\kbpd(w)$, and write $\bpd(w)$ for the subset of reduced BPDs.
	
	Our first main theorem states that we may expand the Grothendieck polynomial $\mathfrak{G}_w$ as a signed sum over Schubert polynomials, indexed by permutations associated to the reduced co-BPDs arising from BPDs of $w$. 
	\begin{restatable}{thm}{thmmainA}
		\label{thm:main1}
		Given $w \in S_n$, we have 
		\[\mathfrak{G}_w = \sum_{\substack{\mathcal{B} \in \kbpd(w)\\ \check{\mathcal{B}} \text{ is reduced}}} (-1)^{\ell(\demprod(\check{\mathcal{B}}))-\ell(w)} \mathfrak{S}_{\demprod(\check{\mathcal{B}})}.\]
	\end{restatable}
	Although this formula involves signs, they alternate with degree, and the monomial expansion is cancellation free. 
	
	We also show that each Schubert polynomial $\mathfrak{S}_w$ expands as a positive sum of Grothendieck polynomials, indexed by permutations arising from the co-BPDs associated to the reduced BPDs of $w$. 
	\begin{restatable}{thm}{thmmainB}
		\label{thm:main2}
		Given $w \in S_n$, we have 
		\[\mathfrak{S}_w = \sum_{\mathcal{B} \in \bpd(w)} \mathfrak{G}_{\demprod(\check{\mathcal{B}})}.\]
	\end{restatable}
	At the level of monomials, this formula is not generally cancellation free.

	We now give an example to illustrate the main theorems.
	
	\begin{example}\label{example:bpd1423}
		Let $w = 1423$. 
		The BPDs of $w$ are listed below.
		\[

		\]
		All BPDs of $w$ are reduced. 
		We also list the corresponding co-BPDs, displayed in the same relative order.
		\[
		%
		\]
		Among these, the first two co-BPDs are reduced, but the third is not. Thus, \cref{thm:main1} says that 
		\[\mathfrak{G}_{1423} = \mathfrak{S}_{1423}-\mathfrak{S}_{2413}.\]
		Also, \cref{thm:main2} says that 
		\[\mathfrak{S}_{1423} = \mathfrak{G}_{1423} + \mathfrak{G}_{2413} + \mathfrak{G}_{3412}. \qedhere\]
	\end{example}
	
	We give another example below.
	
	\begin{example}\label{example:bpd2143}
		Let $w = 2143$. The BPDs of $w$ are listed below.
		\[%
\]
		We also list the co-BPDs of $w$, in the same order as their corresponding BPDs. 
		\[%
\]
		The co-BPDs are all reduced. 
		Thus, applying \cref{thm:main1} gives
		\[\mathfrak{G}_{2143} = \mathfrak{S}_{2143}-\mathfrak{S}_{3142}-\mathfrak{S}_{2341} + \mathfrak{S}_{3241}.\] 
		Also, by \cref{thm:main2}, 
		\[\mathfrak{S}_{2143} = \mathfrak{G}_{2143} + \mathfrak{G}_{3142} + \mathfrak{G}_{2341}.\] 
		In this expansion, we do not include a term for the fourth co-BPD because its corresponding BPD is not reduced.
	\end{example}
	See \cref{example:bpd21534} and \cref{ex:13452} for additional illustrations of \cref{thm:main1} and \cref{thm:main2}.

	One advantage of the new BPD-based expansion of Grothendieck polynomials into Schubert polynomials is an application to back stable Schubert calculus. 
	\cref{thm:backstable} provides the first combinatorial formula for expanding back stable Grothendieck polynomials into back stable Schubert polynomials.

	The statements of \cref{thm:main1} and \cref{thm:main2} closely parallel the change of basis theorems of Lenart and Lascoux. 
	These earlier results were originally stated using binary triangular arrays. 
	Here, we provide a reformulation using pipe dreams and co-pipe dreams (see \cref{subsection:pipes} and \cref{subsection:copipes} for these definitions). 
	This pictorial perspective makes it possible to read off the associated permutations directly from the diagrams. 
	We now state these formulas.
	
	If $w\in S_n$, we denote the set of pipe dreams of $w$ by $\kpipes(w)$, and write $\pipes(w)$ for the subset of reduced pipe dreams.
	To a pipe dream $\mathcal{P}$, we associate a co-pipe dream $\check{\mathcal{P}}$ (see \cref{lemma:copipedreambijection} for this map).
	The following result, due to Lenart, gives a formula for expanding a Grothendieck polynomial in the basis of Schubert polynomials.
	\begin{restatable}[{\cite{Lenart}}]{thm}{thmpipeA}
		\label{thm:pipe_groth_to_schub}
		Let $w \in S_n$. Then
		\[\mathfrak{G}_w = \sum_{\substack{\mathcal{P} \in \kpipes(w)\\ \check{\mathcal{P}} \text{ is reduced}}}(-1)^{\ell(\demprod(\check{\mathcal{P}}))-\ell(w)}\mathfrak{S}_{\demprod(\check{\mathcal{P}})}. \]
	\end{restatable}
	Lascoux gave the following formula for expressing a Schubert polynomial as a positive sum of Grothendieck polynomials.
	\begin{restatable}[{\cite{Lascoux.03}}]{thm}{thmpipeB}
		\label{thm:pipe_schub_to_groth}
		Let $w \in S_n$. Then 
		\[\mathfrak{S}_w = \sum_{\mathcal{P} \in \pipes(w)}\mathfrak{G}_{\demprod(\check{\mathcal{P}})}.\]
	\end{restatable}
	We illustrate \cref{thm:pipe_groth_to_schub} and \cref{thm:pipe_schub_to_groth} below.
	
	\begin{example}
		\label{example:pipe1423}
	Let $w = 1423$. The corresponding pipe dreams and their associated co-pipe dreams are shown below.
		\[

		\]
		Of the co-pipe dreams, only the first and fourth are reduced. 
		Lenart's formula says that
		\[\mathfrak{G}_{1423} = \mathfrak{S}_{1423}-\mathfrak{S}_{2413}.\]
		Of the pipe dreams, the first three are reduced. 
		Thus, Lascoux's formula tells us that 
		\[\mathfrak{S}_{1423} = \mathfrak{G}_{1423} + \mathfrak{G}_{2413} + \mathfrak{G}_{3412}.\]
		Both of these equalities agree with the expansions in terms of BPDs from \cref{example:bpd1423}.
	\end{example}

	We give a second example below.
	\begin{example}
		Let $w = 2143$. The pipe dreams of $w$ are pictured below.
		\[%
\]
		
		Here are the associated co-pipe dreams.
		\[%
\]
		Of the co-pipe dreams, the first, fourth, fifth, and seventh are reduced. 
		This tells us that 
		\[\mathfrak{G}_{2143} = \mathfrak{S}_{2143}-\mathfrak{S}_{3142}-\mathfrak{S}_{2341} + \mathfrak{S}_{3241}.\]
		Of the pipe dreams, the first three are reduced. 
		Therefore, 
		\[\mathfrak{S}_{2143} = \mathfrak{G}_{2143} + \mathfrak{G}_{3142} + \mathfrak{G}_{2341}.\]
		Both of these expansions agree with the results from \cref{example:bpd2143}.
	\end{example}
	
	We present new proofs of \cref{thm:pipe_groth_to_schub} and \cref{thm:pipe_schub_to_groth}. 
	Our approach relies on the combinatorial co-transition recurrence of \cite{Knutson:cotransition}. 
	One advantage of these new proofs is that they lead to a method for constructing the pipe dreams for $w$ that have reduced co-pipe dreams recursively in terms of certain chains in Bruhat order (see \cref{section:chains}). 
	This chain theoretic construction is similar in spirit to the climbing chains of \cite{Bergeron.Sottile.2002} and \cite{Lenart.Robinson.Sottile}, which were used to give formulas for Schubert and Grothendieck polynomials, respectively. 
	Though we still sum over a subset of these chains, there are typically fewer chains to check than there are pipe dreams for $w$.

	Our new results on BPDs follow from transferring properties from pipe dreams to BPDs. 
	We use the canonical bijection between these objects, given by Gao and Huang \cite{Gao.Huang}.
	Our key observation is the following:
	\begin{restatable}{thm}{coperm}
		\label{thm:coperm}
		Suppose $\mathcal{P} \in \pipes_n$ maps to $\mathcal{B} \in \bpd_n$ under the column-weight preserving canonical bijection of Gao--Huang. 
		Then $\demprod(\check{\mathcal{P}}) = \demprod(\check{\mathcal{B}})$.
	\end{restatable}
	
	In other words, \cref{thm:coperm} states that the canonical bijection preserves \emph{co-permutations}. 
	This statement provides evidence that, despite pipe dreams and BPDs exhibiting a number of different properties, there is underlying structure governing aspects of their behavior in parallel. 
	We prove \cref{thm:coperm} by analyzing the co-transition recurrence on both pipe dreams and BPDs, focusing on how co-transition modifies the associated co-objects.
	
\subsection{Connections to the literature}

In addition to the binary triangular array formulas of \cite{Lenart,Lascoux.03} for changing bases between Schubert and Grothendieck polynomials, Lenart \cite{Lenart.2000} also gave a tableau formula for the special case of symmetric Grothendieck polynomials and \emph{Schur polynomials}.
The \emph{stable Grothendieck polynomials} of \cite{Fomin.Kirillov} arise as specializations of back stable Grothendieck polynomials. 
The results of Lenart \cite{Lenart.2000} naturally extend to give expansions of stable Grothendieck polynomials indexed by partition shapes into \emph{Schur functions}, and vice versa. 
Chan and Pflueger \cite{Chan.Pflueger} generalized this result to provide analogous expansions for \emph{skew stable Grothendieck polynomials} into \emph{skew Schur functions}, and vice versa. 
Skew stable Grothendieck polynomials (and likewise skew Schur functions) are not linearly independent, thus the resulting expansions are not canonical. 
It would be interesting to compare the result of specializing \cref{thm:backstable} with the work of Chan--Pflueger, but we do not pursue this direction here.

The expansion of Grothendieck polynomials into the Schubert basis has been a useful tool for studying related algebraic and geometric questions.
For instance, Monical, Tokcan, and Yong conjectured that Grothendieck polynomials have saturated Newton polytopes \cite{Monical.Tokcan.Yong}. In the special case of symmetric Grothendieck polynomials, Escobar and Yong \cite{Escobar.Yong} proved the conjecture using Lenart's tableau change of basis formula. 
Although there has been additional progress \cite{Meszaros.StDizier,CCRMM}, the full conjecture remains open. Lenart's formula was also used to prove a combinatorial rule for computing the Castelnuovo--Mumford regularity of Grassmannian matrix Schubert varieties \cite{RRRSW}. 
A number of open questions about the monomial support of Grothendieck polynomials remain, see, for example, \cite{Meszaros.Setiabrata.StDizier}. 
Additional understanding of the Grothendieck to Schubert expansion may be useful for studying these problems.

	\subsection{Organization}
	In \cref{section:background}, we provide necessary background on the symmetric group, Schubert polynomials, and Grothendieck polynomials. 
	We discuss four closely related types of planar histories in \cref{section:planar}. 
	In \cref{section:pipes}, we recall basic facts about pipe dreams and bumpless pipe dreams. 
	We also introduce co-pipe dreams and co-bumpless pipe dreams, which are used in our change of basis formulas. 
	\cref{section:cotransition} concerns the co-transition recurrences of \cite{Knutson:cotransition}, with particular focus on the K-theoretic co-transition recurrence on pipe dreams.
	
	In \cref{section:pipechangebasis}, we provide new proofs of \cref{thm:pipe_groth_to_schub} and \cref{thm:pipe_schub_to_groth}, the pipe dream change of basis formulas. 
	In particular, our proof technique leads to a new description of the Grothendieck to Schubert expansion in terms of chains in Bruhat order (see \cref{thm:pathformulas}). 
	We show in \cref{section:canonicalbijection} that co-permutations are preserved by the canonical bijection of \cite{Gao.Huang}, and use this to prove \cref{thm:main1} and \cref{thm:main2}.
	Finally, \cref{section:backstable} concerns the expansion of back stable Grothendieck polynomials into back stable Schubert polynomials. 
	We recall necessary background from \cite{Lam.Lee.Shimozono,Lam.Lee.Shimozono-K-theory} and explain how to extend our BPD Grothendieck to Schubert expansion formula to the back stable setting (see \cref{thm:backstable}).

	\section{Background}
	\label{section:background}
	
	Write $\mathbb{Z}_{+} = \{1,2,3,\ldots\}$ and $\mathbb{Z}_{\geq 0} = \{0,1,2,\ldots\}$. Given $m,n \in \mathbb{Z}$, we write $[m] = \{i \in \mathbb{Z} : 1\leq i\leq m\}$ and $[m,n] = \{i \in \mathbb{Z} : m\leq i\leq n\}$. 
	Many of our objects are tilings of an $n \times n$ grid. 
	We use matrix coordinates to refer to the cells in the grid: the cell $(i,j)$ sits in the $i$th row from the top of the grid and the $j$th column from the left.
	
	\subsection{The symmetric group}
	
	In this section, we recall basic facts about the \emph{symmetric group}. 
	We refer the reader to \cite{Manivel} for background.

	We write $S_n$ for the \newword{symmetric group} on $n$ letters, i.e., the set of bijections from $[n]$ to itself, with multiplication defined by composition of functions. 
	We often represent $w \in S_n$ using one-line notation as $w = w_1\,w_2\,\cdots \, w_n$, where $w_i = w(i)$ for all $i \in [n]$. 
	
	The \newword{permutation matrix} of $w \in S_n$ is the $n \times n$ matrix with ones in positions $(i,w(i))$ for all $i \in [n]$, and zeros elsewhere. 
	When convenient, we conflate $w$ with its permutation matrix. 
	We sometimes represent permutation matrices visually by drawing an $n \times n$ grid with dots in cells $(i,w(i))$ for all $i \in [n]$. Call this the \newword{graph} of a permutation.
	We may strike out all cells that lie weakly to the right and weakly below each dot; the set of coordinates of the remaining cells is the \newword{Rothe diagram} of $w$, which we denote by $\rothe(w)$. 
	\begin{example}
		\label{example:rothe}
		Let $w = 52413$.
		
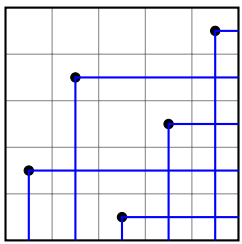
\[
		\begin{tikzpicture}[x = 1.5em,y = 1.5em]
			\draw[step = 1,gray, very thin] (0,0) grid (5,5);
			\draw[thick](0,0)rectangle(5,5);
			\filldraw [black](2.5,.5)circle(.1);
			\filldraw [black](.5,1.5)circle(.1);
			\filldraw [black](3.5,2.5)circle(.1);
			\filldraw [black](1.5,3.5)circle(.1);
			\filldraw [black](4.5,4.5)circle(.1);
			\draw[thick, color = blue] (2.5,0)--(2.5,.5)--(5,.5);
			\draw[thick, color = blue] (.5,0)--(.5,1.5)--(5,1.5);
			\draw[thick, color = blue] (3.5,0)--(3.5,2.5)--(5,2.5);
			\draw[thick, color = blue] (1.5,0)--(1.5,3.5)--(5,3.5);
			\draw[thick, color = blue] (4.5,0)--(4.5,4.5)--(5,4.5); 
		\end{tikzpicture} 
		\] 
		By striking out cells to the right and below each dot in the grid pictured above, we see that \[\rothe(w) = \{(1,1),(1,2),(1,3),(1,4),(2,1),(3,1),(3,3)\}. \qedhere\]
	\end{example}
	
	Given $w \in S_n$, we define $c_w(i) = \#\{j \in[n]:i < j \text{ and } w(i) > w(j)\}$. Equivalently, $c_w(i)$ is the number of cells in row $i$ of the Rothe diagram of $w$. 
	The \newword{Lehmer code} of $w$ is the tuple $c_w = (c_w(1),\ldots,c_w(n))$. 
	The map from permutations to their Lehmer codes is a bijection from $S_n$ to $\{(c_1,c_2,\ldots,c_n):c_i\leq n-i \text{ for all } i \in[n]\}.$
	
	Let $s_i = (i \, i + 1) \in S_n$ denote the simple transposition that exchanges $i$ and $i + 1$. 
	Also, write $t_{i,j} = (i \, j) \in S_n$ for the transposition that exchanges $i$ and $j$.
	Given a tuple of positive integers $\mathbf{a} = (a_1,\ldots,a_k)$, we say that $\mathbf{a}$ is a \newword{word} for $w$ if $s_{a_1}\cdots s_{a_k} = w$, and a \newword{reduced word} for $w$ if $\mathbf{a}$ is a minimum length word for $w$. 
	If $(a_1,\ldots,a_k)$ is a reduced word for $w$, then $k$ is the \newword{Coxeter length} of $w$, and we write $\ell(w) = k$.

	The \newword{Bruhat order} on the symmetric group is the poset that is the transitive closure of the covering relations $u < v$ if there exist $i,j$ such that $ut_{i,j} = v$ and $\ell(u) + 1 = \ell(v)$.
	Given $w \in S_n$ and $i \in[n-1]$, we say that $i$ is a \newword{descent} of $w$ if $ws_i < w$. 
	Write $\des(w) = \{i : ws_i < w\}$. 
	We say that $i \in [n-1]$ is an \newword{ascent} of $w$ if $ws_i > w$.
	
	To each permutation $w \in S_n$, we associate a \newword{rank function} $\rk_w:[0,n] \times [0,n]\rightarrow \mathbb{Z}_{\geq 0}$ defined by \[\rk_w(a,b) = \#\{i: i \in [a] \text{ and } w(i) \in [b]\}.\] The map from permutations to their rank functions is bijective.
	Given $u,v \in S_n$, if there exists a least upper bound of $\{u,v\}$ in $S_n$, we denote it by $u \vee v$ and call it the \newword{join} of $u$ and $v$.
	
	\begin{lemma}
		\label{lemma:rankbruhat}
		Let $u,v,w \in S_n$.
		\begin{enumerate}
			\item $u\leq v$ if and only if $\rk_u(a,b)\geq \rk_v(a,b)$ for all $(a,b) \in [0,n] \times [0,n]$.
			\item If $\rk_w(a,b) = \min(\rk_u(a,b),\rk_v(a,b))$ for all $(a,b) \in [0,n] \times [0,n]$, then $w = u\vee v$.
		\end{enumerate}
	\end{lemma}
	\begin{proof}
		\noindent (1) This is the Ehresmann criterion for Bruhat order; see, e.g., \cite[Theorem 2.1.5]{Bjorner.Brenti}.
		
		\noindent (2) Suppose $\rk_w(a,b) = \min(\rk_u(a,b),\rk_v(a,b))$ for all $(a,b) \in [0,n] \times [0,n]$. By Part (1), this implies $w$ is an upper bound of both $u$ and $v$.
		
		Suppose that $w' \leq w$ is an upper bound of $u$ and $v$. Then, again by Part (1), we have
		$\rk_{w'}(a,b)\leq \rk_u(a,b),\rk_v(a,b)$ which implies $\rk_{w'}(a,b)\leq \min(\rk_u(a,b),\rk_v(a,b))$ for all $(a,b) \in [0,n] \times [0,n]$.
		
		Because
		\[\min(\rk_u(a,b),\rk_v(a,b)) = \rk_w(a,b)\leq \rk_{w'}(a,b)\leq \min(\rk_u(a,b),\rk_v(a,b)),\] we conclude that $\rk_{w'}(a,b) = \rk_w(a,b)$ for all $(a,b)\in [0,n] \times [0,n]$, and therefore $w = w'$.
	\end{proof}
	
	It will sometimes be useful to consider permutations as elements of $S_{+} $, the set of bijections from $\mathbb{Z}_{+}$ to itself such that all but finitely many elements are fixed points. 
	There is a natural embedding of $S_n$ into $S_{+}$ via the map $w\mapsto \tilde{w}$, where
	\[
	\tilde{w}(i) = 
	\begin{cases}
		w(i), & \text{if } i \in [n], \\
		i, & \text{if } i > n.
	\end{cases}
	\]
	Conversely, each $w \in S_{+} $ may be represented by some permutation in $S_N$, for $N$ sufficiently large. 
	In this case, we write $w \in S_N$. 
	The definitions of Bruhat order, reduced words, Coxeter length, Lehmer code, rank function, etc., all extend naturally to $S_{+} $.
	
	\subsection{Hecke actions}
	
	The \newword{$0$-Hecke monoid} $\mathcal H_n$ is the free monoid on generators $\{\tau_i:i \in[n-1]\}$, subject to the relations:
	\begin{enumerate}
		\item $\tau_i^2 = \tau_i$,
		\item $\tau_i \tau_j = \tau_j \tau_i$ if $|i - j| > 1$, and
		\item $\tau_i \tau_{i + 1} \tau_i = \tau_{i + 1} \tau_{i} \tau_{i + 1}$.
	\end{enumerate}
	There is an action $*$ of $\mathcal H_n$ on $S_n$ defined by
	\[w* \tau_k = 
	\begin{cases}
		w &\text{if } w s_k < w,\\
		w s_k & \text{if } w s_k > w.
	\end{cases}\]
	The \newword{Demazure product} of a word is defined recursively as follows: $\demprod(\emptyset) = \id$, and $\demprod((a_1, \ldots, a_k)) = \demprod((a_1, \ldots, a_{k-1}))*\tau_{a_k}$. 
	Note that if $\mathbf{a}=(a_1, \ldots, a_k)$ is a reduced word, then $\demprod(\mathbf{a})=s_{a_1}s_{a_2}\cdots s_{a_k}$.
		
	Also, given $w \in S_n$ and $k \in [n - 1]$, we define 
	\[w \square \tau_k = 
	\begin{cases}
		w &\text{if } w s_k > w, \\
		w s_k & \text{if } w s_k < w.
	\end{cases}\]
	The operator $\square$ is closely related to $*$. 
	Write $w_0$ for the \newword{longest permutation} in $S_n$, that is, $w_0= n \, n-1 \, \ldots \, 1$.
	\begin{lemma}
		\label{lemma:upanddownoperatorswithflips}
		Let $w \in S_n$ and $i \in [n - 1]$. 
		Then the following hold:
		\begin{enumerate}
			\item $(w * \tau_i) w_0 = (w w_0)\square \tau_{n - i}$.
			\item $(w\square \tau_i) w_0 = (w w_0)* \tau_{n - i}$.
			\item $w_0 (w*\tau_i)= (w_0w) \square \tau_i$.
			\item $w_0 (w\square \tau_i)=(w_0 w)*\tau_i$.
		\end{enumerate}
	\end{lemma}
	\begin{proof}
		In one-line notation, $ww_0$ is $w_{n}w_{n-1}\cdots w_1$, i.e., right multiplication by $w_0$ reverses the string. 
		In particular, we have $w_i=(w w_0)_{n - i + 1}$ and $w_{i + 1} = (w w_0)_{n - i}$.
		Thus, $w_i < w_{i + 1}$ if and only if $(w w_0)_{n-i} > (w w_0)_{n-i + 1}$. 
		Thus, parts (1) and (2) follow directly from this observation and the definitions of $*$ and $\square$.
		
		The second two statements follow from observing that $w$ has a descent at $i$ if and only if $w_0w$ has an ascent at $i$.
	\end{proof}
	
	We record another lemma for later use.
	\begin{lemma}
		\label{lemma:conjugationandwords}
		Fix a word $\mathbf{a}=(a_1, a_2,\ldots,a_k)$ with $a_i\in[n-1]$ for all $i$. 
		Let $\mathbf{b}=(n-a_1,n-a_{2},\ldots,n-a_k)$. Then $\demprod(\mathbf{a})=w_0\demprod(\mathbf{b}) w_0$, where $w_0\in S_n$ is the longest permutation.
	\end{lemma}
\begin{proof}
	We proceed by induction. The statement is trivially true for the empty word.
	Now let $\mathbf{a'}=(a_1, a_2,\ldots,a_{k-1})$ and $\mathbf{b'}=(n-a_1,n-a_{2},\ldots,n-a_{k-1})$. 
	Assume that $\demprod(\mathbf{a'})=w_0\demprod(\mathbf{b'})w_0$. 
	Then by applying \cref{lemma:upanddownoperatorswithflips}, we have
	\begin{align*}
		\demprod(\mathbf{a})&=\demprod(\mathbf{a'})*\tau_{a_{k}}\\
		&=(w_0\demprod(\mathbf{b'})w_0)*\tau_{a_k}\\
		&=((w_0\demprod(\mathbf{b'}))\square \tau_{n-a_k})w_0\\
		&=w_0 (\demprod(\mathbf{b'})*\tau_{n-a_k})w_0\\
		&=w_0\demprod(\mathbf{b}) w_0,
	\end{align*}
	as desired.
\end{proof}

	\subsection{Schubert polynomials and Grothendieck polynomials}
	
	The symmetric group $S_n$ acts on the ring $\mathbb{Z}[\beta][x_1, \ldots, x_n]$ by permuting the indices of the $x_i$'s. Specifically, given $w \in S_n$ and $f \in \mathbb{Z}[\beta][x_1, \ldots, x_n]$, we define 
	\[w\cdot f = f(x_{w_1}, x_{w_2}, \ldots, x_{w_n}).
	\]
	Given $i \in [n - 1]$, we define the \newword{Newton divided difference operator} $\divdiff_i$ by $\divdiff_i(f) = \frac{f - s_i \cdot f}{x_i - x_{i + 1}}$. 
	We also define the \newword{K-theoretic divided difference operator} $\kdivdiff^{(\beta)}_i$ by 
	\[\kdivdiff^{(\beta)}_i(f) = \divdiff_i \left((1 + \beta x_{i + 1})f \right).\]
	Note that $\divdiff_i = \kdivdiff^{(0)}_i$.

	We use these operators to recursively define two families of polynomials: \newword{Schubert polynomials} $\{\mathfrak{S}_w : w \in S_n\}$ and \newword{$\beta$-Grothendieck polynomials} $\{\mathfrak{G}^{(\beta)}_w : w \in S_n\}$.
	Both families start with the same initial condition:
	\[\mathfrak{S}_{w_0} = \mathfrak{G}^{(\beta)}_{w_0} = x_1^{n-1 }x_2^{n-2} \cdots x_{n-1}.\]
	If $w_i < w_{i + 1}$, then we define
	\[
	\mathfrak{S}_{w s_i} = \divdiff_i(\mathfrak{S}_w)
	\quad \text{and} \quad
	\mathfrak{G}^{(\beta)}_{w s_i} = \kdivdiff^{(\beta)}_i(\mathfrak{G}^{(\beta)}_w).
	\]
	Note that $\mathfrak{S}_w = \mathfrak{G}^{(0)}_{w}$.

	The divided difference and K-theoretic divided difference operators satisfy the same braid and commutation relations as the simple reflections. 
	Hence, Schubert and $\beta$-Grothendieck polynomials are well-defined. 
	Furthermore, $\left(\kdivdiff_i^{(\beta)}\right)^2 = -\beta \kdivdiff_i^{(\beta)}$ and $\divdiff_i^2 = 0$.
	
	As such,
	\begin{equation}\label{eq:shubdivdiff}
		\divdiff_i(\mathfrak{S}_w) = 
		\begin{cases}
			\mathfrak{S}_{ws_i} &\text{if } w > ws_i, \\
			0 &\text {if } w < ws_i
		\end{cases}
	\end{equation}
	and
	\begin{equation}\label{eq:grothdivdiffbeta}
		\kdivdiff^{(\beta)}_i(\mathfrak{G}^{(\beta)}_w) = 
		\begin{cases}
			\mathfrak{G}^{(\beta)}_{ws_i} &\text{if } w > ws_i, \\
			-\beta \mathfrak{G}^{(\beta)}_{w} &\text {if } w < ws_i.
		\end{cases}
	\end{equation}

	There is a natural inclusion $\iota:S_n \hookrightarrow S_{n + 1}$ defined by 
	\[
	\iota(w)(k) = 
	\begin{cases}
	w(k), & \text{if } 1 \leq k \leq n, \\
	n + 1, & \text{if } k = n + 1.
	\end{cases}
	\]
	Schubert and $\beta$-Grothendieck polynomials are stable with respect to this inclusion, that is,
	\[
	\mathfrak{S}_{\iota(w)} = \mathfrak{S}_w
	\quad \text{and} \quad
	\mathfrak{G}^{(\beta)}_{\iota(w)} = \mathfrak{G}^{(\beta)}_w.
	\]
	Because of this stability, we may regard these families of polynomials as being indexed by permutations $w \in S_{+}$. 
	
	We will often be interested in the specialization $\beta = -1$. 
	For convenience, we define $\kdivdiff_i = \kdivdiff^{(-1)}_i$ and $\mathfrak{G}_w = \mathfrak{G}^{(-1)}_w$. 
	We call the polynomials $\mathfrak{G}_w$ \newword{Grothendieck polynomials}, omitting the $\beta$.
	In this case, we may write \cref{eq:grothdivdiffbeta} more compactly as
	\begin{equation}\label{eq:grothdivdiff}
		\kdivdiff_i(\mathfrak{G}_w) = \mathfrak{G}_{w \square \tau_i}.
	\end{equation}
	Note that $\mathfrak G_w(-\beta x_1,-\beta x_2,\ldots, -\beta x_n)=(-\beta)^{\ell(w)}\mathfrak{G}_w^{(\beta)}$. Thus, there is no generality lost by working with Grothendieck polynomials as opposed to $\beta$-Grothendieck polynomials.
	
	The sets of Schubert polynomials and Grothendieck polynomials each form a $\mathbb Z$-linear basis for $\mathbb Z[x_1,x_2,\ldots]$. 
	As such, any Schubert polynomial can be expressed as a \(\mathbb{Z}\)-linear combination of Grothendieck polynomials, and vice versa.

	\section{Planar histories}
	\label{section:planar}
	
	The main combinatorial objects in this paper are pipe dreams and bumpless pipe dreams. Both are special cases of \emph{planar histories}. 
	For historical reasons, we draw pipe dreams so that pipes start at the top of the $n \times n$ grid, and exit at the left edge. 
	Pipes in bumpless pipe dreams, on the other hand, start at the bottom of the grid and exit at the right. 
	We also introduce co-pipe dreams, in which pipes start at the bottom and exit on the left, and co-bumpless pipe dreams, in which pipes start at the top and exit on the right.
	Despite the variety of orientations among these diagrams, they share certain fundamental similarities as planar histories.
	
	We state our main definitions and facts in terms of \emph{northeast planar histories}. We will also discuss \emph{southeast}, \emph{northwest}, and \emph{southwest} planar histories, by modifying the definitions and results for NE planar histories as needed.
	
	\subsection{Preliminaries on planar histories}
	
	A \newword{pseudoline} is a piecewise linear curve in the plane.
	A \newword{northeast (NE) planar history} $\mathcal{P}$ is a collection of pseudolines in a rectangular region such that:
	\begin{enumerate}
		\item each pseudoline starts at the bottom edge of the rectangle,
		\item each pseudoline ends at the right edge of the rectangle,
		\item each pseudoline travels only north or east at all times, and
		\item distinct pseudolines do not travel concurrently.
	\end{enumerate}
	We may modify the conditions above as appropriate to define southeast (SE), northwest (NW), and southwest (SW) planar histories. 
	
	To align with the terminology for our primary objects of study (BPDs and pipe dreams), we refer to pseudolines as \newword{pipes} from this point on.
	We say that a planar history is \newword{reduced} if each pair of pipes crosses at most one time. 
	We draw sharp turns in pipes as rounded bends. 
	This allows us to visually distinguish between the case when two pipes meet at a point but do not cross from the case where they do cross. 
	When two pipes meet but do not cross, this is sometimes known as an \emph{osculating point} in the literature. See \cite{Behrend}.
	
	\begin{example}
		In the picture below, one pipe (drawn as a solid blue line) passes horizontally across another vertical pipe (drawn as a dotted black line).
		
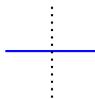
\[
		\begin{tikzpicture}[x = 1.5em,y = 1.5em]
			\draw[thick,rounded corners, color = blue] (-1,2)--(1,2);
			\draw[thick,rounded corners, color = black, dotted] (0,1)--(0,3);
		\end{tikzpicture}
		\]
		We could also have the following situation, where the blue pipe turns north and the black pipe turns east.
		
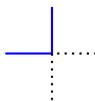
\[\begin{tikzpicture}[x = 1.5em,y = 1.5em]
			\draw[thick, color = blue] (-1,2)--(0,2)--(0,3);
			\draw[thick, color = black, dotted] (0,1)--(0,2)--(1,2);
		\end{tikzpicture}\]
		By drawing the picture with rounded corners, we no longer need additional visual cues to distinguish between the two situations.
		
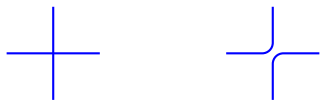
\[
		\begin{tikzpicture}[x = 1.5em,y = 1.5em]
			\draw[thick,rounded corners, color = blue] (-1,2)--(1,2);
			\draw[thick,rounded corners, color = blue] (0,1)--(0,3);
		\end{tikzpicture}
		\hspace{4em}
		\begin{tikzpicture}[x = 1.5em,y = 1.5em]
			\draw[thick,rounded corners, color = blue] (-1,2)--(0,2)--(0,3);
			\draw[thick,rounded corners, color = blue] (0,1)--(0,2)--(1,2);
		\end{tikzpicture}
		\]
		Explicitly, the pipes on the left cross while the pipes on the right do not. 
	\end{example}
	
	Given a NE planar history $\mathcal{P}$ consisting of $n$ pipes, the \newword{permutation of $\mathcal{P}$}, denoted $\demprod(\mathcal{P})$, is defined as follows:
	\begin{enumerate}
		\item Label pipes along the bottom edge of the rectangle from left to right with the numbers $1, 2, \ldots, n$.
		\item Extend these labels across crossings by the rules illustrated below. For each pair of pipes labeled with $i$ and $j$ so that $i < j$, propagate the labels based on the following two cases:
		\begin{equation}
			\label{eqn:NEpipecross}
			\raisebox{-2.75em}{\begin{tikzpicture}[x = 1.5em,y = 1.5em]
					\draw[thick,rounded corners, color = blue,- > ] (-1,2)--(1,2);
					\draw[thick,rounded corners, color = blue,- > ] (0,1)--(0,3);
					\draw (0,.5) node [align = left]{$j$};
					\draw (1.5,2) node [align = left]{$i$};
					\draw (-1.5,2) node [align = left]{$i$};
					\draw (0,3.5) node [align = left]{$j$};
				\end{tikzpicture}
				\hspace{4em}
				\begin{tikzpicture}[x = 1.5em,y = 1.5em]
					\draw[thick,rounded corners, color = blue,- > ] (-1,2)--(1,2);
					\draw[thick,rounded corners, color = blue,- > ] (0,1)--(0,3);
					\draw (0,.5) node [align = left]{$i$};
					\draw (1.5,2) node [align = left]{$i$};
					\draw (-1.5,2) node [align = left]{$j$};
					\draw (0,3.5) node [align = left]{$j$};
			\end{tikzpicture}}
		\end{equation}
		\item Finally, read the resulting labels from top to bottom along the right edge of the grid. 
		This gives a permutation written in one-line notation, which we denote by $\demprod(\mathcal{P})$.
	\end{enumerate}
	
	Informally, we obtain $\demprod(\mathcal{P})$ by tracing along the pipes, ignoring a crossing if the same two pipes have already crossed.

	The following lemma explains how to compute $\demprod(\mathcal{P})$ as the Demazure product of an associated word.
	
	\begin{lemma}
		\label{lemma:planarhistoryword}
		Let $\mathcal{P}$ be a NE planar history with $h$ crossings. 
		Label the crossings in $\mathcal{P}$ with the integers $1,2,\ldots, h$ so that as we read along each pipe, the sequence of labels is strictly increasing. 
		With respect to this labeling, define a word $\mathbf{a} = (a_1, a_2, \ldots, a_h)$, where $a_i$ is the number of pipes that pass above the point immediately to the right and above the $i$th crossing. 
		Then $\demprod(\mathcal{P}) = \demprod(\mathbf{a})$.
	\end{lemma}
	\begin{proof}
		We proceed by induction on the number of crossings $h$. 
		In the base case, $\mathcal{P}$ is a planar history with no crossings. 
		Since there are no crossings, $\demprod(\mathcal{P})$ is the identity, and $\mathbf{a}$ is the empty word. 
		The Demazure product of the empty word is also the identity permutation. 
		Thus, $\demprod(\mathcal{P}) = \demprod(\mathbf{a})$.
		
		Suppose now that the result holds for all planar histories with $h-1$ crossings, and let $\mathcal{P}$ be a planar history with $h$ crossings. 
		Form a new planar history $\mathcal{P}'$ with $h-1$ crossings by taking crossing $h$ and replacing it with a bump, as pictured below. 
		
		
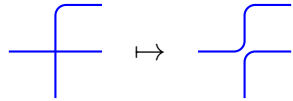
\[\begin{tikzpicture}[x = 1.5em,y = 1.5em]
			\draw[thick,rounded corners, color = blue] (-1,2)--(1,2);
			\draw[thick,rounded corners, color = blue] (0,1)--(0,3)--(1,3);
		\end{tikzpicture} 
		\hspace{1em} \raisebox{1.25em}{$\mapsto$} \hspace{1em}
		\begin{tikzpicture}[x = 1.5em,y = 1.5em]
			\draw[thick,rounded corners, color = blue] (-1,2)--(0,2)--(0,3)--(1,3);
			\draw[thick,rounded corners, color = blue] (0,1)--(0,2)--(1,2);
		\end{tikzpicture}
		\]

		By the inductive hypothesis, $\demprod(\mathcal{P}') = \demprod((a_1,a_2,\ldots,a_{h-1}))$. 
		It is enough to show that $\demprod(\mathcal{P}) = \demprod(\mathcal{P}') * \tau_{a_{h}}$.
		Suppose crossing $h$ involves pipes labeled $i$ and $j$ with $i < j$. 
		Then there are two possible labelings at this crossing, as pictured below.
		
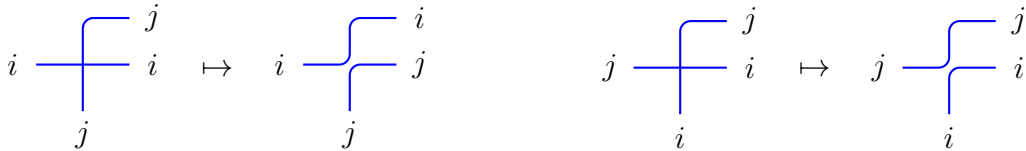
\[\begin{tikzpicture}[x = 1.5em,y = 1.5em]
			\draw[thick,rounded corners, color = blue] (-1,2)--(1,2);
			\draw[thick,rounded corners, color = blue] (0,1)--(0,3)--(1,3);
			\draw (0,.5) node [align = left]{$j$};
			\draw (1.5,2) node [align = left]{$i$};
			\draw (-1.5,2) node [align = left]{$i$};
			\draw (1.5,3) node [align = left]{$j$};
		\end{tikzpicture} 
		\hspace{1em} \raisebox{2.65em}{$\mapsto$} \hspace{1em}
		\begin{tikzpicture}[x = 1.5em,y = 1.5em]
			\draw[thick,rounded corners, color = blue] (-1,2)--(0,2)--(0,3)--(1,3);
			\draw[thick,rounded corners, color = blue] (0,1)--(0,2)--(1,2);
			\draw (0,.5) node [align = left]{$j$};
			\draw (1.5,2) node [align = left]{$j$};
			\draw (-1.5,2) node [align = left]{$i$};
			\draw (1.5,3) node [align = left]{$i$};
		\end{tikzpicture}
		\hspace{5em}
		\begin{tikzpicture}[x = 1.5em,y = 1.5em]
			\draw[thick,rounded corners, color = blue] (-1,2)--(1,2);
			\draw[thick,rounded corners, color = blue] (0,1)--(0,3)--(1,3);
			\draw (0,.5) node [align = left]{$i$};
			\draw (1.5,2) node [align = left]{$i$};
			\draw (-1.5,2) node [align = left]{$j$};
			\draw (1.5,3) node [align = left]{$j$};
		\end{tikzpicture} 
		\hspace{1em} \raisebox{2.65em}{$\mapsto$} \hspace{1em}
		\begin{tikzpicture}[x = 1.5em,y = 1.5em]
			\draw[thick,rounded corners, color = blue] (-1,2)--(0,2)--(0,3)--(1,3);
			\draw[thick,rounded corners, color = blue] (0,1)--(0,2)--(1,2);
			\draw (0,.5) node [align = left]{$i$};
			\draw (1.5,2) node [align = left]{$i$};
			\draw (-1.5,2) node [align = left]{$j$};
			\draw (1.5,3) node [align = left]{$j$};
		\end{tikzpicture}\]
		In both replacements, the diagram on the left belongs to $\mathcal{P}$ and the diagram on the right is part of $\mathcal{P}'$.
	
		Because we are considering the last crossing in $\mathcal{P}$, the two pipes continue to the right side of the rectangle and do not cross any other pipes after this point. By definition of $a_h$, there are exactly $a_h$ pipes that pass above the point just to the right and above crossing $h$ in $\mathcal{P}$. In particular, when the pipe labeled $j$ exits the rectangle on the right, it is the $a_h$th pipe from the top.
		Therefore, $\demprod(\mathcal{P})(a_h) = j$ and $\demprod(\mathcal{P})(a_h + 1) = i$.
		
		In the configuration pictured on the left, $\demprod(\mathcal{P}')$ has an ascent at $a_h$, so $\demprod(\mathcal{P}')*\tau_{a_h} = \demprod(\mathcal{P}')s_{a_h} = \demprod(\mathcal{P})$. 
		In the configuration on the right, $\demprod(\mathcal{P}')$ has a descent at $a_h$, and hence $\demprod(\mathcal{P}')*\tau_{a_h} = \demprod(\mathcal{P}') = \demprod(\mathcal{P})$.
	\end{proof}
	
	\begin{example}
		Pictured below is a NE planar history $\mathcal{P}$.
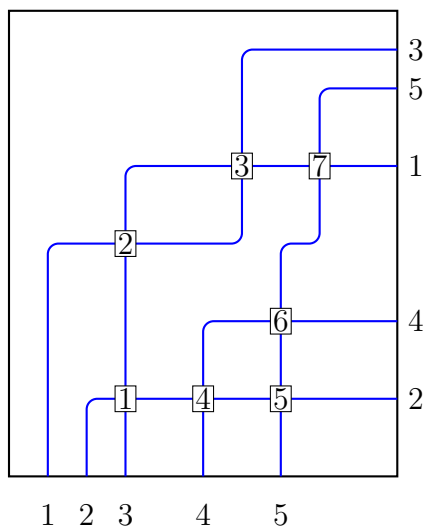
		\[
		\begin{tikzpicture}[x = 2.5em,y = 2.5em]
			\draw[color = black, thick] (0,-1) rectangle (5, 5);
			\draw[thick,rounded corners,color = blue] (1/2,-1)--(1/2, 2)--(3,2)--(3,3)--(3,4.5)--(5,4.5);
			\draw[thick,rounded corners,color = blue] (3/2,-1)--(3/2, 3)--(3.5,3)--(5,3);
			\draw[thick,rounded corners,color = blue] (5/2,-1)--(5/2,1)--(5,1);
			\draw[thick,rounded corners,color = blue] (7/2,-1)--(7/2,1.5)--(3.5,2)--(4,2)--(4,4)--(5,4);
			\draw[thick,rounded corners,color = blue] (1,-1)--(1,0)--(5,0);
			\node[draw = black,fill = white,inner sep = 1pt] at (1.5, 0) {$1$}; 
			\node[draw = black,fill = white,inner sep = 1pt] at (1.5, 2) {$2$};
			\node[draw = black,fill = white,inner sep = 1pt] at (3, 3) {$3$};
			\node[draw = black,fill = white,inner sep = 1pt] at (2.5, 0) {$4$}; 
			\node[draw = black,fill = white,inner sep = 1pt] at (3.5, 0) {$5$}; 
			\node[draw = black,fill = white,inner sep = 1pt] at (3.5, 1) {$6$}; 
			\node[draw = black,fill = white,inner sep = 1pt] at (4, 3) {$7$}; 
			\node at (1/2, -1.5) {$1$};
			\node at (1, -1.5) {$2$};
			\node at (3/2, -1.5) {$3$};
			\node at (5/2, -1.5) {$4$};
			\node at (7/2, -1.5) {$5$};
			\node at (5, 4.5) [right] {$3$};
			\node at (5, 3) [right] {$1$};
			\node at (5, 1) [right] {$4$};
			\node at (5, 4) [right] {$5$};
			\node at (5, 0) [right] {$2$};
		\end{tikzpicture}
		\]
		
		Reading along pipes, we see that $\demprod(\mathcal{P}) = 35142$. 
		We can also compute this using \cref{lemma:planarhistoryword}. 
		In the diagram above, the crossings have been labeled so that for each pipe, the sequence of crossing labels encountered from bottom to top is strictly increasing. From this labeling, we obtain the word $\mathbf{a} = (2,1,1,3,4,3,2)$. The reader may verify that $\demprod(\mathbf{a}) = \demprod(\mathcal{P})$.
	\end{example}

	We now describe how to associate permutations to the other types of planar histories. We have chosen these conventions to be compatible with our change of basis theorems, as well as the usual way of reading permutations from pipe dreams.
	
	For NW planar histories, label pipes along the bottom edge of the rectangle from left to right with the numbers $1, 2,\ldots, n$. 
	For SW and SE planar histories, label pipes along the top edge of the rectangle from left to right with the numbers $1, 2,\ldots, n$. 
	As before, we propagate labels along pipes using local rules at each crossing.
	Assuming $i < j$, the local rules for the SW, NW, and SE planar histories respectively, are pictured below. 
	\begin{equation}
		\label{eqn:SWpipecross}
		\raisebox{-2.75em}{\begin{tikzpicture}[x = 1.5em,y = 1.5em]
				\draw[thick,rounded corners, color = blue,- > ] (1,2)--(-1,2);
				\draw[thick,rounded corners, color = blue,- > ] (0,3)--(0,1);
				\draw (0,3.5) node [align = left]{$i$};
				\draw (-1.5,2) node [align = left]{$j$};
				\draw (1.5,2) node [align = left]{$j$};
				\draw (0,.5) node [align = left]{$i$};
			\end{tikzpicture}
			\hspace{4em}
			\begin{tikzpicture}[x = 1.5em,y = 1.5em]
				\draw[thick,rounded corners, color = blue,- > ] (1,2)--(-1,2);
				\draw[thick,rounded corners, color = blue,- > ] (0,3)--(0,1);
				\draw (0,3.5) node [align = left]{$j$};
				\draw (-1.5,2) node [align = left]{$j$};
				\draw (1.5,2) node [align = left]{$i$};
				\draw (0,.5) node [align = left]{$i$};
		\end{tikzpicture}}
	\end{equation}
	
	\begin{equation}
		\label{eqn:NWpipecross}
		\raisebox{-2.75em}{
			\begin{tikzpicture}[x = 1.5em,y = 1.5em]
				\draw[thick,rounded corners, color = blue,- > ] (1,2)--(-1,2);
				\draw[thick,rounded corners, color = blue,- > ] (0,1)--(0,3);
				\draw (0,.5) node [align = left]{$i$};
				\draw (-1.5,2) node [align = left]{$j$};
				\draw (1.5,2) node [align = left]{$j$};
				\draw (0,3.5) node [align = left]{$i$};
			\end{tikzpicture}
			\hspace{4em}
			\begin{tikzpicture}[x = 1.5em,y = 1.5em]
				\draw[thick,rounded corners, color = blue,- > ] (1,2)--(-1,2);
				\draw[thick,rounded corners, color = blue,- > ] (0,1)--(0,3);
				\draw (0,.5) node [align = left]{$j$};
				\draw (-1.5,2) node [align = left]{$j$};
				\draw (1.5,2) node [align = left]{$i$};
				\draw (0,3.5) node [align = left]{$i$};
		\end{tikzpicture}}
	\end{equation}
	
	\begin{equation}
		\label{eqn:SEpipecross}
		\raisebox{-2.75em}{\begin{tikzpicture}[x = 1.5em,y = 1.5em]
				\draw[thick,rounded corners, color = blue,- > ] (-1,2)--(1,2);
				\draw[thick,rounded corners, color = blue,- > ] (0,3)--(0,1);
				\draw (0,.5) node [align = left]{$j$};
				\draw (-1.5,2) node [align = left]{$i$};
				\draw (1.5,2) node [align = left]{$i$};
				\draw (0,3.5) node [align = left]{$j$};
			\end{tikzpicture}
			\hspace{4em}
			\begin{tikzpicture}[x = 1.5em,y = 1.5em]
				\draw[thick,rounded corners, color = blue,- > ] (-1,2)--(1,2);
				\draw[thick,rounded corners, color = blue,- > ] (0,3)--(0,1);
				\draw (0,.5) node [align = left]{$j$};
				\draw (-1.5,2) node [align = left]{$j$};
				\draw (1.5,2) node [align = left]{$i$};
				\draw (0,3.5) node [align = left]{$i$};
		\end{tikzpicture}}
	\end{equation}
	
	The pipes in a planar history will end on the left or right side of the rectangle. 
	In either case, we define the permutation $\demprod(\mathcal{P})$ by reading the labels from top to bottom.  As before, we obtain $\demprod(\mathcal{P})$ by tracing along the pipes, ignoring a crossing if the same two pipes have already crossed.
	
	A planar history can be reoriented into a NE planar history by reflecting the diagram horizontally, vertically, or both horizontally and vertically.
	Starting from $\mathcal{P}$, write $\NE(\mathcal{P})$ for the NE planar history obtained in this way. 
	Because of the conventions we have chosen, we have the following relationships between $\demprod(\mathcal{P})$ and $\demprod(\NE(\mathcal{P}))$.
	\begin{lemma} 
		\label{lemma:flipdiagramschangepermutation}
		Let $\mathcal{P}$ be a planar history with $n$ pipes, and let $w_0 \in S_n$ be the longest permutation.
		\begin{enumerate}
			\item If $\mathcal{P}$ is a NW planar history, then $\demprod(\mathcal{P}) = w_0\demprod(\NE(\mathcal{P}))$.
			\item If $\mathcal{P}$ is a SE planar history, then $\demprod(\mathcal{P}) = \demprod(\NE(\mathcal{P}))w_0$.
			\item If $\mathcal{P}$ is a SW planar history, then $\demprod(\mathcal{P}) = w_0\demprod(\NE(\mathcal{P}))w_0$.
		\end{enumerate}
	\end{lemma}
	\begin{proof}
		The statement follows directly from the rules pictured in \cref{eqn:NEpipecross}, \cref{eqn:SWpipecross}, \cref{eqn:NWpipecross}, and \cref{eqn:SEpipecross}, as well as the definition of $\demprod(\mathcal{P})$.
	\end{proof}

	We record the following fact about inversion pairs in $\demprod(\mathcal{P})$.
	\begin{lemma}
		\label{lemma:planarhistoryinversions}
		Let $\mathcal{P}$ be a planar history with $n$ pipes, and fix indices $a<b$.
		\begin{enumerate}
			\item Suppose $\mathcal{P}$ is a NE or SW planar history.
			Then $\demprod(\mathcal{P})(b) > \demprod(\mathcal{P})(a)$ if and only if the pipes exiting in rows $a$ and $b$ cross in $\mathcal{P}$.
			\item Suppose $\mathcal{P}$ is a NW or SE planar history.
			Then $\demprod(\mathcal{P})(b)>\demprod(\mathcal{P})(a)$ if and only if the pipes exiting in rows $a$ and $b$ {\bf do not} cross in $\mathcal{P}$.
		\end{enumerate}
	\end{lemma}
	\begin{proof}
		The statement follows directly from the rules pictured in \cref{eqn:NEpipecross}, \cref{eqn:SWpipecross}, \cref{eqn:NWpipecross}, and \cref{eqn:SEpipecross}, as well as the definition of $\demprod(\mathcal{P})$.
	\end{proof}
	
	\subsection{Local moves on planar histories}
	
	Many of our arguments will involve analyzing local modifications of pipes. 
	Below, we list several moves that preserve permutations of planar histories.
	\begin{lemma}
		\label{lemma:planarmoves}
		\begin{enumerate}
			\item Suppose $\mathcal{P}$ is a NE planar history or a SW planar history. Applying any of the local moves pictured below to $\mathcal{P}$ produces a new planar history $\mathcal{P}'$ such that $\demprod(\mathcal{P}) = \demprod(\mathcal{P}')$.
			\begin{equation}
				\raisebox{-1.25em}{\begin{tikzpicture}[x = .75em,y = .75em]
						\draw[thick] (2,0)--(2,4);
						\draw[thick,rounded corners, color = blue] (1,1)--(1,3)--(3,3);
					\end{tikzpicture}
					\hspace{1em} \raisebox{1.3em}{$\leftrightarrow$} \hspace{1em}
					\begin{tikzpicture}[x = .75em,y = .75em]
						\draw[thick] (2,0)--(2,4);
						\draw[thick,rounded corners, color = blue] (1,1)--(3,1)--(3,3);
					\end{tikzpicture}
					\hspace{2em}
					\raisebox{.7em}{
						\begin{tikzpicture}[x = .75em,y = .75em]
							\draw[thick] (0,2)--(4,2);
							\draw[thick,rounded corners, color = blue] (1,1)--(1,3)--(3,3);
						\end{tikzpicture}
						\hspace{1em} \raisebox{.6em}{$\leftrightarrow$} \hspace{1em}
						\begin{tikzpicture}[x = .75em,y = .75em]
							\draw[thick] (0,2)--(4,2);
							\draw[thick,rounded corners, color = blue] (1,1)--(3,1)--(3,3);
					\end{tikzpicture}}
					\hspace{2em}
					\begin{tikzpicture}[x = .75em,y = .75em]
						\draw[thick] (2,0)--(2,4);
						\draw[thick] (0,2)--(4,2);
						\draw[thick,rounded corners, color = blue] (1,1)--(1,3)--(3,3);
					\end{tikzpicture}
					\hspace{1em} \raisebox{1.3em}{$\leftrightarrow$} \hspace{1em}
					\begin{tikzpicture}[x = .75em,y = .75em]
						\draw[thick] (2,0)--(2,4);
						\draw[thick] (0,2)--(4,2);
						\draw[thick,rounded corners, color = blue] (1,1)--(3,1)--(3,3);
				\end{tikzpicture}}
			\end{equation}
			\begin{equation}
				\raisebox{-1.25em}{\begin{tikzpicture}[x = .75em,y = .75em]
						\draw[thick,rounded corners, color = blue] (0,1)--(1,1)--(1,3)--(4,3);
						\draw[thick,rounded corners, color = blue] (1,0)--(1,1)--(3,1)--(3,4);
					\end{tikzpicture}
					\hspace{1em}
					\raisebox{1.3em}{$\leftrightarrow$} \hspace{1em}
					\begin{tikzpicture}[x = .75em,y = .75em]
						\draw[thick,rounded corners, color = blue] (0,1)--(3,1)--(3,4);
						\draw[thick,rounded corners, color = blue] (1,0)--(1,3)--(4,3);
					\end{tikzpicture}
					\hspace{1em} \raisebox{1.3em}{$\leftrightarrow$} \hspace{1em}
					\begin{tikzpicture}[x = .75em,y = .75em]
						\draw[thick,rounded corners, color = blue] (1,0)--(1,3)--(3,3)--(3,4);
						\draw[thick,rounded corners, color = blue] (0,1)--(3,1)--(3,3)--(4,3);
				\end{tikzpicture}}
			\end{equation}
						
			\item Suppose $\mathcal{P}$ is a NW planar history or a SE planar history. Applying any of the local moves pictured below to $\mathcal{P}$ produces a new planar history $\mathcal{P}'$ such that $\demprod(\mathcal{P}) = \demprod(\mathcal{P}')$.
			
			\begin{equation}
				\raisebox{-1.25em}{\begin{tikzpicture}[x = .75em,y = .75em]
						\draw[thick] (2,0)--(2,4);
						\draw[thick,rounded corners, color = blue] (3,1)--(3,3)--(1,3);
					\end{tikzpicture}
					\hspace{1em} \raisebox{1.3em}{$\leftrightarrow$} \hspace{1em}
					\begin{tikzpicture}[x = .75em,y = .75em]
						\draw[thick] (2,0)--(2,4);
						\draw[thick,rounded corners, color = blue] (3,1)--(1,1)--(1,3);
					\end{tikzpicture}
					\hspace{2em}
					\raisebox{.7em}{
						\begin{tikzpicture}[x = .75em,y = .75em]
							\draw[thick] (0,2)--(4,2);
							\draw[thick,rounded corners, color = blue] (3,1)--(3,3)--(1,3);
						\end{tikzpicture}
						\hspace{1em} \raisebox{.6em}{$\leftrightarrow$} \hspace{1em}
						\begin{tikzpicture}[x = .75em,y = .75em]
							\draw[thick] (0,2)--(4,2);
							\draw[thick,rounded corners, color = blue] (3,1)--(1,1)--(1,3);
					\end{tikzpicture}}
					\hspace{2em}
					\begin{tikzpicture}[x = .75em,y = .75em]
						\draw[thick] (2,0)--(2,4);
						\draw[thick] (0,2)--(4,2);
						\draw[thick,rounded corners, color = blue] (3,1)--(3,3)--(1,3);
					\end{tikzpicture}
					\hspace{1em} \raisebox{1.3em}{$\leftrightarrow$} \hspace{1em}
					\begin{tikzpicture}[x = .75em,y = .75em]
						\draw[thick] (2,0)--(2,4);
						\draw[thick] (0,2)--(4,2);
						\draw[thick,rounded corners, color = blue] (3,1)--(1,1)--(1,3);
				\end{tikzpicture}}
			\end{equation}

			\begin{equation}
				\raisebox{-1.25em}{\begin{tikzpicture}[x = .75em,y = .75em]
						\draw[thick,rounded corners, color = blue] (4,1)--(3,1)--(3,3)--(0,3);
						\draw[thick,rounded corners, color = blue] (3,0)--(3,1)--(1,1)--(1,4);
					\end{tikzpicture}
					\hspace{1em}
					\raisebox{1.3em}{$\leftrightarrow$} \hspace{1em}
					\begin{tikzpicture}[x = .75em,y = .75em]
						\draw[thick,rounded corners, color = blue] (4,1)--(1,1)--(1,4);
						\draw[thick,rounded corners, color = blue] (3,0)--(3,3)--(0,3);
					\end{tikzpicture}
					\hspace{1em} \raisebox{1.3em}{$\leftrightarrow$} \hspace{1em}
					\begin{tikzpicture}[x = .75em,y = .75em]
						\draw[thick,rounded corners, color = blue] (3,0)--(3,3)--(1,3)--(1,4);
						\draw[thick,rounded corners, color = blue] (4,1)--(1,1)--(1,3)--(0,3);
				\end{tikzpicture}}
			\end{equation}
		\end{enumerate}
	\end{lemma}
	
	\begin{proof}
		This is standard; see, e.g., \cite[Lemma~2.1]{Weigandt}.
	\end{proof}
	
	Note also that we may rescale and shift pipes locally in various ways, so long as the relative order of the crossings is preserved; in such cases, the resulting permutation remains the same.
	
	In the next lemma, we describe another move on planar histories that preserves permutations.
	
	\begin{lemma}
		\label{lemma:planarhistoryverticalisok}
		Suppose $\mathcal{P}$ is a NW planar history or SE planar history. 
		Applying any of the local moves pictured below to $\mathcal{P}$ produces a new planar history $\mathcal{P}'$ so that $\demprod(\mathcal{P}) = \demprod(\mathcal{P}')$. 
			\[	\begin{tikzpicture}[x = 1.25em,y = .75em]
			\foreach \x in {1.5,2,3.5} {
				\draw[thick] (\x,0) -- (\x,4);
			}
			\draw[thick,rounded corners, color = blue] (5,1)--(4,1)--(4,3)--(0,3);
			\draw[thick,rounded corners, color = blue] (4,0)--(4,1)--(1,1)--(1,4);
			\draw (2.8,2) node {$\cdots$};
		\end{tikzpicture}
		\hspace{2em}
		\raisebox{1.3em}{$\leftrightarrow$} \hspace{2em}
		\begin{tikzpicture}[x = 1.25em,y = .75em]
			\foreach \x in {1.5,2,3.5} {
				\draw[thick] (\x,0) -- (\x,4);
			}
			\draw[thick,rounded corners, color = blue] (5,1)--(1,1)--(1,4);
			\draw[thick,rounded corners, color = blue] (4,0)--(4,3)--(0,3);
			\draw (2.8,2) node {$\cdots$};
		\end{tikzpicture}
		\hspace{2em}
		\raisebox{1.3em}{$\leftrightarrow$} \hspace{2em}
		\begin{tikzpicture}[x = 1.25em,y = .75em]
			\foreach \x in {1.5,2,3.5} {
				\draw[thick] (\x,0) -- (\x,4);
			}
			\draw[thick,rounded corners, color = blue] (4,0)--(4,3)--(1,3)--(1,4);
			\draw[thick,rounded corners, color = blue] (5,1)--(1,1)--(1,3)--(0,3);
			\draw (2.8,2) node {$\cdots$};
		\end{tikzpicture}
		\]
	\end{lemma} 
	\begin{proof}
		This follows by applying a sequence of local moves from \cref{lemma:planarmoves}.
	\end{proof}

	\section{Pipe dreams, bumpless pipe dreams, and their co-objects}
	\label{section:pipes}
	
	In this section, we recall fundamental facts about pipe dreams and BPDs. 
	We also define their co-objects: co-pipe dreams and co-BPDs. 
	These co-objects can be viewed as pipe dreams and BPDs flipped upside down.
	However, we find these orientations to be the most natural for stating our main theorems and arguments, so we maintain the distinction between these objects.
	
	\subsection{Pipe dreams}
	\label{subsection:pipes}
	
	Fix $n \in\mathbb{Z}_+$. 
	A \newword{pipe dream} of size $n$ is a tiling of the $n \times n$ grid using the tiles pictured below
	
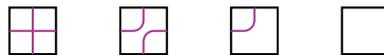
\[\begin{tikzpicture}[x = 1.5em,y = 1.5em]
		\draw[step = 1,gray, very thin] (0,0) grid (1, -1);
		\draw[color = black, thick] (0,0) rectangle (1, -1);
		\draw[thick,rounded corners,color = Mulberry](0, -1/2)--(1, -1/2);
		\draw[thick,rounded corners,color = Mulberry](1/2, 0)--(1/2, -1);
	\end{tikzpicture}
	\hspace{2em}
	\begin{tikzpicture}[x = 1.5em,y = 1.5em]
		\draw[step = 1,gray, very thin] (0,0) grid (1, -1);
		\draw[color = black, thick] (0,0) rectangle (1, -1);
		\draw[thick,rounded corners,color = Mulberry](1/2, 0)--(1/2, -1/2)--(0, -1/2);
		\draw[thick,rounded corners,color = Mulberry](1, -1/2)--(1/2, -1/2)--(1/2, -1);
	\end{tikzpicture}
	\hspace{2em}
	\begin{tikzpicture}[x = 1.5em,y = 1.5em]
		\draw[step = 1,gray, very thin] (0,0) grid (1, -1);
		\draw[color = black, thick] (0,0) rectangle (1, -1);
		\draw[thick,rounded corners,color = Mulberry](1/2, 0)--(1/2, -1/2)--(0, -1/2);
	\end{tikzpicture}
	\hspace{2em}
	\begin{tikzpicture}[x = 1.5em,y = 1.5em]
		\draw[step = 1,gray, very thin] (0,0) grid (1, -1);
		\draw[color = black, thick] (0,0) rectangle (1, -1);
	\end{tikzpicture}\]
	such that 
	\begin{enumerate}
		\item every tile on the antidiagonal is a \pipeultile,
		\item every tile strictly below the antidiagonal is a \btile, and
		\item every tile strictly above the antidiagonal is either a \pipecross
		or a \pipebump.
	\end{enumerate}
	Each such tiling gives rise to a network of $n$ pipes, which we interpret as a SW planar history. 
	Write $\kpipes_n$ for the set of all pipe dreams of size $n$, and $\pipes_n$ for the set of reduced pipe dreams of size $n$.
	
	Let $\mathbb{T}_n = \{(i,j) \in [n] \times [n] : i + j \leq n\}$.
	A pipe dream is uniquely determined by the positions of its crossing tiles.
	Moreover, the map 
	\[\mathcal{P} \mapsto \{(i,j) \in \mathbb{T}_n: \text{tile } (i,j) \text{ is a \pipecross in } \mathcal{P}\}\]
	is a bijection from $\kpipes_n$ to the set of subsets of $\mathbb{T}_n$. 
	For convenience, we freely identify pipe dreams with subsets of $\mathbb{T}_n$.
	
	To each pipe dream $\mathcal{P}$, we associate a word $\mathbf{a}_{\mathcal{P}}$ by reading crossings from right to left, breaking ties within columns by reading from top to bottom, recording the letter $i + j - 1$ for each $(i,j) \in \mathcal{P}$. 
	
	The following lemma states that the Demazure product of the word $\mathbf{a}_{\mathcal{P}}$ agrees with the permutation associated to $\mathcal{P}$, when we view $\mathcal{P}$ as a SW planar history.
	\begin{lemma}
		\label{lemma:pipeworddemproduct}
		Given $\mathcal{P} \in \pipes_n$, we have $\demprod(\mathcal{P}) = \demprod(\mathbf{a}_{\mathcal{P}})$.
	\end{lemma}
	\begin{proof}
		Recall that we may reflect $\mathcal{P}$ to obtain a NE planar history $\NE(\mathcal{P})$. Write $\mathbf{a}_{\mathcal{P}}=(a_1,a_2,\ldots,a_k)$. By reading along the pipes of $\NE(\mathcal{P})$ from left to right, breaking ties within columns by reading from bottom to top, we have valid ordering on crossings as described in \cref{lemma:planarhistoryword} and obtain the word $(n-a_1,n-a_2,\ldots,n-a_k)$.
		Thus, 
		\begin{align*}
		\demprod(\mathbf{a}_{\mathcal{P}})&=w_0\demprod((n-a_1,n-a_2,\ldots,n-a_k))w_0 &\text{(by \cref{lemma:conjugationandwords})}\\
		&=w_0\demprod(\NE(\mathcal{P})) w_0 &\text{(by \cref{lemma:planarhistoryword})}\\
		&=\demprod(\mathcal{P}) & \text{(by \cref{lemma:flipdiagramschangepermutation}),}
		\end{align*}
	as desired.
	\end{proof}	
	
		Given $w \in S_n$, we define 
		\[\pipes(w) = \{\mathcal{P} \in \pipes_n : \demprod(\mathcal{P}) = w\} \quad \text{and} \quad \kpipes(w) = \{\mathcal{P} \in \kpipes_n : \demprod(\mathcal{P}) = w\}.\]
	We associate a monomial weight to each pipe dream. Given $\mathcal{P} \in \kpipes_n$, define 
	\[\wt(\mathcal{P}) = \prod_{(i,j) \in \mathcal{P}} x_i.\]

	\begin{example}
		\label{example:pipedream}
		Let $\mathcal{P}$ be the pipe dream pictured below. 
		
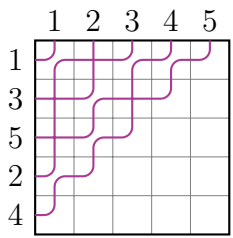
\[\begin{tikzpicture}[x=1.25em, y=1.25em]
			\draw[step=1, gray, very thin] (0,0) grid (5,-5);
			\draw[thick] (0,0) rectangle (5,-5);
			\draw[thick, rounded corners, Mulberry]
			(0.5,0)--(0.5,-0.5)--(0,-0.5)
			(1,-0.5)--(0.5,-0.5)--(0.5,-1)
			(1,-0.5)--(2,-0.5)
			(1.5,0)--(1.5,-1)
			(2.5,0)--(2.5,-0.5)--(2,-0.5)
			(3,-0.5)--(2.5,-0.5)--(2.5,-1)
			(3.5,0)--(3.5,-0.5)--(3,-0.5)
			(4,-0.5)--(3.5,-0.5)--(3.5,-1)
			(4.5,0)--(4.5,-0.5)--(4,-0.5)
			(0,-1.5)--(1,-1.5)
			(0.5,-1)--(0.5,-2)
			(1.5,-1)--(1.5,-1.5)--(1,-1.5)
			(2,-1.5)--(1.5,-1.5)--(1.5,-2)
			(2,-1.5)--(3,-1.5)
			(2.5,-1)--(2.5,-2)
			(3.5,-1)--(3.5,-1.5)--(3,-1.5)
			(0,-2.5)--(1,-2.5)
			(0.5,-2)--(0.5,-3)
			(1.5,-2)--(1.5,-2.5)--(1,-2.5)
			(2,-2.5)--(1.5,-2.5)--(1.5,-3)
			(2.5,-2)--(2.5,-2.5)--(2,-2.5)
			(0.5,-3)--(0.5,-3.5)--(0,-3.5)
			(1,-3.5)--(0.5,-3.5)--(0.5,-4)
			(1.5,-3)--(1.5,-3.5)--(1,-3.5)
			(0.5,-4)--(0.5,-4.5)--(0,-4.5);
			\node at (0.5, 0.5) {1};
			\node at (-0.5,-0.5) {1};
			\node at (1.5, 0.5) {2};
			\node at (-0.5,-1.5) {3};
			\node at (2.5, 0.5) {3};
			\node at (-0.5,-2.5) {5};
			\node at (3.5, 0.5) {4};
			\node at (-0.5,-3.5) {2};
			\node at (4.5, 0.5) {5};
			\node at (-0.5,-4.5) {4};
		\end{tikzpicture}
		\]
		Then $\mathcal{P}$ has word $\mathbf{a}_{\mathcal{P}} = (4,2,2,3)$ and $\demprod(\mathcal{P}) = 13524$. 
		Because $\ell(13524) = 3$, $\mathcal{P}$ is not reduced. The monomial weight of $\mathcal{P}$ is $\wt(\mathcal{P})=x_1x_2^2x_3$.
	\end{example}
		
We now recall the combinatorial formulas for Schubert and $\beta$-Grothendieck polynomials as sums over pipe dreams.
	
	\begin{theorem}
		\label{theorem:pipesschubertandgrothendieck}
		Let $w \in S_n$. Then
		\[\mathfrak{S}_w = \sum_{\mathcal{P} \in \pipes(w)} \wt(\mathcal{P})\]
		and 
		\[\mathfrak{G}^{(\beta)}_w = \sum_{\mathcal{P} \in \kpipes(w)} \beta^{\#\mathcal{P}-\ell(w)}\wt(\mathcal{P}).\]
	\end{theorem}
	
	The pipe dream formula for Schubert polynomials is due to \cite{Bergeron.Billey,Fomin.Kirillov.96}, building on work of \cite{BJS93,Fomin.Stanley}. 
	The formula for Grothendieck polynomials is due to \cite{Fomin.Kirillov}; see also \cite{Knutson.Miller:subword, Knutson.Miller}.

	\subsection{Co-pipe dreams}
	\label{subsection:copipes}
	
	A \newword{co-pipe dream} of size $n$ is a tiling of the $n \times n$ grid using the tiles pictured below
	
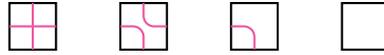
\[\begin{tikzpicture}[x = 1.5em,y = 1.5em]
		\draw[step = 1,gray, very thin] (0,0) grid (1, -1);
		\draw[color = black, thick] (0,0) rectangle (1, -1);
		\draw[thick,rounded corners,color = Rhodamine](0, -1/2)--(1, -1/2);
		\draw[thick,rounded corners,color = Rhodamine](1/2, 0)--(1/2, -1);
	\end{tikzpicture}
	\hspace{2em}
	\begin{tikzpicture}[x = 1.5em,y = 1.5em]
		\draw[step = 1,gray, very thin] (0,0) grid (1, -1);
		\draw[color = black, thick] (0,0) rectangle (1, -1);
		\draw[thick,rounded corners,color = Rhodamine](1/2, 0)--(1/2, -1/2)--(1, -1/2);
		\draw[thick,rounded corners,color = Rhodamine](0, -1/2)--(1/2, -1/2)--(1/2, -1);
	\end{tikzpicture}
	\hspace{2em}
	\begin{tikzpicture}[x = 1.5em,y = 1.5em]
		\draw[step = 1,gray, very thin] (0,0) grid (1, -1);
		\draw[color = black, thick] (0,0) rectangle (1, -1);
		\draw[thick,rounded corners,color = Rhodamine](1/2, -1)--(1/2, -1/2)--(0, -1/2);
	\end{tikzpicture}
	\hspace{2em}
	\begin{tikzpicture}[x = 1.5em,y = 1.5em]
		\draw[step = 1,gray, very thin] (0,0) grid (1, -1);
		\draw[color = black, thick] (0,0) rectangle (1, -1);
	\end{tikzpicture}\]
	such that
	\begin{enumerate}
		\item every tile on the diagonal is a \dltile,
		\item every tile strictly above the diagonal is a \btile, and
		\item every tile strictly below the diagonal is either a \copipecross
		or a \copipebump.
	\end{enumerate}
	
	Each such tiling gives rise to a network of $n$ pipes, which we interpret as a NW planar history. Write $\kcopipes_n$ for the set of all co-pipe dreams of size $n$, and $\copipes_n$ for the set of reduced co-pipe dreams of size $n$.
	Given $w \in S_n$, we define 
	\[\copipes(w) = \{\mathcal{P} \in \copipes_n : \demprod(\mathcal{P}) = w\} \quad \text{and} \quad \kcopipes(w) = \{\mathcal{P} \in \kcopipes_n : \demprod(\mathcal{P}) = w\}.\]

	Let $\check{\mathbb{T}}_n = \{(i, j) \in [n] \times [n] : i - j \geq 1\}$. 
	Each co-pipe dream is uniquely determined by the positions of its crossing tiles, so we identify co-pipe dreams with subsets of $\check{\mathbb{T}}_n$.

	Flipping a co-pipe dream $\mathcal{P} \in \copipes_n$ upside down yields a pipe dream $\mathcal{Q}$. 
	We define the word of $\mathcal{P}$ to be the same as the word of $\mathcal{Q}$, and write $\mathbf{a}_{\mathcal{P}} = \mathbf{a}_{\mathcal{Q}}$. 
	Equivalently, one obtains $\mathbf{a}_{\mathcal{P}}$ by reading crossings in $\mathcal{P}$ from right to left, breaking ties within columns by reading from bottom to top, and recording $n-i+j$ whenever there is a crossing in position $(i,j)$.

	\begin{lemma}
		\label{lemma:copipewordandpermutation}
		Given $\mathcal{P} \in \kcopipes_n$, we have $\demprod(\mathcal{P}) = \demprod(\mathbf{a}_{\mathcal{P}}) \cdot w_0$.
	\end{lemma}
	\begin{proof}
		This follows directly from \cref{lemma:pipeworddemproduct} and \cref{lemma:flipdiagramschangepermutation}.
	\end{proof}
	
	The map from $\kpipes_n$ to $\kcopipes_n$ given by reflecting pipe dreams vertically is a bijection. 
	A different bijection will be used in our change of basis formulas. 
	
	\begin{lemma}
		There is a bijection from $\kpipes_n$ to $\kcopipes_n$ defined by $\mathcal{P} \mapsto \check{\mathcal{P}}$ where $\check{\mathcal{P}} = \{(i + j,j) : (i,j) \in \mathbb{T}_n-\mathcal{P}\}.$
	\end{lemma}
	\begin{proof}
		The map $(i,j) \mapsto (i + j,j)$ defines a bijection from $\mathbb{T}_n$ to $\check{\mathbb T}_n$. 
		Thus it induces a bijection from subsets of $\mathbb{T}_n$ to subsets of $\check{\mathbb T}_n$. 
		The map 
		\[\mathcal{P} \mapsto \{(i + j,j):(i,j) \in \mathbb{T}_n-\mathcal{P}\}\]
		is the result of precomposing this bijection with complementation, and hence is itself a bijection. 
	\end{proof}
	
	We call $\check{\mathcal{P}}$ the co-pipe dream \newword{associated to} $\mathcal{P}$. 
	We also say that $\demprod(\check{\mathcal{P}})$ is the \newword{co-permutation} of $\mathcal{P}$.
	
	It is helpful to visualize this map as follows. 
	First, remove each tile on the main antidiagonal. 
	Then, change each bump \pipebump to a cross \copipecross and each cross \pipecross to an upside down bump \copipebump. 
	Next, keeping tiles in order within their columns, move them downward so that the crosses and bumps are bottom-justified in the $n\times n$ grid. 
	Finally, place a downward bump \dltile in every cell along the diagonal.

	\begin{example}
		Let $\mathcal{P}$ be the pipe dream from \cref{example:pipedream}. 
		Its associated co-pipe dream $\check{\mathcal{P}}$ is pictured below.
		
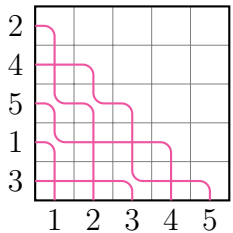
\[\begin{tikzpicture}[x = 1.25em,y = 1.25em]
			\draw[step = 1,gray, very thin] (0,0) grid (5, -5);
			\draw[color = black, thick] (0,0) rectangle (5, -5);
			\draw[thick,rounded corners,color = Rhodamine](1/2, -1)--(1/2, -1/2)--(0, -1/2);
			\draw[thick,rounded corners,color = Rhodamine](0, -3/2)--(1, -3/2);
			\draw[thick,rounded corners,color = Rhodamine](1/2, -1)--(1/2, -2);
			\draw[thick,rounded corners,color = Rhodamine](3/2, -2)--(3/2, -3/2)--(1, -3/2);
			\draw[thick,rounded corners,color = Rhodamine](1/2, -2)--(1/2, -5/2)--(1, -5/2);
			\draw[thick,rounded corners,color = Rhodamine](0, -5/2)--(1/2, -5/2)--(1/2, -3);
			\draw[thick,rounded corners,color = Rhodamine](3/2, -2)--(3/2, -5/2)--(2, -5/2);
			\draw[thick,rounded corners,color = Rhodamine](1, -5/2)--(3/2, -5/2)--(3/2, -3);
			\draw[thick,rounded corners,color = Rhodamine](5/2, -3)--(5/2, -5/2)--(2, -5/2);
			\draw[thick,rounded corners,color = Rhodamine](1/2, -3)--(1/2, -7/2)--(1, -7/2);
			\draw[thick,rounded corners,color = Rhodamine](0, -7/2)--(1/2, -7/2)--(1/2, -4);
			\draw[thick,rounded corners,color = Rhodamine](1, -7/2)--(2, -7/2);
			\draw[thick,rounded corners,color = Rhodamine](3/2, -3)--(3/2, -4);
			\draw[thick,rounded corners,color = Rhodamine](2, -7/2)--(3, -7/2);
			\draw[thick,rounded corners,color = Rhodamine](5/2, -3)--(5/2, -4);
			\draw[thick,rounded corners,color = Rhodamine](7/2, -4)--(7/2, -7/2)--(3, -7/2);
			\draw[thick,rounded corners,color = Rhodamine](0, -9/2)--(1, -9/2);
			\draw[thick,rounded corners,color = Rhodamine](1/2, -4)--(1/2, -5);
			\draw[thick,rounded corners,color = Rhodamine](1, -9/2)--(2, -9/2);
			\draw[thick,rounded corners,color = Rhodamine](3/2, -4)--(3/2, -5);
			\draw[thick,rounded corners,color = Rhodamine](5/2, -4)--(5/2, -9/2)--(3, -9/2);
			\draw[thick,rounded corners,color = Rhodamine](2, -9/2)--(5/2, -9/2)--(5/2, -5);
			\draw[thick,rounded corners,color = Rhodamine](3, -9/2)--(4, -9/2);
			\draw[thick,rounded corners,color = Rhodamine](7/2, -4)--(7/2, -5);
			\draw[thick,rounded corners,color = Rhodamine](9/2, -5)--(9/2, -9/2)--(4, -9/2);
			\node at (1/2, -11/2){1};
			\node at (-1/2, -1/2){2};
			\node at (3/2, -11/2){2};
			\node at (-1/2, -3/2){4};
			\node at (5/2, -11/2){3};
			\node at (-1/2, -5/2){5};
			\node at (7/2, -11/2){4};
			\node at (-1/2, -7/2){1};
			\node at (9/2, -11/2){5};
			\node at (-1/2, -9/2){3};
		\end{tikzpicture}
		\]
		We have $\demprod(\check{\mathcal{P}}) = 24513$. 
		Note also that $\mathbf{a}_{\check{\mathcal{P}}} = (4,4,2,3,1,4)$. 
		Furthermore, $\demprod(\mathbf{a}_{\check{\mathcal{P}}}) = 31542$, and thus $\demprod(\check{\mathcal{P}}) = \demprod(\mathbf{a}_{\check{\mathcal{P}}}) w_0$, agreeing with \cref{lemma:copipewordandpermutation}.
	\end{example}
	
	The next lemma says that given a pipe dream $\mathcal{P}$, the descent set of the co-permutation of $\mathcal{P}$ is a subset of the descent set of $\demprod(\mathcal{P})$.
	\begin{lemma}
		\label{lemma:copdascent}
		Let $\mathcal{P} \in \kpipes_n$. 
		If $\demprod(\mathcal{P})$ has an ascent at $k$, then $\demprod(\check{\mathcal{P}})$ also has an ascent at $k$. In other words, $\des(\demprod(\check{\mathcal{P}}))\subseteq \des(\demprod(\mathcal{P}))$.
	\end{lemma}
	\begin{proof}
		If $\demprod(\mathcal{P})$ has an ascent at $k$, then the pipe exiting the grid in row $k$ of $\mathcal{P}$ has a smaller label than the pipe exiting in row $k+1$. 
		By \cref{lemma:planarhistoryinversions}, this implies that the two pipes do not cross in~$\mathcal{P}$.
		In particular, the tile at position $(k, 1)$ is not a \pipecross. 
		This implies that in $\check{\mathcal{P}}$, the tile at $(k + 1,1 )$ is a \copipecross. 
		Thus, the pipes exiting in rows $k + 1$ and $k$ in $\check{\mathcal{P}}$ must cross. 
		Since $\check{\mathcal{P}}$ is a NW planar history, it follows from \cref{lemma:planarhistoryinversions} that $\demprod(\check{\mathcal{P}})$ has an ascent at $k$.
	\end{proof}
	
	\subsection{Bumpless pipe dreams}
	\label{subsection:BDP}
	
	A \newword{bumpless pipe dream (BPD)} of size $n$ is a tiling of the $n \times n$ grid using the tiles pictured below
	\begin{equation}
		\label{eqn:BPDtiles}
		\begin{tikzpicture}[x = 1.5em,y = 1.5em]
			\draw[step = 1,gray, very thin] (0,0) grid (1, -1);
			\draw[color = black, thick] (0,0) rectangle (1, -1);
			\draw[thick,rounded corners,color = blue](1/2, -1)--(1/2, -1/2)--(1, -1/2);
		\end{tikzpicture} 
		\hspace{2em}
		\begin{tikzpicture}[x = 1.5em,y = 1.5em]
			\draw[step = 1,gray, very thin] (0,0) grid (1, -1);
			\draw[color = black, thick] (0,0) rectangle (1, -1);
			\draw[thick,rounded corners,color = blue](1/2, 0)--(1/2, -1/2)--(0, -1/2);
		\end{tikzpicture}
		\hspace{2em}
		\begin{tikzpicture}[x = 1.5em,y = 1.5em]
			\draw[step = 1,gray, very thin] (0,0) grid (1, -1);
			\draw[color = black, thick] (0,0) rectangle (1, -1);
			\draw[thick,rounded corners,color = blue](0, -1/2)--(1, -1/2);
		\end{tikzpicture}
		\hspace{2em}
		\begin{tikzpicture}[x = 1.5em,y = 1.5em]
			\draw[step = 1,gray, very thin] (0,0) grid (1, -1);
			\draw[color = black, thick] (0,0) rectangle (1, -1);
			\draw[thick,rounded corners,color = blue](1/2, 0)--(1/2, -1);
		\end{tikzpicture}
		\hspace{2em}
		\begin{tikzpicture}[x = 1.5em,y = 1.5em]
			\draw[step = 1,gray, very thin] (0,0) grid (1, -1);
			\draw[color = black, thick] (0,0) rectangle (1, -1);
			\draw[thick,rounded corners,color = blue](0, -1/2)--(1, -1/2);
			\draw[thick,rounded corners,color = blue](1/2, 0)--(1/2, -1);
		\end{tikzpicture}
		\hspace{2em}
		\begin{tikzpicture}[x = 1.5em,y = 1.5em]
			\draw[step = 1,gray, very thin] (0,0) grid (1, -1);
			\draw[color = black, thick] (0,0) rectangle (1, -1);
		\end{tikzpicture}
	\end{equation}
	so that we get a network of $n$ pipes, with each pipe starting at the bottom edge of the grid and ending at the right. 
	Any such tiling is a NE planar history.

	Write $\kbpd_n$ for the set of all BPDs of size $n$, and $\bpd_n$ for the reduced BPDs of size $n$. 
	Given $w \in S_n$, let 
	\[\bpd(w) = \{\mathcal{B} \in \bpd_n : \demprod(\mathcal{B}) = w\}\quad \text{and} \quad\kbpd(w) = \{\mathcal{B} \in \kbpd_n :\demprod(\mathcal{B}) = w\}.\]
	For each $w \in S_n$, there is a unique BPD in $\kbpd(w)$ which does not contain any \ultile tiles. 
	We call this the \newword{Rothe BPD} for $w$. 
	All BPDs of $w$ can be obtained from the Rothe BPD by applying a sequence of \emph{droop moves} and \emph{K-theoretic droop moves} (see \cite{Lam.Lee.Shimozono, Weigandt}).

	Given $\mathcal{B} \in \kbpd_n$, define 
	\[\rothe(\mathcal{B}) = \{(i,j) : (i,j) \text{ is a \btile in } \mathcal{B}\}\] and 
	\[\up(\mathcal{B}) = \{(i,j):(i,j) \text{ is a \ultile in } \mathcal{B}\}.\]
	There are two types of weights that we assign to BPDs. 
	Given $\mathcal{B} \in \kbpd_n$, we define the ordinary weight to be \[\wt(\mathcal{B}) = \prod_{(i,j) \in \rothe(\mathcal{B})} x_i.\] 
	There is also a K-theoretic weight defined by
	\[\kwt(\mathcal{B}) = \left(\prod_{(i,j) \in \rothe(\mathcal{B})} \beta x_i\right) \left( \prod_{(i,j) \in \up(\mathcal{B})} 1 + \beta x_i \right).\]
	
	\begin{theorem}
		\label{bpd:pipesschubertandgrothendieck}
		Let $w \in S_n$. 
		Then
		\[\mathfrak{S}_w = \sum_{\mathcal{B} \in \bpd(w)} \wt(\mathcal{B}),\]
		and 
		\[\mathfrak{G}^{(\beta)}_w = \beta^{-\ell(w)}\sum_{\mathcal{B} \in \kbpd(w)} \kwt(\mathcal{B}).\]
	\end{theorem}

	The BPD formula for Schubert polynomials is due to \cite{Lam.Lee.Shimozono}. 
	The formula for Grothendieck polynomials is given in \cite{Weigandt}, and is a translation of an earlier formula of Lascoux \cite{Lascoux}, originally stated in terms of \emph{alternating sign matrices}.

	\subsection{Co-bumpless pipe dreams}
	
	A \newword{co-bumpless pipe dream (co-BPD)} of size $n$ is a tiling of the $n \times n$ grid using the tiles pictured below
	\begin{equation}
		\label{eqn:coBPDtiles}
		\begin{tikzpicture}[x = 1.5em,y = 1.5em]
			\draw[step = 1,gray, very thin] (0,0) grid (1, -1);
			\draw[color = black, thick] (0,0) rectangle (1, -1);
			\draw[thick,rounded corners,color = ForestGreen](1/2, 0)--(1/2, -1/2)--(1, -1/2);
		\end{tikzpicture} 
		\hspace{2em}
		\begin{tikzpicture}[x = 1.5em,y = 1.5em]
			\draw[step = 1,gray, very thin] (0,0) grid (1, -1);
			\draw[color = black, thick] (0,0) rectangle (1, -1);
			\draw[thick,rounded corners,color = ForestGreen](1/2, -1)--(1/2, -1/2)--(0, -1/2);
		\end{tikzpicture}
		\hspace{2em}
		\begin{tikzpicture}[x = 1.5em,y = 1.5em]
			\draw[step = 1,gray, very thin] (0,0) grid (1, -1);
			\draw[color = black, thick] (0,0) rectangle (1, -1);
			\draw[thick,rounded corners,color = ForestGreen](0, -1/2)--(1, -1/2);
			\draw[thick,rounded corners,color = ForestGreen](1/2, 0)--(1/2, -1);
		\end{tikzpicture}
		\hspace{2em}
		\begin{tikzpicture}[x = 1.5em,y = 1.5em]
			\draw[step = 1,gray, very thin] (0,0) grid (1, -1);
			\draw[color = black, thick] (0,0) rectangle (1, -1);
		\end{tikzpicture}
		\hspace{2em}
		\begin{tikzpicture}[x = 1.5em,y = 1.5em]
			\draw[step = 1,gray, very thin] (0,0) grid (1, -1);
			\draw[color = black, thick] (0,0) rectangle (1, -1);
			\draw[thick,rounded corners,color = ForestGreen](0, -1/2)--(1, -1/2);
		\end{tikzpicture}
		\hspace{2em}
		\begin{tikzpicture}[x = 1.5em,y = 1.5em]
			\draw[step = 1,gray, very thin] (0,0) grid (1, -1);
			\draw[color = black, thick] (0,0) rectangle (1, -1);
			\draw[thick,rounded corners,color = ForestGreen](1/2, 0)--(1/2, -1);
		\end{tikzpicture}
	\end{equation}
	so that we obtain a network of $n$ pipes, with each pipe starting at the top of the grid and ending at the right. 
	Any such tiling is a SE planar history. 
		
	Write $\kcobpd_n$ for the set of all $n \times n$ co-BPDs, and $\cobpd_n$ for the set of reduced $n \times n$ co-BPDs. 
	Given $w \in S_n$, define 
	\[\cobpd(w) = \{\mathcal{B} \in \cobpd_n : \demprod(\mathcal{B}) = w\} \quad \text{and} \quad \kcobpd(w) = \{\mathcal{B} \in \kcobpd_n : \demprod(\mathcal{B}) = w\}.\]

	If we flip a co-BPD upside down, we obtain a BPD. 
	This map is clearly bijective. 
	However, we will use a different bijection in our change of basis formulas. 
	\begin{lemma}
		\label{lemma:copipedreambijection}
		There is a bijection from $\kbpd_n\rightarrow \kcobpd_n$ defined by making the following tile-by-tile replacements:
		\[\begin{tikzpicture}[x = 1.5em,y = 1.5em]
			\draw[step = 1,gray, very thin] (0,0) grid (1, -1);
			\draw[color = black, thick] (0,0) rectangle (1, -1);
			\draw[thick,rounded corners,color = blue](1/2, -1)--(1/2, -1/2)--(1, -1/2);
		\end{tikzpicture} 
		\quad 
		\raisebox{.5em}{$\mapsto$}
		\quad
		\begin{tikzpicture}[x = 1.5em,y = 1.5em]
			\draw[step = 1,gray, very thin] (0,0) grid (1, -1);
			\draw[color = black, thick] (0,0) rectangle (1, -1);
			\draw[thick,rounded corners,color = ForestGreen](1/2, 0)--(1/2, -1/2)--(1, -1/2);
		\end{tikzpicture} 
		\hspace{3em}
		\begin{tikzpicture}[x = 1.5em,y = 1.5em]
			\draw[step = 1,gray, very thin] (0,0) grid (1, -1);
			\draw[color = black, thick] (0,0) rectangle (1, -1);
			\draw[thick,rounded corners,color = blue](1/2, 0)--(1/2, -1/2)--(0, -1/2);
		\end{tikzpicture}
		\quad 
		\raisebox{.5em}{$\mapsto$}
		\quad
		\begin{tikzpicture}[x = 1.5em,y = 1.5em]
			\draw[step = 1,gray, very thin] (0,0) grid (1, -1);
			\draw[color = black, thick] (0,0) rectangle (1, -1);
			\draw[thick,rounded corners,color = ForestGreen](1/2, -1)--(1/2, -1/2)--(0, -1/2);
		\end{tikzpicture}
		\hspace{3em}
		\begin{tikzpicture}[x = 1.5em,y = 1.5em]
			\draw[step = 1,gray, very thin] (0,0) grid (1, -1);
			\draw[color = black, thick] (0,0) rectangle (1, -1);
			\draw[thick,rounded corners,color = blue](0, -1/2)--(1, -1/2);
		\end{tikzpicture}
		\quad 
		\raisebox{.5em}{$\mapsto$}
		\quad
		\begin{tikzpicture}[x = 1.5em,y = 1.5em]
			\draw[step = 1,gray, very thin] (0,0) grid (1, -1);
			\draw[color = black, thick] (0,0) rectangle (1, -1);
			\draw[thick,rounded corners,color = ForestGreen](0, -1/2)--(1, -1/2);
			\draw[thick,rounded corners,color = ForestGreen](1/2, 0)--(1/2, -1);
		\end{tikzpicture}
		\]\[
		\begin{tikzpicture}[x = 1.5em,y = 1.5em]
			\draw[step = 1,gray, very thin] (0,0) grid (1, -1);
			\draw[color = black, thick] (0,0) rectangle (1, -1);
			\draw[thick,rounded corners,color = blue](1/2, 0)--(1/2, -1);
		\end{tikzpicture}
		\quad 
		\raisebox{.5em}{$\mapsto$}
		\quad
		\begin{tikzpicture}[x = 1.5em,y = 1.5em]
			\draw[step = 1,gray, very thin] (0,0) grid (1, -1);
			\draw[color = black, thick] (0,0) rectangle (1, -1);
		\end{tikzpicture}
		\hspace{3em}
		\begin{tikzpicture}[x = 1.5em,y = 1.5em]
			\draw[step = 1,gray, very thin] (0,0) grid (1, -1);
			\draw[color = black, thick] (0,0) rectangle (1, -1);
			\draw[thick,rounded corners,color = blue](0, -1/2)--(1, -1/2);
			\draw[thick,rounded corners,color = blue](1/2, 0)--(1/2, -1);
		\end{tikzpicture}
		\quad 
		\raisebox{.5em}{$\mapsto$}
		\quad
		\begin{tikzpicture}[x = 1.5em,y = 1.5em]
			\draw[step = 1,gray, very thin] (0,0) grid (1, -1);
			\draw[color = black, thick] (0,0) rectangle (1, -1);
			\draw[thick,rounded corners,color = ForestGreen](0, -1/2)--(1, -1/2);
		\end{tikzpicture}
		\hspace{3em}
		\begin{tikzpicture}[x = 1.5em,y = 1.5em]
			\draw[step = 1,gray, very thin] (0,0) grid (1, -1);
			\draw[color = black, thick] (0,0) rectangle (1, -1);
		\end{tikzpicture}
		\quad 
		\raisebox{.5em}{$\mapsto$}
		\quad
		\begin{tikzpicture}[x = 1.5em,y = 1.5em]
			\draw[step = 1,gray, very thin] (0,0) grid (1, -1);
			\draw[color = black, thick] (0,0) rectangle (1, -1);
			\draw[thick,rounded corners,color = ForestGreen](1/2, 0)--(1/2, -1);
		\end{tikzpicture}\]
	\end{lemma}
	\begin{proof}
		It is immediate from the substitutions that this map is well-defined and invertible.
	\end{proof}
	
	Given $\mathcal{B} \in \kbpd_n$, we write $\check{\mathcal{B}}$ for its corresponding co-BPD under the bijection from \cref{lemma:copipedreambijection} and call $\check{\mathcal{B}}$ the co-BPD \newword{associated to} $\mathcal{B}$.
	We define the \newword{co-permutation} of $\mathcal{B}$ to be $\demprod(\check{\mathcal{B}})$.

	\begin{remark}
		BPDs are in direct bijection with states the six-vertex model. 
		As explained in \cite{Weigandt}, one can obtain a BPD $\mathcal{B}$ from a state of the six-vertex model by drawing pipe segments corresponding to the arrows that point left and down. 
		To obtain the corresponding co-BPD $\check{\mathcal{B}}$, instead select the arrows that point left and up.
	\end{remark}

	\begin{lemma}
		\label{lemma:cobpdascent}
		Let $\mathcal{B} \in \kbpd_n$, and fix $k \in [n-1]$.
		If $\demprod(\mathcal{B})$ has an ascent at $k$, then $\demprod(\check{\mathcal{B}})$ also has an ascent at $k$. In other words, $\des(\demprod(\check{\mathcal{B}}))\subseteq \des(\demprod(\mathcal{B}))$.
	\end{lemma}
	\begin{proof}
		We will prove the second statement, from which the first follows.
		
		Assume $\check{\mathcal{B}}$ has a descent at $k$.
		Then by \cref{lemma:planarhistoryinversions}, the pipes that exit the grid in rows $k$ and $k + 1$ of $\check{\mathcal{B}}$ do not cross. 
		This implies that the last \coultile in row $k$ of $\check{\mathcal{B}}$ occurs strictly to the right of the last \coultile in row $k + 1$. 
		Thus, in $\mathcal{B}$, the last \drtile in row $k$ occurs strictly to the right of the last \drtile in row $k + 1$. 
		Therefore, the pipes ending in row $k$ and $k + 1$ in $\mathcal{B}$ must cross. By \cref{lemma:planarhistoryinversions}, this implies that $\demprod(\mathcal{B})$ has a descent at position $k$. 
	\end{proof}

		\section{Co-transition recurrences}

		\label{section:cotransition}
		
		The goal of this section is to establish the K-theoretic co-transition recurrence on pipe dreams of \cite{Knutson:cotransition}. We will use facts about K-theoretic ladder moves on pipe dreams, which we prove in the next section.
		
		\subsection{Ladder moves on pipe dreams}
		We begin by describing certain local replacements on pipe dreams which are permutation preserving. 
		Let $w \in S_n$ and $\mathcal{P} \in \kpipes(w)$. 
		Fix $i,j,k \in [n-1]$ such that $i\leq j$. Suppose further that $[i + 1, j] \times [k, k + 1] \subseteq \mathcal{P}$, $(j + 1,k) \in \mathcal{P}$, and $(j + 1,k + 1),(i,k),(i,k + 1)\notin \mathcal{P}$. 
		We call the replacement
		\[\mathcal{P} \mapsto (\mathcal{P}-\{(j + 1, k)\}) \cup \{(i, k + 1)\}\] 
		a \newword{ladder move} and the replacement 
		\[\mathcal{P} \mapsto \mathcal{P} \cup \{(i, k + 1)\}\] a \newword{K-theoretic ladder move}. 
		See \cref{figure:ladderandkladder} for examples of these moves.
		Note that when $i = j$, the set $[i + 1,j] \times [k,k + 1]$ is empty. 
		In this case, we refer to the (K-theoretic) ladder move as \newword{simple}. Pictured below on the left is a simple ladder move; on the right is a simple K-theoretic ladder move.
		
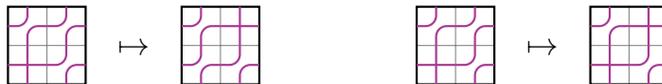
\[
		\begin{tikzpicture}[x = 1.25em,y = 1.25em]
			\draw[step = 1,gray, very thin] (0,0) grid (2, -2);
			\draw[color = black, thick] (0,0) rectangle (2, -2);
			\draw[thick,rounded corners,color = Mulberry](1/2, 0)--(1/2, -1/2)--(0, -1/2);
			\draw[thick,rounded corners,color = Mulberry](1, -1/2)--(1/2, -1/2)--(1/2, -1);
			\draw[thick,rounded corners,color = Mulberry](3/2, 0)--(3/2, -1/2)--(1, -1/2);
			\draw[thick,rounded corners,color = Mulberry](2, -1/2)--(3/2, -1/2)--(3/2, -1);
			\draw[thick,rounded corners,color = Mulberry](0, -3/2)--(1, -3/2);
			\draw[thick,rounded corners,color = Mulberry](1/2, -1)--(1/2, -2);
			\draw[thick,rounded corners,color = Mulberry](3/2, -1)--(3/2, -3/2)--(1, -3/2);
			\draw[thick,rounded corners,color = Mulberry](2, -3/2)--(3/2, -3/2)--(3/2, -2);
		\end{tikzpicture}
		\quad
		\raisebox{1em}{$\mapsto$}
		\quad
		\begin{tikzpicture}[x = 1.25em,y = 1.25em]
			\draw[step = 1,gray, very thin] (0,0) grid (2, -2);
			\draw[color = black, thick] (0,0) rectangle (2, -2);
			\draw[thick,rounded corners,color = Mulberry](1/2, 0)--(1/2, -1/2)--(0, -1/2);
			\draw[thick,rounded corners,color = Mulberry](1, -1/2)--(1/2, -1/2)--(1/2, -1);
			\draw[thick,rounded corners,color = Mulberry](1, -1/2)--(2, -1/2);
			\draw[thick,rounded corners,color = Mulberry](3/2, 0)--(3/2, -1);
			\draw[thick,rounded corners,color = Mulberry](1/2, -1)--(1/2, -3/2)--(0, -3/2);
			\draw[thick,rounded corners,color = Mulberry](1, -3/2)--(1/2, -3/2)--(1/2, -2);
			\draw[thick,rounded corners,color = Mulberry](3/2, -1)--(3/2, -3/2)--(1, -3/2);
			\draw[thick,rounded corners,color = Mulberry](2, -3/2)--(3/2, -3/2)--(3/2, -2);
		\end{tikzpicture}
		\hspace{5em}
		\begin{tikzpicture}[x = 1.25em,y = 1.25em]
			\draw[step = 1,gray, very thin] (0,0) grid (2, -2);
			\draw[color = black, thick] (0,0) rectangle (2, -2);
			\draw[thick,rounded corners,color = Mulberry](1/2, 0)--(1/2, -1/2)--(0, -1/2);
			\draw[thick,rounded corners,color = Mulberry](1, -1/2)--(1/2, -1/2)--(1/2, -1);
			\draw[thick,rounded corners,color = Mulberry](3/2, 0)--(3/2, -1/2)--(1, -1/2);
			\draw[thick,rounded corners,color = Mulberry](2, -1/2)--(3/2, -1/2)--(3/2, -1);
			\draw[thick,rounded corners,color = Mulberry](0, -3/2)--(1, -3/2);
			\draw[thick,rounded corners,color = Mulberry](1/2, -1)--(1/2, -2);
			\draw[thick,rounded corners,color = Mulberry](3/2, -1)--(3/2, -3/2)--(1, -3/2);
			\draw[thick,rounded corners,color = Mulberry](2, -3/2)--(3/2, -3/2)--(3/2, -2);
		\end{tikzpicture}
		\quad
		\raisebox{1em}{$\mapsto$}
		\quad
		\begin{tikzpicture}[x = 1.25em,y = 1.25em]
			\draw[step = 1,gray, very thin] (0,0) grid (2, -2);
			\draw[color = black, thick] (0,0) rectangle (2, -2);
			\draw[thick,rounded corners,color = Mulberry](1/2, 0)--(1/2, -1/2)--(0, -1/2);
			\draw[thick,rounded corners,color = Mulberry](1, -1/2)--(1/2, -1/2)--(1/2, -1);
			\draw[thick,rounded corners,color = Mulberry](1, -1/2)--(2, -1/2);
			\draw[thick,rounded corners,color = Mulberry](3/2, 0)--(3/2, -1);
			\draw[thick,rounded corners,color = Mulberry](0, -3/2)--(1, -3/2);
			\draw[thick,rounded corners,color = Mulberry](1/2, -1)--(1/2, -2);
			\draw[thick,rounded corners,color = Mulberry](3/2, -1)--(3/2, -3/2)--(1, -3/2);
			\draw[thick,rounded corners,color = Mulberry](2, -3/2)--(3/2, -3/2)--(3/2, -2);
		\end{tikzpicture}
		\]
		If $\mathcal{P} \mapsto \mathcal{Q}$ is a (K-theoretic) ladder move, then the reverse operation $\mathcal{Q} \mapsto \mathcal{P}$ is called an \newword{inverse (K-theoretic) ladder move}.
		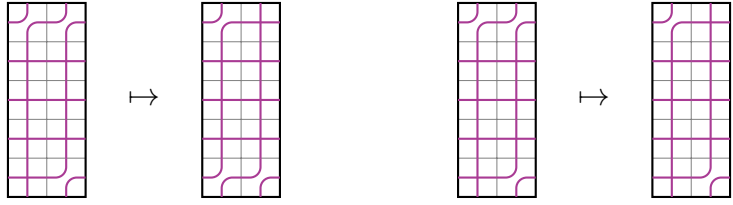
\begin{figure}
			\begin{tikzpicture}[x = 1.25em,y = 1.25em]
				\draw[step = 1,gray, very thin] (0,0) grid (2, -5);
				\draw[color = black, thick] (0,0) rectangle (2, -5);
				\draw[thick,rounded corners,color = Mulberry](1/2, 0)--(1/2, -1/2)--(0, -1/2);
				\draw[thick,rounded corners,color = Mulberry](1, -1/2)--(1/2, -1/2)--(1/2, -1);
				\draw[thick,rounded corners,color = Mulberry](3/2, 0)--(3/2, -1/2)--(1, -1/2);
				\draw[thick,rounded corners,color = Mulberry](2, -1/2)--(3/2, -1/2)--(3/2, -1);
				\draw[thick,rounded corners,color = Mulberry](0, -3/2)--(1, -3/2);
				\draw[thick,rounded corners,color = Mulberry](1/2, -1)--(1/2, -2);
				\draw[thick,rounded corners,color = Mulberry](1, -3/2)--(2, -3/2);
				\draw[thick,rounded corners,color = Mulberry](3/2, -1)--(3/2, -2);
				\draw[thick,rounded corners,color = Mulberry](0, -5/2)--(1, -5/2);
				\draw[thick,rounded corners,color = Mulberry](1/2, -2)--(1/2, -3);
				\draw[thick,rounded corners,color = Mulberry](1, -5/2)--(2, -5/2);
				\draw[thick,rounded corners,color = Mulberry](3/2, -2)--(3/2, -3);
				\draw[thick,rounded corners,color = Mulberry](0, -7/2)--(1, -7/2);
				\draw[thick,rounded corners,color = Mulberry](1/2, -3)--(1/2, -4);
				\draw[thick,rounded corners,color = Mulberry](1, -7/2)--(2, -7/2);
				\draw[thick,rounded corners,color = Mulberry](3/2, -3)--(3/2, -4);
				\draw[thick,rounded corners,color = Mulberry](0, -9/2)--(1, -9/2);
				\draw[thick,rounded corners,color = Mulberry](1/2, -4)--(1/2, -5);
				\draw[thick,rounded corners,color = Mulberry](3/2, -4)--(3/2, -9/2)--(1, -9/2);
				\draw[thick,rounded corners,color = Mulberry](2, -9/2)--(3/2, -9/2)--(3/2, -5);
			\end{tikzpicture}
			\quad
			\raisebox{3em}{$\mapsto$}
			\quad
			\begin{tikzpicture}[x = 1.25em,y = 1.25em]
				\draw[step = 1,gray, very thin] (0,0) grid (2, -5);
				\draw[color = black, thick] (0,0) rectangle (2, -5);
				\draw[thick,rounded corners,color = Mulberry](1/2, 0)--(1/2, -1/2)--(0, -1/2);
				\draw[thick,rounded corners,color = Mulberry](1, -1/2)--(1/2, -1/2)--(1/2, -1);
				\draw[thick,rounded corners,color = Mulberry](1, -1/2)--(2, -1/2);
				\draw[thick,rounded corners,color = Mulberry](3/2, 0)--(3/2, -1);
				\draw[thick,rounded corners,color = Mulberry](0, -3/2)--(1, -3/2);
				\draw[thick,rounded corners,color = Mulberry](1/2, -1)--(1/2, -2);
				\draw[thick,rounded corners,color = Mulberry](1, -3/2)--(2, -3/2);
				\draw[thick,rounded corners,color = Mulberry](3/2, -1)--(3/2, -2);
				\draw[thick,rounded corners,color = Mulberry](0, -5/2)--(1, -5/2);
				\draw[thick,rounded corners,color = Mulberry](1/2, -2)--(1/2, -3);
				\draw[thick,rounded corners,color = Mulberry](1, -5/2)--(2, -5/2);
				\draw[thick,rounded corners,color = Mulberry](3/2, -2)--(3/2, -3);
				\draw[thick,rounded corners,color = Mulberry](0, -7/2)--(1, -7/2);
				\draw[thick,rounded corners,color = Mulberry](1/2, -3)--(1/2, -4);
				\draw[thick,rounded corners,color = Mulberry](1, -7/2)--(2, -7/2);
				\draw[thick,rounded corners,color = Mulberry](3/2, -3)--(3/2, -4);
				\draw[thick,rounded corners,color = Mulberry](1/2, -4)--(1/2, -9/2)--(0, -9/2);
				\draw[thick,rounded corners,color = Mulberry](1, -9/2)--(1/2, -9/2)--(1/2, -5);
				\draw[thick,rounded corners,color = Mulberry](3/2, -4)--(3/2, -9/2)--(1, -9/2);
				\draw[thick,rounded corners,color = Mulberry](2, -9/2)--(3/2, -9/2)--(3/2, -5);
			\end{tikzpicture}
			\hspace{5em}
			\begin{tikzpicture}[x = 1.25em,y = 1.25em]
				\draw[step = 1,gray, very thin] (0,0) grid (2, -5);
				\draw[color = black, thick] (0,0) rectangle (2, -5);
				\draw[thick,rounded corners,color = Mulberry](1/2, 0)--(1/2, -1/2)--(0, -1/2);
				\draw[thick,rounded corners,color = Mulberry](1, -1/2)--(1/2, -1/2)--(1/2, -1);
				\draw[thick,rounded corners,color = Mulberry](3/2, 0)--(3/2, -1/2)--(1, -1/2);
				\draw[thick,rounded corners,color = Mulberry](2, -1/2)--(3/2, -1/2)--(3/2, -1);
				\draw[thick,rounded corners,color = Mulberry](0, -3/2)--(1, -3/2);
				\draw[thick,rounded corners,color = Mulberry](1/2, -1)--(1/2, -2);
				\draw[thick,rounded corners,color = Mulberry](1, -3/2)--(2, -3/2);
				\draw[thick,rounded corners,color = Mulberry](3/2, -1)--(3/2, -2);
				\draw[thick,rounded corners,color = Mulberry](0, -5/2)--(1, -5/2);
				\draw[thick,rounded corners,color = Mulberry](1/2, -2)--(1/2, -3);
				\draw[thick,rounded corners,color = Mulberry](1, -5/2)--(2, -5/2);
				\draw[thick,rounded corners,color = Mulberry](3/2, -2)--(3/2, -3);
				\draw[thick,rounded corners,color = Mulberry](0, -7/2)--(1, -7/2);
				\draw[thick,rounded corners,color = Mulberry](1/2, -3)--(1/2, -4);
				\draw[thick,rounded corners,color = Mulberry](1, -7/2)--(2, -7/2);
				\draw[thick,rounded corners,color = Mulberry](3/2, -3)--(3/2, -4);
				\draw[thick,rounded corners,color = Mulberry](0, -9/2)--(1, -9/2);
				\draw[thick,rounded corners,color = Mulberry](1/2, -4)--(1/2, -5);
				\draw[thick,rounded corners,color = Mulberry](3/2, -4)--(3/2, -9/2)--(1, -9/2);
				\draw[thick,rounded corners,color = Mulberry](2, -9/2)--(3/2, -9/2)--(3/2, -5);
			\end{tikzpicture}
			\quad
			\raisebox{3em}{$\mapsto$}
			\quad
			\begin{tikzpicture}[x = 1.25em,y = 1.25em]
				\draw[step = 1,gray, very thin] (0,0) grid (2, -5);
				\draw[color = black, thick] (0,0) rectangle (2, -5);
				\draw[thick,rounded corners,color = Mulberry](1/2, 0)--(1/2, -1/2)--(0, -1/2);
				\draw[thick,rounded corners,color = Mulberry](1, -1/2)--(1/2, -1/2)--(1/2, -1);
				\draw[thick,rounded corners,color = Mulberry](1, -1/2)--(2, -1/2);
				\draw[thick,rounded corners,color = Mulberry](3/2, 0)--(3/2, -1);
				\draw[thick,rounded corners,color = Mulberry](0, -3/2)--(1, -3/2);
				\draw[thick,rounded corners,color = Mulberry](1/2, -1)--(1/2, -2);
				\draw[thick,rounded corners,color = Mulberry](1, -3/2)--(2, -3/2);
				\draw[thick,rounded corners,color = Mulberry](3/2, -1)--(3/2, -2);
				\draw[thick,rounded corners,color = Mulberry](0, -5/2)--(1, -5/2);
				\draw[thick,rounded corners,color = Mulberry](1/2, -2)--(1/2, -3);
				\draw[thick,rounded corners,color = Mulberry](1, -5/2)--(2, -5/2);
				\draw[thick,rounded corners,color = Mulberry](3/2, -2)--(3/2, -3);
				\draw[thick,rounded corners,color = Mulberry](0, -7/2)--(1, -7/2);
				\draw[thick,rounded corners,color = Mulberry](1/2, -3)--(1/2, -4);
				\draw[thick,rounded corners,color = Mulberry](1, -7/2)--(2, -7/2);
				\draw[thick,rounded corners,color = Mulberry](3/2, -3)--(3/2, -4);
				\draw[thick,rounded corners,color = Mulberry](0, -9/2)--(1, -9/2);
				\draw[thick,rounded corners,color = Mulberry](1/2, -4)--(1/2, -5);
				\draw[thick,rounded corners,color = Mulberry](3/2, -4)--(3/2, -9/2)--(1, -9/2);
				\draw[thick,rounded corners,color = Mulberry](2, -9/2)--(3/2, -9/2)--(3/2, -5);
			\end{tikzpicture}
			\caption{Pictured on the left is a ladder move and on the right a K-theoretic ladder move. }
			\label{figure:ladderandkladder}
		\end{figure}
				
		The \newword{bottom pipe dream} for $w \in S_n$ is the pipe dream
		\[\{(i,j): i \in [n-1] \text{ and } j \in [c_w(i)] \}.\] 
		The bottom pipe dream of $w$ is an element of $\pipes(w)$. 
		Furthermore, $\mathcal{P} \subseteq \mathbb{T}_n$ is the bottom pipe dream of some permutation if and only if its crossings are left-justified; that is, $(i,j) \in \mathcal{P}$ with $j > 1$ implies $(i,j-1) \in \mathcal{P}$. In particular, the map sending a permutation to its Lehmer code defines a bijection between permutations and left-justified pipe dreams.
		
		Bergeron and Billey \cite{Bergeron.Billey} showed the following:
		\begin{lemma}
			Let $\mathcal{P} \in \pipes(w)$. 
			\begin{enumerate}
				\item If $\mathcal{P} \mapsto \mathcal{Q}$ is a ladder move, then $\mathcal{Q} \in \pipes(w)$.
				\item $\mathcal{P}$ can be obtained from the bottom pipe dream of $w$ by a sequence of ladder moves.
			\end{enumerate} 
		\end{lemma}
		
		We will make use of the following K-theoretic generalization. Although this result is known to experts, we include a proof for completeness.
		
		\begin{lemma}
			\label{lemma:kladdermoves}
			Let $\mathcal{P} \in \kpipes(w)$. 
			\begin{enumerate}
				\item If $\mathcal{P} \mapsto \mathcal{Q}$ is a ladder move or a K-theoretic ladder move, then $\mathcal{Q} \in \kpipes(w)$.
				\item The pipe dream $\mathcal{P}$ can be obtained from the bottom pipe dream of $w$ by a sequence of ladder moves and K-theoretic ladder moves.
			\end{enumerate}
		\end{lemma}
		\begin{proof}
			\noindent (1) Let $w \in S_n$ and $\mathcal{P} \in \kpipes(w)$. Suppose $\mathcal{P} \mapsto \mathcal{Q}$ is a ladder move or K-theoretic ladder move. A (K-theoretic) ladder move may be obtained by applying a sequence of the moves described in \cref{lemma:planarmoves}. Thus $w = \demprod(\mathcal{P}) = \demprod(\mathcal{Q})$, and therefore $\mathcal{Q} \in \kpipes(w)$.
						
			\noindent (2) Let $\mathcal{P} \in \kpipes(w)$. 
			If the crossings of $\mathcal{P}$ are left-justified, then $\mathcal{P}$ must be the bottom pipe dream of $w$. 
			Otherwise, there exists some $(i, j) \in \mathcal{P}$ with $j > 1$ such that $(i, j - 1)\notin \mathcal{P}$. 
			Choose $i$ to be maximal with this property. Now let $i' > i$ be the smallest index such that $(i', j)\notin \mathcal{P}$. 
			Then we have $[i + 1, i'- 1] \times [j - 1, j] \subseteq \mathcal{P}$. 
			If $(i',j-1) \in \mathcal{P}$, we may apply an inverse K-theoretic ladder move to $\mathcal{P}$ by removing the crossing at position $(i, j)$. 
			If instead $(i', j - 1) \notin \mathcal{P}$, then we may apply an inverse ladder move to $\mathcal{P}$ by removing the crossing at $(i, j)$ and adding a crossing at $(i', j - 1)$.
			
			We may iterate this process until we reach a pipe dream that is left-justified. At each step, we either delete a crossing or move a crossing down and to the left. Because $\kpipes_n$ is finite, the process terminates. By the first part, each intermediate pipe dream lies in $\kpipes(w)$. Since the final pipe dream is left-justified, it must be the bottom pipe dream of $w$.
		\end{proof}
		
		\subsection{Dominant parts of diagrams}
		
		A \newword{partition} is a tuple of positive integers $\lambda = (\lambda_1,\ldots,\lambda_k)$ with $\lambda_1 \geq \lambda_2 \geq \cdots \geq \lambda_k$. 
		Given a partition $\lambda$, we define the \newword{Young diagram} of $\lambda$ to be $\partitionsubset_{\lambda} = \{(i,j) : j \in [\lambda_i], i \in [k]\}$. 
		An \newword{addable cell} of $\partitionsubset_{\lambda}$ is a cell $(i,j)$ that satisfies all of the following conditions:
		\begin{enumerate}
			\item $(i,j)\notin \partitionsubset_{\lambda}$,
			\item if $i > 1$, then $(i - 1, j) \in \partitionsubset_{\lambda}$, and 
			\item if $j > 1$, then $(i, j - 1) \in \partitionsubset_{\lambda}$.
		\end{enumerate}
		Intuitively, $(i,j)$ is an addable cell of $\partitionsubset_{\lambda}$ if adding it to the Young diagram of $\lambda$ produces the Young diagram corresponding to a larger partition $\lambda'$.

		Given $U \subseteq [n] \times [n]$, the \newword{dominant part} of $U$, denoted $\dom(U)$, is the largest set of the form $\partitionsubset_{\lambda}$ such that $\partitionsubset_{\lambda} \subseteq U$ (under the containment order). Define $\dom(w) = \dom(\rothe(w))$.
		
		\begin{example}
			Let $w = 52413$ as in \cref{example:rothe}.
			\[\begin{tikzpicture}[x = 1.5em,y = 1.5em]
				\draw[fill = gray,semitransparent](0,5) rectangle (1,4);
				\draw[fill = gray,semitransparent](1,5) rectangle (2,4);
				\draw[fill = gray,semitransparent](2,5) rectangle (3,4);
				\draw[fill = gray,semitransparent](3,5) rectangle (4,4);
				\draw[fill = gray,semitransparent](0,4) rectangle (1,3);
				\draw[fill = gray,semitransparent](0,3) rectangle (1,2);
				\draw[step = 1,gray, very thin] (0,0) grid (5,5);
				\draw[color = black, thick](0,0)rectangle(5,5);
				\filldraw [black](2.5,.5)circle(.1);
				\filldraw [black](.5,1.5)circle(.1);
				\filldraw [black](3.5,2.5)circle(.1);
				\filldraw [black](1.5,3.5)circle(.1);
				\filldraw [black](4.5,4.5)circle(.1);
				
				\draw[thick, color = blue] (2.5,0)--(2.5,.5)--(5,.5);
				\draw[thick, color = blue] (.5,0)--(.5,1.5)--(5,1.5);
				\draw[thick, color = blue] (3.5,0)--(3.5,2.5)--(5,2.5);
				\draw[thick, color = blue] (1.5,0)--(1.5,3.5)--(5,3.5);
				\draw[thick, color = blue] (4.5,0)--(4.5,4.5)--(5,4.5); 
			\end{tikzpicture} \]
			We have shaded the cells in $\dom(w) = \partitionsubset_{(4,1,1)}$. The Young diagram $\partitionsubset_{(4,1,1)}$ has three addable cells, $(1,5), (2,2),$ and $(4,1)$. Note that the addable cells are of the form $(i,w(i))$ with $i\in\{1,2,4\}$.
		\end{example}

		In the next lemma, we show that $\dom(\mathcal{P})$ is invariant among elements $\mathcal{P} \in\kpipes(w)$.
		
		\begin{lemma}
			\label{lemma:dominantpartofpipedreams}
			Let $w \in S_{+} $. 
			The following statements hold.
			\begin{enumerate}
				\item If $(i,j)$ is an addable cell of $\dom(w)$, then $j = w(i)$.
				\item $\dom(w) = \dom(\mathcal{P})$ for all $\mathcal{P} \in \kpipes(w)$.
				\item If $(i,j)$ is an addable cell of $\dom(w)$, then $(i,j)\notin \mathcal{P}$ for all $\mathcal{P} \in \kpipes(w)$.
			\end{enumerate}
		\end{lemma}
		\begin{proof}
			
			\noindent (1) Suppose $(i,j)$ is an addable cell of $\dom(w)$. Because $(i,j)\notin \dom(w)$, but is adjacent to $\dom(w)$, we must have $(i,j)\notin \rothe(w)$. 
			Thus, in the Rothe diagram, the cell $(i,j)$ must have a dot weakly above it or weakly to its left. Since the Rothe diagram is formed by removing all cells weakly to the right and below each dot in the graph of $w$, such a dot must occur at $(i, w(i))$, and hence $(i, j) = (i, w(i))$.
			
			\noindent (2) Note first, if $\mathcal{P} \mapsto \mathcal{Q}$ is a ladder move or K-theoretic ladder move, then $\dom(\mathcal{P}) = \dom(\mathcal{Q})$. By \cref{lemma:kladdermoves}, any $\mathcal{P} \in \kpipes(w)$ can be reached from the bottom pipe dream by such moves, so it suffices to verify the claim for the bottom pipe dream.
			
			If $\mathcal{P}$ is the bottom pipe dream of $w$, then by construction we have $\dom(w) \subseteq \dom(\mathcal{P})$. From Part (1) and the fact that for every $(i, j) = (i, w(i))$ that is an addable cell of $\rothe(w)$, we have $(i,w(i))$ is not a crossing in the bottom pipe dream of $w$, we conclude that $\dom(\mathcal{P})$ cannot contain any partition shape larger than $\dom(w)$. Thus, $\dom(\mathcal{P}) = \dom(w)$.
			
			\noindent (3) Let $\mathcal{P} \in \kpipes(w)$ and suppose for contradiction that $(i, j) \in \mathcal{P}$. 
			From Part (2), we know that $\dom(\mathcal{P}) = \dom(w)$, and that $(i, j)$ is an addable cell for this partition.
			
			By \cref{lemma:kladdermoves}, there exists a sequence of inverse (K-theoretic) ladder moves taking $\mathcal{P}$ to the bottom pipe dream of $w$. 
			Since $(i, j)$ is not present in the bottom pipe dream, there must exist an intermediate pipe dream $\mathcal{P}' \in \kpipes(w)$ such that $(i, j) \in \mathcal{P}'$ but $(i, j - 1) \notin \mathcal{P}'$. 
			
			But then $(i, j)$ would not be an addable cell for $\dom(\mathcal{P}') = \dom(w)$, contradicting our assumption.
		\end{proof}
		
		The next lemma follows readily from the interpretation of pipe dreams as subwords of an ambient word. We refer the reader to \cite{Knutson.Miller:subword} for background.
		
		\begin{lemma}
			\label{lemma:subpipedream}
			Suppose $w \in S_{+} $ and let $\mathcal{P} \in \kpipes(w)$.
			\begin{enumerate}
				\item If $u \in S_{+} $ with $u\leq w$, then there exists $\mathcal{Q} \in \kpipes(u)$ such that $\mathcal{Q} \subseteq \mathcal{P}$.
				\item If $\mathcal{R} \in \kpipes(v)$ and $\mathcal{P} \subseteq \mathcal{R}$, then $w\leq v$.
			\end{enumerate}
		\end{lemma}
		\begin{proof}
			\noindent (1) This follows directly from \cite[Lemma 3.4]{Knutson.Miller:subword}.
			
			\noindent (2) By \cite[Lemma 3.4]{Knutson.Miller:subword}, there exists a reduced pipe dream $\mathcal{P}' \in \pipes(w)$ such that $\mathcal{P}' \subseteq \mathcal{P}$. 
			Again, applying \cite[Lemma 3.4]{Knutson.Miller:subword}, because $\mathcal{P}' \subseteq \mathcal{R}$, it follows that $w\leq v$.
		\end{proof}
				
		\subsection{Statements of the co-transition recurrences}
		
		In this section we state the algebraic and combinatorial co-transition recurrences of Knutson \cite{Knutson:cotransition}.
		Given $w \in S_{+} $ and a cell $(i,j) \in \mathbb{Z}_{+} \times \mathbb{Z}_{+}$, we say that $(a,b)$ is a \newword{neighbor} of $(i,j)$ in $w$ if all of the following conditions hold:
		\begin{enumerate}
			\item $b = w(a)$,
			\item $i \leq a$ and $j \leq b$,
			\item $(i,j) \neq (a,b)$, and
			\item for all $(a',b') \in [i,a] \times [j,b] - \{(i,j),(a,b)\}$, we have $b' \neq w(a')$.
		\end{enumerate}
		Given $w \in S_{+} $, we define
		\begin{equation}
			\cotransitionindices_i(w) = \{j : (j, w(j)) \text{ is a neighbor of } (i, w(i))\},
		\end{equation}
		and
		\begin{equation}
			\label{eq:cotransitionpermsdef}
			\cotransitionperms_i(w) = \{wt_{i,j} \in S_{+} : j \in\cotransitionindices_i(w)\}.
		\end{equation}
		It follows directly from the definition of neighbor and \cref{eq:cotransitionpermsdef} that 
		\[\cotransitionperms_i(w) = \{wt_{i,j} \in S_{+} : \ell(w t_{i,j}) = \ell(w) + 1\}.\]
		We also define 
		\[\kcotransitionperms_i(w) = \{wt_{i,i_k}t_{i,i_{k-1}}\cdots t_{i,i_1}:\emptyset \neq \{i_1 < i_2 < \cdots < i_k\} \subseteq \cotransitionindices_i(w)\}.\]
		
		Observe that each nonempty subset of $\cotransitionindices_i(w)$ uniquely determines an element of $\kcotransitionperms_i(w)$. 
		For convenience, given $U \subseteq \cotransitionindices_i(w)$ with $U = \{i_1 < i_2 < \cdots < i_k\}$, we define 
		$w_{U,i} = wt_{i,i_k} t_{i,i_{k-1}}\cdots t_{i,i_1}$. We also set $w_{\emptyset,i} = w$. It is immediate from the definitions that
		\begin{equation}
		\{w\}\cup \kcotransitionperms_i(w) = \{w_{U,i} : U \subseteq \cotransitionindices_{i}(w)\}.
		\end{equation}
		Note also that $\ell(w_{U,i})=\ell(w)+\#U$.
		
		We now recall the combinatorial co-transition recurrence on pipe dreams from \cite{Knutson:cotransition}. 
		We state both the cohomological and K-theoretic versions. 
		Knutson proved the cohomological statement and sketched a proof of the K-theoretic recurrence \cite{Knutson:cotransition}.
		We include additional details for completeness.
		
		\begin{proposition}
			\label{prop:kcotransitionpipefacts}
			Let $w \in S_+$ and suppose $(i,j)$ is an addable cell of $\dom(w)$.
			\begin{enumerate}
				\item The map $\mathcal{P} \mapsto \mathcal{P} \cup \{(i,j)\}$ defines a bijection from $\kpipes(w)$ to $\displaystyle \bigsqcup_{u \in \kcotransitionperms_i(w)} \kpipes(u)$.
				\item The map $\mathcal{P} \mapsto \mathcal{P} \cup \{(i,j)\}$ defines a bijection from $\pipes(w)$ to $\displaystyle \bigsqcup_{u \in\cotransitionperms_i(w)} \pipes(u)$.
			\end{enumerate}
		\end{proposition}
		We will prove \cref{prop:kcotransitionpipefacts} in \cref{subsection:cotransitionproof}.
		As an immediate consequence, we obtain the algebraic co-transition recurrences for Grothendieck and Schubert polynomials.
	
	\begin{proposition}
		\label{prop:kcotransition}
		Let $w \in S_+$. Suppose $(i,j)$ is an addable cell of $\dom(w)$.
		\begin{enumerate}
			\item $\displaystyle x_i\mathfrak{G}_w = \sum_{u \in \kcotransitionperms_i(w)} (-1)^{\ell(u) - \ell(w) + 1}\mathfrak{G}_u$.
			\item $x_i \displaystyle \mathfrak{S}_w = \sum_{u \in \cotransitionperms_i(w)} \mathfrak{S}_{u}$.
		\end{enumerate}
	\end{proposition}
	\begin{proof}
		This is immediate from \cref{theorem:pipesschubertandgrothendieck} and \cref{prop:kcotransitionpipefacts}. 
	\end{proof}
	
	We note that the first part \cref{prop:kcotransition} is a special case of \cite[Theorem 3.1]{Lenart.2003}. 
	The second part can be derived from Monk's rule. 
	
	\subsection{Proof of combinatorial co-transition recurrences}
	\label{subsection:cotransitionproof}
		
		We now state some technical lemmas, with the goal of establishing the K-theoretic co-transition recurrence for pipe dreams.
		
		We start with the following lemma, which gives a concrete method for constructing elements of $\kcotransitionperms_i(w)$ from the permutation matrix of $w$. See also \cref{figure:replacements}.
		\begin{lemma}
			\label{lemma:kcotransitionreplacements}
			Let $w \in S_{+} $.
			Let $U \subseteq \cotransitionindices_i(w)$ be nonempty and write $U = \{i_1 < i_2 < \cdots < i_k\}$. For convenience, set $i_0 = i$.
			The permutation matrix for $w_{U,i}$ is obtained from the permutation matrix of $w$ by making the following replacements:
			\begin{itemize}
				\item set $(i_h,w(i_h))$ to $0$ for all $h \in[0,k]$,
				\item set $(i_{h-1},w(i_h))$ to $1$ for all $h \in[k]$, and
				\item set $(i_k,w(i_0))$ to $1$.
			\end{itemize}
		\end{lemma}
		\begin{proof}
			This follows directly from the definition of $w_{U,i}$.
		\end{proof}
		
		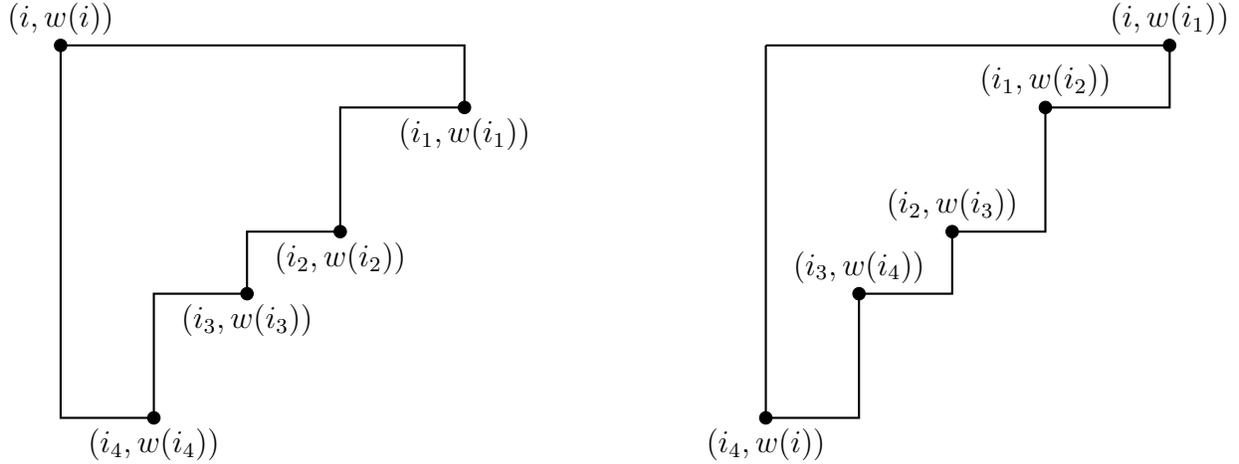
\begin{figure}
			\[\begin{tikzpicture}[x = 2em,y = 2em]
				\draw[thick] (-0.5,7)--(6,7)--(6,6)--(4,6)--(4,4)--(2.5,4)--(2.5,3)--(1,3)--(1,1)--(-0.5,1)--(-0.5,7);
				\filldraw [black](-0.5,7)circle(.1);
				\filldraw [black](6,6)circle(.1);
				\filldraw [black](4,4)circle(.1);
				\filldraw [black](2.5,3)circle(.1);
				\filldraw [black](1,1)circle(.1);
				\node [above] at (-0.5, 7) {$(i,w(i))$}; 
				\node [below] at (6, 6) {$(i_1,w(i_1))$}; 
				\node [below]at (4, 4) {$(i_2,w(i_2))$}; 
				\node [below]at (2.5, 3) {$(i_3,w(i_3))$}; 
				\node [below]at (1, 1) {$(i_4,w(i_4))$}; 
			\end{tikzpicture}
			\hspace{5em}
			\begin{tikzpicture}[x = 2em,y = 2em]
				\draw[thick] (-0.5,7)--(6,7)--(6,6)--(4,6)--(4,4)--(2.5,4)--(2.5,3)--(1,3)--(1,1)--(-0.5,1)--(-0.5,7);
				\filldraw [black](6,7)circle(.1);
				\filldraw [black](4,6)circle(.1);
				\filldraw [black](2.5,4)circle(.1);
				\filldraw [black](1,3)circle(.1);
				\filldraw [black](-0.5,1)circle(.1);
				\node [above] at (6, 7) {$(i,w(i_1))$}; 
				\node [above] at (4, 6) {$(i_1,w(i_2))$}; 
				\node [above]at (2.5, 4) {$(i_2,w(i_3))$}; 
				\node [above]at (1, 3) {$(i_3,w(i_4))$}; 
				\node [below] at (-0.5, 1) {$(i_4,w(i))$}; 
			\end{tikzpicture}
			\]
			\caption{Visualization of the replacements in \cref{lemma:kcotransitionreplacements}, when $k=4$. 
				A portion of the graph of $w$ is on the left, and the graph of $w_{U,i}$ is on the right, with $U=\{i_1,i_2,i_3,i_4\}$.
				Note that no dots appear in the interior of the pictured partition shapes.}
			\label{figure:replacements}
		\end{figure}

		\begin{lemma}
			\label{lemma:booleanranks}
			Let $w \in S_+$ and fix $U = \{i_1 < i_2 < \cdots < i_k\} \subseteq \cotransitionindices_i(w)$ with $U \neq \emptyset$. 
			Then $\rk_w(a,b)-\rk_{w_{U,i}}(a,b) = 1$ if \[(a,b) \in \bigcup_{h = 1}^k[i,i_h-1] \times [w(i),w(i_h)-1]\] and $0$ otherwise.
		\end{lemma}
		\begin{proof}
			This follows from \cref{lemma:kcotransitionreplacements} together with the definition of the rank function of a permutation.
		\end{proof}

		The \newword{Boolean lattice} $B_n$ is the poset whose elements are the subsets of $[n]$, ordered by inclusion. 
		As the name suggests, $B_n$ is also a lattice; join is given by taking unions of subsets, and meet by taking intersections. 
		If a lattice is isomorphic to $B_n$, for some $n$, we also refer to it as a Boolean lattice.

		Given $u,v \in S_{+} $, let $[u,v]_{S_{+} } = \{w \in S_{+} :u\leq w\leq v\}$. 
		We now describe certain intervals in Bruhat order that govern the K-theoretic co-transition recurrences, and show that they are Boolean lattices.
		
		\begin{lemma}
			\label{lemma:booleaninterval}
			Let $w \in S_{+} $.
			\begin{enumerate}
				\item Given $U,U'\subseteq \cotransitionindices_{i}(w)$, we have $w_{U,i}\leq w_{U',i}$ if and only if $U\subseteq U'$. As such, the set $\{w\} \cup \kcotransitionperms_i(w)$, considered as a subposet of the Bruhat order on $S_{+}$, is a Boolean lattice. 				
				\item $[w,\vee \cotransitionperms_i(w)]_{S_{+}} = \{w\} \cup \kcotransitionperms_i(w)$.
			\end{enumerate} 
		\end{lemma}
		
		\begin{proof}
			\noindent (1) This follows directly from \cref{lemma:rankbruhat} and \cref{lemma:booleanranks}. 
			
			\noindent (2) Let $v = \vee \cotransitionperms_i(w)$. Suppose that $u \in S_{+} $ with $w\leq u \leq v$. If $u \in \{w\} \cup \kcotransitionperms_i(w)$ we are done, so assume not. 
			
			Choose $n$ large enough so that $\{w\} \cup \kcotransitionperms_i(w) \subseteq S_n$.
			By \cref{lemma:rankbruhat}, we have
			\[\rk_w(a,b)\geq \rk_u(a,b)\geq \rk_v(a,b) \text{ for all } (a,b) \in [0,n] \times [0,n].\]
			Thus,
			\begin{equation}
				\label{eq:rankinequalities}
				0\leq \rk_w(a,b)-\rk_u(a,b)\leq \rk_w(a,b)-\rk_v(a,b) \text{ for all } (a,b) \in [0,n] \times [0,n].
			\end{equation}
			By \cref{lemma:booleanranks}, $\rk_w(a,b)-\rk_v(a,b)$ equals $1$ if $(a,b) \in\bigcup_{h \in \cotransitionindices_i(w)}[i,i_h-1] \times [w(i),w(i_h)-1]$ and is $0$ otherwise.
			
			Now let $U \subseteq \cotransitionindices_{i}(w)$ be a subset of maximal size such that $w_{U,i}\leq u$. 
			Because $w_{U,i}\leq u$, there exists indices $j < k$ so that $w_{U,i} t_{j,k}\leq u$ and $\ell(w_{U,i} t_{j,k}) = \ell(w_{U,i}) + 1$. Because of \cref{eq:rankinequalities}, we must have \[(j,w_{U,i}(j)) \in\bigcup_{h \in \cotransitionindices_i(w)}[i,i_h-1] \times [w(i),w(i_h)-1].\]
			
			But then by \cref{lemma:kcotransitionreplacements}, we must be in the situation where $w_{U,i} t_{j,k} = w_{U \cup \{k\},i}$ with $j \in U$ and $k \in\cotransitionindices_i(w)-U$, which contradicts the maximality of the set $U$. Hence, no such \( u \notin \{w\} \cup \kcotransitionperms_i(w) \) can exist in the interval $[w,\vee \cotransitionperms_i(w)]_{S_{+}}$, and the claim follows.
		\end{proof}
		
		The next lemma describes how removing a cross from the dominant part of a pipe dream modifies the corresponding permutations.
		\begin{lemma}
			\label{lemma:dominantcrossremoval}
			Let $u \in S_+$ and $(i,j) \in \dom(u)$. 
			Let $i < i_1 < i_2 < \cdots < i_k = u^{-1}(j)$ be the row indices of the neighbors of $(i,j)$ in $u$. 
			Let $w = ut_{i,i_1}t_{i,i_2}\cdots t_{i,i_k}$. 
			The following hold:
			\begin{enumerate}
				\item $u \in \kcotransitionperms_{i}(w)$.
				\item If $\mathcal{Q} \in \kpipes(u)$, then $\mathcal{Q}-\{(i,j)\} \in \kpipes(w)$.
			\end{enumerate}
		\end{lemma}
		
		\begin{proof}
			\noindent (1) By construction, $w(i) = j$. Thus, $(i,j)\notin \dom(w)$. 
			In particular, $(i,j)$ is a northwest most dot in the graph of $w$, hence $(i,j)$ is an addable cell for $\dom(w)$. Furthermore, $u=w_{U,i}$ where $U=\{i_1,i_2,\ldots,i_k\}$, thus $u\in \kcotransitionperms_{i}(w)$.
			
			\noindent (2)
			We have $w\leq u$ and hence, by \cref{lemma:subpipedream}, there exists $\mathcal{P} \in \kpipes(w)$ such that $\mathcal{P} \subseteq \mathcal{Q}$. 
			Because $(i,j)$ is an addable cell for $\dom(w)$, by \cref{lemma:dominantpartofpipedreams}, $(i,j)\notin \mathcal{P}$. 
			Therefore, $\mathcal{P} \subseteq \mathcal{Q}-\{(i,j)\}$. 
			Applying \cref{lemma:subpipedream}, we see that $\mathcal{Q}-\{(i,j)\} \in \kpipes(v)$ for some $v \in S_n$ with $w\leq v\leq u$. 
			
			By \cref{lemma:booleaninterval}, we have $v \in \{w\} \cup \kcotransitionperms_{i}(w)$. 
			By \cref{lemma:dominantpartofpipedreams}, because $(i,j)\notin \dom(\mathcal{Q})$, it follows that $(i,j)\notin \dom(v)$. 
			Because $(i,j) \in \dom(\pi)$ for all $\pi \in \kcotransitionperms_{i}(w)$, we have $v = w$.
		\end{proof}

		We now are ready to prove the co-transition recurrences on pipe dreams.		
		
		\begin{proof}[Proof of {\cref{prop:kcotransitionpipefacts}}]
			\noindent (1) Let $\mathcal{P} \in \kpipes(w)$. 
			Since $(i,j)$ is an addable cell of $\dom(w)$, by \cref{lemma:dominantpartofpipedreams}, we know that $(i,j)\notin \mathcal{P}$. 
			Hence, the given map is injective. 
			
			Let $\mathcal{Q} = \mathcal{P} \cup \{(i,j)\}$. 
			We have $\mathcal{Q} \in \kpipes(u)$ for some $u \in S_+$. 
			Furthermore, by \cref{lemma:dominantpartofpipedreams}, $\dom(\mathcal{P})=\dom(w)$ and hence $(i,j)$ is an addable cell of $\dom(\mathcal{P})$. Therefore, $(i,j) \in \dom(\mathcal{Q}) = \dom(u)$. 
			By \cref{lemma:dominantcrossremoval}, since $\mathcal{P} = \mathcal{Q}-\{(i,j)\} \in \kpipes(w)$, we have $u \in \kcotransitionperms_i(w)$.
			Hence, the map is well-defined.

			To show surjectivity, take some $u \in \kcotransitionperms_{i}(w)$ and let $\mathcal{Q} \in \kpipes(u)$. We can write $u=w_{U,i}$ for some $U\subseteq \cotransitionindices_i(w)$ with $U\neq \emptyset$.
			Let $i < i_1 < i_2 < \cdots < i_k = u^{-1}(j)$ be the row indices of the neighbors of $(i,j)$ in $u$. 
			By construction, and the definition of neighbors, it must be that $U=\{i_1,i_2,\ldots,i_k\}$.
			Furthermore, $w=ut_{i,i_1}t_{i,i_2}\cdots t_{i,i_k}$. Thus, by \cref{lemma:dominantcrossremoval}, we have $\mathcal{P} \in \kpipes(w)$.
						
			\noindent (2) If $\pi \in S_n$ and $\mathcal{P} \in \pipes(\pi)$, then $\#\mathcal{P} = \ell(\pi)$. 
			Furthermore, if $\mathcal{Q} \in \kpipes(\pi)-\pipes(\pi)$, then $\#\mathcal{Q} > \ell(u)$. 
			
			Now take $\mathcal{P}\in \pipes(w)$.
			We have 
			\[\#(\mathcal{P} \cup \{(i,j)\}) = \min\{\ell(u):u \in \kcotransitionperms_i(w)\}.\] 
			In particular, this minimum is achieved exactly when $u \in \cotransitionperms_i(w)$. 
			Thus, restricting the bijection from Part (1) gives a bijection from $\pipes(w)$ to $\bigsqcup_{u \in \cotransitionperms_i(w)} \pipes(u)$.
		\end{proof}

		\section{Changing bases with pipe dreams}
		
		\label{section:pipechangebasis}
		
		In this section, we give new proofs of Lenart's rule \cite{Lenart} for expanding Grothendieck polynomials as sums of Schubert polynomials, as well as Lascoux's rule \cite{Lascoux.03} for expanding Schubert polynomials as sums of Grothendieck polynomials. 
		We have rephrased both rules in terms of pipe dreams and co-pipe dreams; up to reindexing the underlying triangular arrays, the statements are equivalent. 
		Lenart and Lascoux used noncommutative Schubert calculus in their arguments. 
		We give new proofs of both rules using the combinatorial co-transition recurrence on pipe dreams. 
		The supporting framework we develop for these proofs also yields a recursive method to generate the set of pipe dreams for $w\in S_n$ whose co-pipe dreams are reduced (see \cref{section:chains}).
		
		\subsection{Technical lemmas}

		Given $w \in S_n$, we will often consider the leftmost column $j$ such that there exists an addable cell $(i,j)$ of $\dom(w)$ with $i + j \leq n$. If such a column exists, we define $\addable(w) = j$. If it does not, then $\addable(w)$ is undefined. In this case, we must have that $w = w_0$.

		\begin{lemma} 
			\label{lemma:kcotransitiondescents}
			Let $w \in S_{n} $ with $w \neq w_0$. Suppose $(i,j)$ is an addable cell of $\dom(w)$ with $j = \addable(w)$.
			The following statements hold:
			\begin{enumerate}
				\item $w < ws_i$.
				\item $i + 1 \in \cotransitionindices_i(w)$.
				\item $\cotransitionindices_i(w)-\{i + 1\} \subseteq \cotransitionindices_{i + 1}(ws_i)$.
				\item $w_{U,i} \in \kcotransitionperms_i(w)$ has a descent at $i$ if and only if $w_{U,i}\geq ws_i$, if and only if $i + 1 \in U$.
				\item $(ws_i)_{U,i+1} \in \kcotransitionperms_{i + 1}(ws_i)$ has a descent at $i$ if and only if $U \subseteq \cotransitionindices_i(w)$. 
				In this case, $(ws_i)_{U,i + 1}=w_{U \cup \{i + 1\},i}$.
			\end{enumerate}
		\end{lemma}
		
		\begin{proof}
			
			\noindent (1) Because $j = \addable(w)$, we must have $w(i) < w(i + 1)$. 
			If not, then \[w(i + 1) + i + 1\leq w(i) + i\leq n,\] which contradicts the assumption that $j = \addable(w)$. 
			Hence, $w < ws_i$. 
						
			\noindent (2) Because $(i,w(i))$ and $(i + 1,w(i + 1))$ lie in adjacent rows in the graph of $w$, and $w(i) < w(i + 1)$, it follows that $i + 1 \in \cotransitionindices_i(w)$.
						
			\noindent (3) The graphs of $w$ and $ws_i$ differ only by exchanging rows $i$ and $i + 1$. 
			Thus if $j \neq i + 1$ is a neighbor of $(i,w(i))$ in the graph of $w$, then it will still be a neighbor of \[(i + 1,w(i))=(i + 1,(ws_i)(i + 1))\] in the graph of $ws_i$.
			
			\noindent (4) Note that $ws_i=w_{\{i+1\},i}$.  Thus, this statement follows from \cref{lemma:kcotransitionreplacements} and \cref{lemma:booleaninterval}.
			
			\noindent (5) This also follows from \cref{lemma:kcotransitionreplacements} and \cref{lemma:booleaninterval}.
		\end{proof}
		
		The next lemma will be used to aid with divided difference computations in our proofs of \cref{thm:pipe_groth_to_schub} and \cref{thm:pipe_schub_to_groth}.
		
		\begin{lemma}
			\label{lemma:whathappenswithwsianddescents}
			Let $w \in S_n$ with $w \neq w_0$. Suppose that $(i,j)$ is an addable cell of $\dom(w)$ with $j = \addable(w)$. Then:
			\begin{enumerate}
				\item For all $u \in \kcotransitionperms_i(w)$, we have $u \in S_n$.
				\item The cell $(i + 1,j)$ is an addable cell of $\dom(ws_i)$. 
				\item $\{u \in \kcotransitionperms_{i + 1}(ws_i) : u \text{ has a descent at } i\} = \{u \in \kcotransitionperms_{i}(w): u > ws_i\}.$
				\item The map $u \mapsto us_i$ defines a bijection from 
				\[\{v \in \cotransitionperms_{i + 1}(ws_i) : v \text{ has a descent at } i\}\]
				to $\cotransitionperms_{i}(w)-\{ws_i\}$.
			\end{enumerate} 
		\end{lemma}
	
		\begin{proof}
			\noindent (1) Because $j=\addable(w)$, we have $i + j \leq n$.  Thus, $\mathcal{P} \cup \{(i,j)\} \subseteq \mathbb{T}_n$ for all $\mathcal{P} \in \kpipes(w)$. 
			In particular, this means that $\demprod(\mathcal{P}\cup\{(i,j)\})\in S_n$.
			Thus, the claim follows from \cref{prop:kcotransitionpipefacts}.

			\noindent (2) We have $\dom(ws_i) = \dom(w) \cup \{(i,j)\}$, and $(i + 1,j) = (i + 1,(ws_i)(i + 1))$. Since $j = \addable(w)$, if $j > 1$ then $(i + 1,j-1) \in \dom(w)$. Thus, $(i + 1,j)$ is an addable cell of $\dom(ws_i)$.
			
			\noindent (3) $(\subseteq)$ Let $U \subseteq \cotransitionindices_{i + 1}(ws_i)$ so that $(ws_i)_{U,i + 1} \in\{u \in \kcotransitionperms_{i + 1}(ws_i) : u \text{ has a descent at } i\}$. 
			By Part (5) of \cref{lemma:kcotransitiondescents}, it follows that $U \subseteq \cotransitionindices_i(w)$ and $(ws_i)_{U,i + 1}=w_{U \cup \{i + 1\},i}$. Furthermore, $U \neq \emptyset$.
			By Part (4) of \cref{lemma:kcotransitiondescents}, we have $w_{U \cup \{i + 1\},i} > w_{\{i + 1\},i}=ws_i$, thus $(ws_i)_{U,i + 1} \in \{u \in \kcotransitionperms_{i}(w): u > ws_i\}$.
			
			\noindent $(\supseteq)$ Let $U \subseteq \cotransitionindices_i(w)$ so that $w_{U,i} \in \{u \in \kcotransitionperms_{i}(w) : u > ws_i\}$. 
			By Part (4) of \cref{lemma:kcotransitiondescents}, it follows that $i + 1 \in U$ and $w_{U,i}$ has a descent at $i$. 
			From \cref{lemma:kcotransitionreplacements}, we know that $w_{U,i}=(ws_i)_{U',i + 1}$ where $U'=U-\{i + 1\}$. 
			By Part (3) of \cref{lemma:kcotransitiondescents}, we have $U' \subseteq \cotransitionindices_{i + 1}(ws_i)$, thus $w_{U,i} \in \{u \in \kcotransitionperms_{i + 1}(ws_i) : u \text{ has a descent at } i\}$.

			\noindent (4) First we will show that the map is well-defined. 
			Suppose $u \in \cotransitionperms_{i + 1}(ws_i)$ and that $u$ has a descent at $i$.
			By Part (5) of \cref{lemma:kcotransitiondescents}, we must have $u=(ws_i)_{\{j\},i + 1}$ for some $j \in \cotransitionindices_i(w)-\{i + 1\}$. Furthermore, $us_i=(ws_i)_{\{j\},i + 1}s_i=w_{\{j\},i}$ (this follows from \cref{lemma:kcotransitionreplacements}). Thus, $us_i \in \cotransitionperms_i(w)-\{ws_i\}$.
			
			Injectivity follows because the map $u \mapsto us_i$ is an involution on $S_{+} $, and hence injective.
			
			For surjectivity, suppose that $u \in \cotransitionperms_i(w)-\{ws_i\}$. 
			Then $u=w_{\{j\},i}$ for some $j \in \cotransitionindices_i(w)$ with $j \neq i + 1$. 
			Furthermore, $us_i=(ws_i)_{\{j\},i + 1}$.
			Because $\{j\} \subseteq \cotransitionindices_i(w)$, it follows from Part (5) of \cref{lemma:kcotransitiondescents} that 
			\[us_i \in \{v \in \cotransitionperms_{i + 1}(ws_i) : v \text{ has a descent at } i\}.\] 
			Since $us_i \mapsto us_is_i=u$, the map is surjective.
		\end{proof}
	
		The following lemma describes how the the co-transition recurrence modifies co-permutations of pipe dreams.
		
		\begin{lemma}
			\label{lemma:cotransitionflippermutation}
			Let $w \in S_n$ with $w \neq w_0$. 
			Suppose $(i,j)$ is an addable cell of $\dom(w)$ with $j = \addable(w)$. 
			Let $\mathcal{P} \in \kpipes(w)$ and define $\mathcal{Q} = \mathcal{P} \cup \{(i,j)\}$.
			Then $\demprod(\check{\mathcal{Q}})\square \tau_i = \demprod(\check{\mathcal{P}})$.
		\end{lemma}
		
		\begin{proof}
			Because we have chosen $(i,j)$ to be the leftmost addable cell of $\dom(w)$ with $i + j\leq n$, it follows that we obtain ${\mathbf{a}}_{\check{\mathcal{Q}}}$ from ${\mathbf{a}}_{\check{\mathcal{P}}}$ by removing the last entry. 
			Furthermore, this entry is $n - i$. This is because when we map $\mathcal{P}$ to $\check{\mathcal{P}}$, the bump at $(i,j)$ in $\mathcal{P}$ becomes a crossing in $\check{\mathcal{P}}$ at $(i+j,j)$. When we flip vertically to get the word $\mathbf{a}_{\check{\mathcal{P}}}$, the crossing in $(i+j,j)$ moves to $(n-(i+j)+1,j)$, which has associated letter $n-(i+j)+1+j-1=n-i$.
			Thus, $\demprod({\mathbf{a}}_{\check{\mathcal{P}}}) = \demprod({\mathbf{a}}_{\check{\mathcal{Q}}})*\tau_{n-i}$, and so,
			\begin{align*}
				\demprod(\check{\mathcal{P}})
				& = \demprod({\mathbf{a}}_{\check{\mathcal{P}}}) w_0 &\text{(by \cref{lemma:copipewordandpermutation})}\\
				& = (\demprod({\mathbf{a}}_{\check{\mathcal{Q}}})*\tau_{n-i}) w_0 \\
				& = (\demprod({\mathbf{a}}_{\check{\mathcal{Q}}})w_0)\square \tau_{i} & \text{(by \cref{lemma:upanddownoperatorswithflips})}\\
				& = \demprod(\check{\mathcal{Q}})\square \tau_i &\text{(by \cref{lemma:copipewordandpermutation})}
			\end{align*}
			as desired.
		\end{proof}
		
		The next lemma provides a recursive method to construct pipe dreams for $w$ with reduced co-pipe dreams.
		
		\begin{lemma}
			\label{lemma:kcotransitionrestrictedbijection}
			Let $w \in S_n$ with $w \neq w_0$. 
			Suppose that $(i,j)$ is an addable cell of $\dom(w)$ with $j = \addable(w)$.
			The map $\mathcal{P} \mapsto \mathcal{P} \cup \{(i,j)\}$ defines a bijection from $\{\mathcal{P} \in \kpipes(w) : \check{\mathcal{P}} \text{ is reduced}\}$ to 
			\begin{align}
				\label{eq:bigunionintwoforms}
				\bigsqcup_{\substack{u \in \kcotransitionperms_i(w)\\u\geq ws_i}} &\{\mathcal{Q} \in \kpipes(u):\check {\mathcal{Q}} \text{ is reduced and } \demprod(\check{\mathcal{Q}}) \text{ has a descent at } i\}\\
				& = \bigsqcup_{u \in \kcotransitionperms_i(w)} \{\mathcal{Q} \in \kpipes(u):\check {\mathcal{Q}} \text{ is reduced and } \demprod(\check{\mathcal{Q}}) \text{ has a descent at } i\}. \nonumber
			\end{align}
		\end{lemma}
	
		\begin{proof}			
			By \cref{lemma:copdascent}, if $w_{\check {\mathcal{Q}}}$ has a descent at $i$ and $\mathcal{Q} \in \kpipes(u)$, then $u$ must also have a descent at $i$. 
			By \cref{lemma:kcotransitiondescents}, a permutation $u \in \kcotransitionperms_i(w)$ has a descent at $i$ if and only if $u\geq ws_i$. Thus, given $u \in \kcotransitionperms_i(w)$ such that $u\not \geq ws_i$, we have 
			\[\{\mathcal{Q} \in \kpipes(u):\check {\mathcal{Q}} \text{ is reduced and } \demprod(\check{\mathcal{Q}}) \text{ has a descent at } i\} = \emptyset.\]
			This establishes the equality in \cref{eq:bigunionintwoforms}.
			
			We will now show the given map is bijective.
			Since this map is a restriction of the bijection from \cref{prop:kcotransitionpipefacts}, it follows that it is injective. 
			
			We now show the map is well-defined. 
			Let $\mathcal{P} \in \kpipes(w)$ such that $\check{\mathcal{P}}$ is reduced. 
			Define $\mathcal{Q} = \mathcal{P} \cup \{(i,j)\}$. 
			Because we have chosen $(i,j)$ to be the leftmost addable cell with $i + j\leq n$, it follows that the word
			${\mathbf{a}}_{\check{\mathcal{Q}}}$ is obtained from ${\mathbf{a}}_{\check{\mathcal{P}}}$ by removing the last entry. 
			Since $\check{\mathcal{P}}$ is reduced, the word ${\mathbf{a}}_{\check{\mathcal{P}}}$ is reduced. 
			Removing the final entry from ${\mathbf{a}}_{\check{\mathcal{P}}}$ yields a reduced word ${\mathbf{a}}_{\check{\mathcal{Q}}}$, and therefore $\check{\mathcal{Q}}$ is reduced.
			From \cref{lemma:cotransitionflippermutation}, we have $\demprod(\check{\mathcal{Q}}) s_i = \demprod(\check{\mathcal{Q}})\square \tau_i = \demprod(\check{\mathcal{P}})$. 
			Thus, $\demprod(\check{\mathcal{Q}})$ has a descent at $i$. 
			
			It remains to show that the map is surjective. 
			Take $u \in \kcotransitionperms_i(w)$ with $u\geq ws_i$, and let $\mathcal{Q} \in \kpipes(u)$ such that $\check{\mathcal{Q}}$ is reduced and $\demprod(\check{\mathcal{Q}})$ has a descent at $i$. 
			Define $\mathcal{P} = \mathcal{Q}-\{(i,j)\}$. 
			By \cref{prop:kcotransitionpipefacts}, we have $\mathcal{P} \in \kpipes(w)$.
			We must show that $\check{\mathcal{P}}$ is reduced. 
			Because $\check{\mathcal{Q}}$ is reduced, $\mathbf{a}_{\check{\mathcal{Q}}}$ is a reduced word. 
			Since $\demprod(\check{\mathcal{Q}})$ has a descent at $i$, we have $\demprod(\mathbf{a}_{\check{\mathcal{Q}}}) = \demprod(\check{\mathcal{Q}})w_0$ has an ascent at $n-i$. 
			Thus, when we obtain ${\mathbf{a}}_{\check{\mathcal{P}}}$ from ${\mathbf{a}}_{\check{\mathcal{Q}}}$ by appending the letter $n-i$, we still have a reduced word. 
			Therefore, $\check{\mathcal{P}}$ is reduced.
		\end{proof}
		
		\subsection{Proof of the pipe dream change of basis formulas}
		
		In this section, we prove the pipe dream change of basis formulas, \cref{thm:pipe_groth_to_schub} and \cref{thm:pipe_schub_to_groth}.
		
		We start with a lemma.
		\begin{lemma}
			\label{eq:grothorddivdiff}
			Let $w \in S_n$. If $w < ws_i$, then 
			\[\divdiff_i(\mathfrak{G}_{w}) = 0.\]
		\end{lemma}
		\begin{proof}
			We have $\mathfrak{G}_{w} = \kdivdiff_i(\mathfrak{G}_{ws_i}) = \divdiff_i\left((1 - x_{i + 1})\mathfrak{G}_{ws_i}\right)$, and because $\divdiff_i^2 = 0$, the result follows.
		\end{proof}
		
		We now prove the pipe dream formula for expanding Schubert polynomials as sums of Grothendieck polynomials.
		\thmpipeB*
		\begin{proof}
			We proceed by reverse induction on Bruhat order. 
			In the base case, let $w = w_0$.
			Then $\pipes(w_0)=\{\mathbb{T}_n\}$. Because $\check{\mathbb{T}}_n = \emptyset$, we have $\demprod(\check{\mathbb{T}_n})=w_0$. Because $\mathfrak{S}_{w_0}=\mathfrak{G}_{w_0}$, the formula holds.
			
			Now suppose $w \neq w_0$, and assume that the statement holds for all $u \in S_n$ such that $u > w$ in Bruhat order.

			Let $j = \addable(w)$ and $i = w^{-1}(j)$. 
			By \cref{lemma:kcotransitiondescents}, we have $ws_i > w$. 
			By \cref{lemma:whathappenswithwsianddescents}, $ws_i\in S_n$. Thus, $ws_i$ falls within the inductive hypothesis.

			Because $\divdiff_i=\divdiff_i x_{i+1}+\kdivdiff_i$, we obtain
			\begin{equation}
				\label{eqn:writingoutdivideddiffapplication}
				\mathfrak{S}_w=\divdiff_i(\mathfrak{S}_{ws_i}) = \divdiff_i(x_{i + 1}\mathfrak{S}_{ws_i})+\kdivdiff_i(\mathfrak{S}_{ws_i}).
			\end{equation}
		
			By the inductive hypothesis,
			\begin{equation}
				\label{eq:schubertexpansioninduction}
				\mathfrak{S}_{ws_i} = \sum_{\mathcal{P} \in \pipes(ws_i)} \mathfrak{G}_{\demprod(\check{\mathcal{P}})}.
			\end{equation}
			
			We now apply $\kdivdiff_i$ to both sides of \cref{eq:schubertexpansioninduction}, which, by \cref{eq:grothdivdiff}, yields
			\begin{equation}
				\label{eq:KNi}
				\kdivdiff_i(\mathfrak{S}_{ws_i}) = \sum_{\mathcal{P} \in \pipes(ws_i)} \mathfrak{G}_{\demprod(\check{\mathcal{P}})\square \tau_i}.
			\end{equation}

			By \cref{prop:kcotransition}, we have
			\[x_{i + 1}\mathfrak{S}_{ws_i} = \sum_{u \in\cotransitionperms_{i + 1}(ws_i)} \mathfrak{S}_{u}.\]
			Applying Part (4) of \cref{lemma:whathappenswithwsianddescents}, together with \cref{eq:shubdivdiff}, we obtain 
			\[\divdiff_i(x_{i + 1}\mathfrak{S}_{ws_i}) = \sum_{\substack{v \in \cotransitionperms_i(w)\\v \neq ws_i}}\mathfrak{S}_v.\]
			
			By \cref{lemma:whathappenswithwsianddescents}, since $j=\addable(w)$, each $u \in \cotransitionperms_i(w)$ is in $S_n$. 
			Furthermore, each such $u > w$, and thus falls within the inductive hypothesis. 
			Therefore, we can write
			\begin{align*}
				\divdiff_i(x_{i + 1}\mathfrak{S}_{ws_i})
				& = \sum_{\substack{u \in \cotransitionperms_i(w)\\u \neq ws_i}}\mathfrak{S}_{u}\\
				& = \sum_{\substack{u \in \cotransitionperms_i(w)\\u \neq ws_i}}\sum_{\mathcal{Q} \in \pipes(u)}\mathfrak{G}_{\demprod(\check{\mathcal{Q}})}.
			\end{align*}
			By Part (4) of \cref{lemma:whathappenswithwsianddescents}, each $u$ in the sum has an ascent at $i$. 
			By \cref{lemma:copdascent}, because $u$ has an ascent at $i$, it follows that $\demprod(\check{\mathcal{Q}})$ also has an ascent at $i$ for all $\mathcal{Q} \in \pipes(u)$. 
			Thus, $\demprod(\check{\mathcal{Q}}) = \demprod(\check{\mathcal{Q}})\square \tau_i$, and so
			\begin{equation}
				\label{eq:xtimesSwsi}
			\divdiff_i(x_{i + 1}\mathfrak{S}_{ws_i}) = \sum_{\substack{u \in \cotransitionperms_i(w)\\u \neq ws_i}}\sum_{\mathcal{Q} \in \pipes(u)}\mathfrak{G}_{\demprod(\check{\mathcal{Q}})\square \tau_i}.
			\end{equation}
			Hence, by substituting \cref{eq:xtimesSwsi} and \cref{eq:KNi} into \cref{eqn:writingoutdivideddiffapplication}, we obtain
			\begin{align*}
				\mathfrak{S}_{w}
				& = \sum_{\mathcal{Q} \in \pipes(ws_i)} \mathfrak{G}_{\demprod(\check{\mathcal{Q}})\square \tau_i} + \sum_{\substack{u \in \cotransitionperms_i(w)\\u \neq ws_i}}\sum_{\mathcal{Q} \in \pipes(u)}\mathfrak{G}_{\demprod(\check{\mathcal{Q}})\square \tau_i}\\
				& = \sum_{u \in \cotransitionperms_i(w)}\sum_{\mathcal{Q} \in \pipes(u)}\mathfrak{G}_{\demprod(\check{\mathcal{Q}})\square \tau_i}\\
				& = \sum_{\mathcal{P} \in \pipes(w)}\mathfrak{G}_{\demprod(\check{\mathcal{P}})} & (\text{by \cref{prop:kcotransitionpipefacts} and \cref{lemma:cotransitionflippermutation}}),
			\end{align*}
			as desired.
		\end{proof}
		
		For the next part of this section, it is more convenient to work with $\beta = 1$. Recall that $\kdivdiff_i^{(1)} = \divdiff_i + \divdiff_ix_{i + 1}$ and $(\kdivdiff_i^{(1)})^2 = -\kdivdiff_i^{(1)}$.
		We will prove the following:
		\begin{theorem}
			\label{thm:pipe_groth_to_schub_no_signs}
			Let $w \in S_n$. Then
			\[\mathfrak{G}_w^{(1)} = \sum_{\substack{\mathcal{P} \in \kpipes(w)\\ \check{\mathcal{P}} \text{ is reduced}}}\mathfrak{S}_{\demprod(\check{\mathcal{P}})}. \]
		\end{theorem}
		
		\begin{proof}
			We proceed by reverse induction on Bruhat order. In the base case $w = w_0$, there is a single pipe dream of $w_0$, and it is immediate to check that the statement holds. 
			Now suppose $w \neq w_0$, and assume that the statement holds for all $u \in S_n$ such that $u > w$ in Bruhat order. 
			
			Let $j = \addable(w)$ and $i = w^{-1}(j)$. 
			By \cref{lemma:kcotransitiondescents}, we have $ws_i > w$. By \cref{lemma:whathappenswithwsianddescents}, $ws_i\in S_n$. Thus, $ws_i$ falls within the inductive hypothesis.
			As such,
			\begin{equation}
				\label{eq:grothendieckexpansioninductive}
				\mathfrak{G}_{ws_i}^{(1)} = \sum_{\substack{\mathcal{Q} \in \kpipes(ws_i)\\ \check{\mathcal{Q}} \text{ is reduced}}}\mathfrak{S}_{\demprod(\check{\mathcal{Q}})}.
			\end{equation}
			Because $\kdivdiff_i^{(1)} = \divdiff_i + \divdiff_i x_{i + 1}$, we have
			\begin{align*}
				\mathfrak{G}_w^{(1)}
				&= \kdivdiff_i^{(1)}(\mathfrak{G}_{ws_i}^{(1)}) & \text{(by \cref{eq:grothdivdiffbeta})}\\
				& = \divdiff_i(\mathfrak{G}_{ws_i}^{(1)}) + \divdiff_i (x_{i + 1}\mathfrak{G}_{ws_i}^{(1)})\\
				& = \left(\sum_{\substack{\mathcal{Q} \in \kpipes(ws_i)\\ \check{\mathcal{Q}} \text{ is reduced}}}\divdiff_i(\mathfrak{S}_{\demprod(\check{\mathcal{Q}})})\right) + \divdiff_i( x_{i + 1}\mathfrak{G}_{ws_i}^{(1)}) &\text{(by \cref{eq:grothendieckexpansioninductive})}.
			\end{align*}
			Also,
			\begin{align*}
				\divdiff_i( x_{i + 1}\mathfrak{G}_{ws_i}^{(1)})
				& = \sum_{u \in\kcotransitionperms_{i + 1}(ws_i)}\divdiff_i (\mathfrak{G}_u^{(1)} )&\text{(by \cref{prop:kcotransition} )}\\
				& = \sum_{\substack{u \in\kcotransitionperms_{i + 1}(ws_i)\\ u>us_i}}\divdiff_i (\mathfrak{G}_u^{(1)}) &\text{(by \cref{eq:grothorddivdiff})}\\
				& = \sum_{\substack{u \in\kcotransitionperms_{i}(w)\\u > ws_i}} N_i(\mathfrak{G}_u^{(1)}) & \text{(by Part (3) of \cref{lemma:whathappenswithwsianddescents})}.
			\end{align*}
			By \cref{lemma:whathappenswithwsianddescents}, since $j=\addable(w)$, each $u \in \kcotransitionperms_{i}(w)$ is in $S_n$. 
			In particular, each such $u$ satisfies $u > w$, so we may apply the inductive hypothesis to obtain 
			\begin{align*}
				\divdiff_i( x_{i + 1}\mathfrak{G}_{ws_i}^{(1)})
				&= \sum_{\substack{u \in\kcotransitionperms_{i}(w)\\u > ws_i}} \divdiff_i\left(\sum_{\substack{\mathcal{Q} \in \kpipes(u)\\ 
						\check{\mathcal{Q}} \text{ is reduced}}}\mathfrak{S}_{\demprod(\check{\mathcal{Q}})}\right) \\
				& = \sum_{\substack{u \in\kcotransitionperms_{i}(w)\\u > ws_i}} \sum_{\substack{\mathcal{Q} \in \kpipes(u)\\ 
				\check{\mathcal{Q}} \text{ is reduced}}}\divdiff_i(\mathfrak{S}_{\demprod(\check{\mathcal{Q}})}).
			\end{align*}
			Thus,
			\begin{align*}
				\mathfrak{G}_w^{(1)}& = \sum_{\substack{\mathcal{Q} \in \kpipes(ws_i)\\ \check{\mathcal{Q}} \text{ is reduced}}}\divdiff_i(\mathfrak{S}_{\demprod(\check{\mathcal{Q}})}) + \sum_{\substack{u \in\kcotransitionperms_{i}(w)\\u > ws_i}} \sum_{\substack{\mathcal{Q} \in \kpipes(u)\\ \check{\mathcal{Q}} \text{ is reduced}}}\divdiff_i(\mathfrak{S}_{\demprod(\check{\mathcal{Q}})})\\
				& = \sum_{\substack{u \in\kcotransitionperms_{i}(w)\\u\geq ws_i}} \sum_{\substack{\mathcal{Q} \in \kpipes(u)\\ \check{\mathcal{Q}} \text{ is reduced}}}\divdiff_i(\mathfrak{S}_{\demprod(\check{\mathcal{Q}})}).
			\end{align*}
			Indeed, the only $\mathcal{Q}$'s which contribute nonzero terms to the sum above have the property that $\demprod(\check{\mathcal{Q}})$ has a descent at $i$. 
			Therefore, by \cref{lemma:cotransitionflippermutation} and \cref{lemma:kcotransitionrestrictedbijection},
			\[\mathfrak{G}_w^{(1)} = \sum_{\substack{\mathcal{P} \in \kpipes(w)\\ \check{\mathcal{P}} \text{ is reduced}}}\mathfrak{S}_{\demprod(\check{\mathcal{P}})}. \qedhere\]
		\end{proof}

		We will now prove \cref{thm:pipe_groth_to_schub}.
		\thmpipeA*
		\begin{proof}
			By \cref{thm:pipe_groth_to_schub_no_signs}, we have \[\mathfrak{G}_w^{(1)} = \sum_{\substack{\mathcal{P} \in \kpipes(w)\\ \check{\mathcal{P}} \text{ is reduced}}}\mathfrak{S}_{\demprod(\check{\mathcal{P}})}.\]
			Make the substitution $x_i \mapsto -x_i$ for all $i$ into both sides of the above equation. 
			This substitution maps $\mathfrak{G}_w^{(1)} \mapsto (-1)^{\ell(w)}\mathfrak{G}_w$. 
			Furthermore, making the same substitution into a Schubert polynomial $\mathfrak S_v$ yields $(-1)^{\ell(v)}\mathfrak S_v$. 
			Thus, 
			\begin{align*}
			(-1)^{\ell(w)}\mathfrak{G}_w&=\mathfrak{G}_w^{(1)}(-x_1,-x_2,\ldots,-x_n)\\
			&=\sum_{\substack{\mathcal{P} \in \kpipes(w)\\ \check{\mathcal{P}} \text{ is reduced}}}\mathfrak{S}_{\demprod(\check{\mathcal{P}})}(-x_1,-x_2,\ldots,-x_n)\\
			&=\sum_{\substack{\mathcal{P} \in \kpipes(w)\\ \check{\mathcal{P}} \text{ is reduced}}}(-1)^{\ell(\demprod(\check{\mathcal{P}}))}\mathfrak{S}_{\demprod(\check{\mathcal{P}})},
			\end{align*}
			and the theorem follows after multiplying both sides by $(-1)^{-\ell(w)}$.
		\end{proof}

		\subsection{Changing bases with chains}
		\label{section:chains}
		
		The goal of this section is to interpret the Grothendieck to Schubert expansion as a sum over certain chains in Bruhat order. 
		We also give formulas for the monomial expansions of Schubert and Grothendieck polynomials that are closely related to the climbing chain models of \cite{Bergeron.Sottile} and \cite{Lenart.Robinson.Sottile}. 
		To state our theorems, we need the following definitions.
		
		We define three types of directed, labeled graphs: $\mathcal{G}_n$, $\overline{\mathcal{G}}_n$, and $\mathcal{R}_n$. 
		All three of these graphs have the same vertex set, $S_n$. 
		Their edge sets are defined as follows.
		Given $w\in S_n$, let $j=\addable(w)$ and $i=w^{-1}(j)$. 
		In $\mathcal{G}_n$, there is a directed edge $w\xrightarrow{(i,j)} u$ for all $u\in \cotransitionperms_i(w)$.
		In $\overline{\mathcal{G}}_n$, there is a directed edge $w\xrightarrow{(i,j)} u$ for all $u\in \kcotransitionperms_i(w)$.
		Finally, in $\mathcal{R}_n$, there is a directed edge $w\xrightarrow{(i,j)}u$ for all $u\in \kcotransitionperms_i(w)$ such that $u\geq ws_i$.		
		
		See \cref{figure:G3graphs} for an illustration of the graphs $\mathcal{G}_3$ and $\overline{\mathcal{G}}_3$. 
		See \cref{figure:R3graph} for an illustration of the graph $\mathcal{R}_3$. 
		Note that edge sets of $\mathcal{G}_n$ and $\mathcal{R}_n$ are each a subset of the edges of $\overline{\mathcal{G}}_n$.
		
		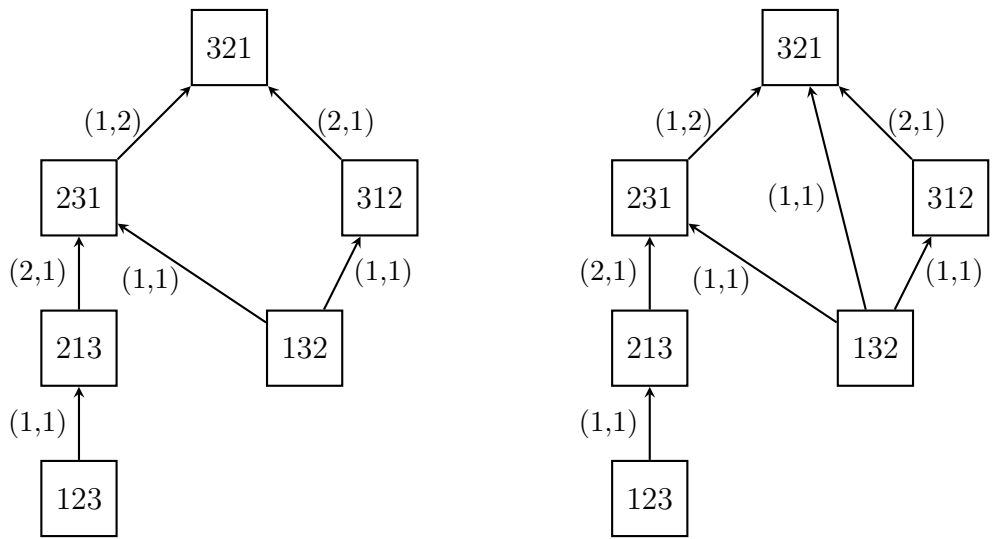
\begin{figure}
			\begin{tikzpicture}[
				- > , 
				 > = stealth, 
				node distance = 2cm, 
				thick, 
				vertex/.style = { draw, minimum size = 1cm, inner sep = 0pt}, 
				edge label/.style = {font = \small} 
				]
				\node[vertex] (123) at (0,0) {123};
				\node[vertex] (132) at (3,2) {132};
				\node[vertex] (213) at (0,2) {213};
				\node[vertex] (231) at (0,4) {231};
				\node[vertex] (312) at (4,4) {312};
				\node[vertex] (321) at (2,6) {321};
				\draw (123) to node[edge label, left] {(1,1)} (213);
				\draw (213) to node[edge label, left] {(2,1)} (231);
				\draw (231) to node[edge label, left] {(1,2)} (321);
				\draw (312) to node[edge label, right] {(2,1)} (321);
				\draw (132) to node[edge label, right] {(1,1)} (312);
				\draw (132) to node[edge label,yshift=-.25em, left] {(1,1)} (231);
			\end{tikzpicture}
			\hspace{4em}
			\begin{tikzpicture}[
				- > , 
				> = stealth, 
				node distance = 2cm, 
				thick, 
				vertex/.style = { draw, minimum size = 1cm, inner sep = 0pt}, 
				edge label/.style = {font = \small} 
				]
				\node[vertex] (123) at (0,0) {123};
				\node[vertex] (132) at (3,2) {132};
				\node[vertex] (213) at (0,2) {213};
				\node[vertex] (231) at (0,4) {231};
				\node[vertex] (312) at (4,4) {312};
				\node[vertex] (321) at (2,6) {321};
				\draw (123) to node[edge label, left] {(1,1)} (213);
				\draw (213) to node[edge label, left] {(2,1)} (231);
				\draw (231) to node[edge label, left] {(1,2)} (321);
				\draw (312) to node[edge label, right] {(2,1)} (321);
				\draw (132) to node[edge label, right] {(1,1)} (312);
				\draw (132) to node[edge label, left] {(1,1)} (321);
				\draw (132) to node[edge label,yshift=-.25em, left] {(1,1)} (231);
			\end{tikzpicture}
			\caption{The graph $\mathcal{G}_3$ is pictured on the left. $\overline{\mathcal{G}}_3$ is on the right.}
			\label{figure:G3graphs}
			\end{figure}
			
			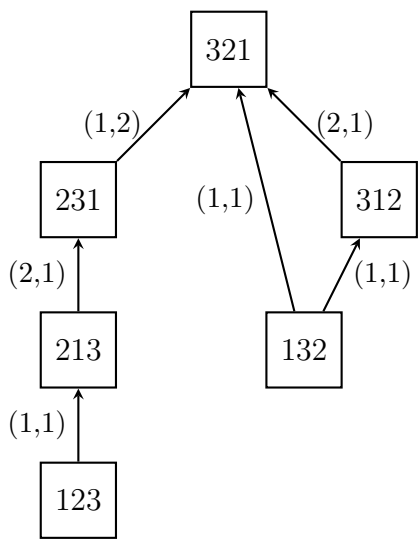
\begin{figure}
				\begin{tikzpicture}[
				- > , 
				> = stealth, 
				node distance = 2cm, 
				thick, 
				vertex/.style = { draw, minimum size = 1cm, inner sep = 0pt}, 
				edge label/.style = {font = \small} 
				]
				\node[vertex] (123) at (0,0) {123};
				\node[vertex] (132) at (3,2) {132};
				\node[vertex] (213) at (0,2) {213};
				\node[vertex] (231) at (0,4) {231};
				\node[vertex] (312) at (4,4) {312};
				\node[vertex] (321) at (2,6) {321};
				\draw (123) to node[edge label, left] {(1,1)} (213);
				\draw (213) to node[edge label, left] {(2,1)} (231);
				\draw (231) to node[edge label, left] {(1,2)} (321);
				\draw (312) to node[edge label, right] {(2,1)} (321);
				\draw (132) to node[edge label, right] {(1,1)} (312);
				\draw (132) to node[edge label, left] {(1,1)} (321);
			\end{tikzpicture}
			\caption{The graph $\mathcal{R}_3$.}
			\label{figure:R3graph}
		\end{figure}
		
		Given a directed path $p$ of the form 
		\[u_1\xrightarrow{(i_1,j_1)}u_2 \xrightarrow{(i_2,j_2)} u_3\xrightarrow{(i_3,j_3)} \cdots \xrightarrow{(i_{k-1},j_{k-1})} u_k\] 
		in $\overline{\mathcal{G}}_n$, we define the \newword{word} of $p$ to be $\word(p) = (i_{k-1},i_{k-2},\ldots, i_2, i_{1})$. Also, the \newword{weight} of $p$ is $\wt(p)=\prod_{h=1}^{k-1} x_h$.
		We say that $p$ is \newword{reduced} if $\word(p)$ is a reduced word. 
		
		Given $\heartsuit\in\{\mathcal{G}_n,\overline{\mathcal{G}}_n,\mathcal{R}_n\}$ and $u,v \in S_n$, let $\paths_\heartsuit(u,v)$ denote the set of paths in $\heartsuit$ that start at $u$ and end at $v$. 
		
		\begin{theorem}
			\label{thm:pathformulas}
			Let $w\in S_n$. The following hold.
			\begin{enumerate}
				\item $\displaystyle \mathfrak G_w^{(1)}=x_1^{n-1}x_2^{n-2}\cdots x_{n-1} \sum_{p\in \paths_{\overline{\mathcal{G}}_n}(w,w_0)}\frac{1}{\wt(p)}$.
				\item $\displaystyle \mathfrak S_w=x_1^{n-1}x_2^{n-2}\cdots x_{n-1} \sum_{p\in \paths_{\mathcal{G}_n}(w,w_0)}\frac{1}{\wt(p)}$.
				\item $\displaystyle\mathfrak{G}_w^{(1)} = \sum_{\substack{p \in \paths_{\mathcal{R}_n}(w,w_0)\\p \text{ is reduced}}}\mathfrak{S}_{w_0 \cdot \demprod(\word(p))}$.
				\item $\displaystyle\mathfrak{S}_w = \sum_{p \in \paths_{\mathcal{G}_n}(w,w_0)}\mathfrak{G}_{w_0 \cdot \demprod(\word(p))}$.
			\end{enumerate}
		\end{theorem}
	
	Before proceeding, we give an example of \cref{thm:pathformulas}.
	\begin{example}
		Let $w=132$. From \cref{figure:G3graphs}, we see that there are two paths in $\mathcal G_3$ from $132$ to $321$. In particular, Part (1) of \cref{thm:pathformulas} says that
		\[\mathfrak S_w=x_1^2x_2(\frac{1}{x_1^2}+\frac{1}{x_1x_2})=x_2+x_1.\]
		Part (4) says that
		\[\mathfrak S_w=\mathfrak G_{w_0\cdot \demprod((1,1))}+\mathfrak G_{w_0\cdot \demprod((2,1))}=\mathfrak G_{231}+\mathfrak G_{132}.\]
		There are three paths from $132$ to $321$ in $\overline{\mathcal G}_3$. 
		Part (2) of \cref{thm:pathformulas} tells us that
		\[\mathfrak G_{w}^{(1)}=x_1^2x_2(\frac{1}{x_1^2}+\frac{1}{x_1x_2}+\frac{1}{x_1})=x_2+x_1+x_1x_2.\]
		Finally, there are two paths from $132$ to $321$ in $\mathcal R_3$. Part (3) of \cref{thm:pathformulas} says that
		\[\mathfrak G_{w}^{(1)}=\mathfrak S_{w_0\cdot(s_1)}+\mathfrak S_{w_0\cdot (s_2s_1)}=\mathfrak S_{231}+\mathfrak S_{132}. \qedhere\]
	\end{example}
		
		\begin{remark}
			The first part of \cref{thm:pathformulas} is closely related to the marked climbing chain formula for Grothendieck polynomials of \cite{Lenart.Robinson.Sottile}, but is presented in a different form. Namely, we allow chains of different lengths, as our edges arise from K-theoretic co-transition.
			The second part of \cref{thm:pathformulas} is directly equivalent to Bergeron and Sottile's climbing chains formula for Schubert polynomials \cite{Bergeron.Sottile}. 
			See also \cite{Lenart.Sottile} which gives a direct bijection from $\pipes(w)$ to $\paths_{\mathcal{G}_n}(w,w_0)$. 
			We recall this bijection in \cref{lemma:graphbijections}. 
			The third part, the Grothendieck to Schubert expansion, expressed as a sum over chains, is new. 
		\end{remark}
		
		We will prove \cref{thm:pathformulas} after establishing bijections from sets of chains in Bruhat order to sets of pipe dreams; our three bijections are all defined by the same mapping. Namely, if we have a path $p$ of the form
		\[w=u_1\xrightarrow{(i_1,j_1)}u_2 \xrightarrow{(i_2,j_2)} u_3\xrightarrow{(i_3,j_3)} \cdots \xrightarrow{(i_{k-1},j_{k-1})} u_k = w_0\] in $\mathcal{G}_n$, we associate to $p$ the pipe dream $\mathbb{T}_n-\{(i_h,j_h):h \in [k-1]\}$. 
		Write $p \mapsto \varphi(p)$ for this map. We sometimes restrict $\varphi$ to subsets of $\displaystyle \bigsqcup_{u\in S_n} \paths_{\overline{\mathcal{G}}_n}(u,w_0)$, and abuse notation by also calling the restriction $\varphi$.

	We start by stating bijections between certain sets of paths and pipe dreams. These bijections follow quickly from the co-transition recurrences.  
		
		\begin{lemma}
			\label{lemma:graphbijections}
		\begin{enumerate}
		\item The map 
		$\varphi:\paths_{\overline{\mathcal{G}}_n}(w,w_0)\rightarrow \kpipes(w)$ is a bijection.
		\item The map 
		$\varphi:\paths_{\mathcal{G}_n}(w,w_0)\rightarrow \pipes(w)$ is a bijection.
		\end{enumerate}
		\end{lemma}
		
		\begin{proof}
			
		\noindent (1) We proceed by reverse induction on Bruhat order. In the base case $w=w_0$, there is a single path $p$ with no edges in $\paths_{\overline{\mathcal{G}}_n}(w_0,w_0)$, and $\varphi(p)=\mathbb T_n$. Because $\kpipes(w_0)=\{\mathbb T_n\}$, the base case holds.
		
		Now fix $w\neq w_0$ and assume the statement holds for all $u>w$. Write $j=\addable(w)$ and $i=w^{-1}(j)$. 
		Each $p\in \paths_{\overline{\mathcal{G}}_n}(w,w_0)$ is formed by taking $p'\in \paths_{\overline{\mathcal{G}}_n}(u,w_0)$, for some $u\in \kcotransitionperms_{i}(w)$. 
		Furthermore, $\varphi(p)=\varphi(p')-\{(i,j)\}$. 
		The map 
		\[\paths_{\overline{\mathcal{G}}_n}(w,w_0)\rightarrow \bigsqcup_{u \in \kcotransitionperms_i(w)} \paths_{\overline{\mathcal{G}}_n}(u,w_0) \]
		defined by deleting the first edge in the path is bijective.

		By the inductive hypothesis, the map
		\[\bigsqcup_{u \in \kcotransitionperms_i(w)} \paths_{\overline{\mathcal{G}}_n}(u,w_0)\rightarrow \bigsqcup_{u \in \kcotransitionperms_i(w)} \kpipes(u)\]
		defined by $p\mapsto \varphi(p)$
		is a bijection.
		By Part (1) of \cref{prop:kcotransitionpipefacts}, the map $\mathcal Q\mapsto \mathcal Q- \{(i,j)\}$ defines a bijection 
		\[\bigsqcup_{u \in \kcotransitionperms_i(w)} \kpipes(u)\rightarrow \kpipes(w).\]
		Thus the result follows by noting that $\varphi:\paths_{\overline{\mathcal{G}}_n}(w,w_0)\rightarrow \kpipes(w)$ is the composition of these bijections.

		\noindent (2) The proof of Part (2) follows by modifying the proof given in Part (1), replacing all K-theoretic objects with their cohomological counterparts.  
		\end{proof}
	
	\begin{remark}
		Part (2) of \cref{lemma:graphbijections} was originally proved in  \cite[Lemma~10]{Lenart.Sottile}.
	\end{remark}

	Next, we examine how the map $\varphi$ impacts permutations associated to paths and their corresponding pipe dreams.
	
		\begin{lemma}
			\label{lemma:pathwords}
			Fix a path $p$ from $w$ to $w_0$ in $\overline{\mathcal{G}}_n$. Write $\mathcal{P} = \varphi(p)$. The following hold.
		\begin{enumerate}
		\item $w_0\cdot \demprod(\word(p)) = \demprod(\check{\mathcal{P}})$.
		\item $p$ is reduced if and only if $\check{\mathcal{P}}$ is reduced.
		\end{enumerate}
		\end{lemma}
		
		\begin{proof}
		\noindent (1) Let $\mathcal{P} = \varphi(p)$. Observe that $\word(p)$ can be obtained from $\mathbf{a}_{\check{\mathcal{P}}}$ by replacing each letter $i$ with $n - i$ for all $i$ in the word.
			Thus, by \cref{lemma:conjugationandwords}, $w_0\demprod(\word(p))w_0 = \demprod(\mathbf{a}_{\check{\mathcal{P}}})$,
			and hence,
			$w_0\demprod(\word(p)) = \demprod(\mathbf{a}_{\check{\mathcal{P}}})w_0 = \demprod(\check{\mathcal{P}}).$
			
		\noindent (2) Again, this follows from the observation that $\word(p)$ can be obtained from $\mathbf{a}_{\check{\mathcal{P}}}$ by replacing each letter $i$ with $n - i$ for all $i$ in the word.
		\end{proof}
		
		We now show that the reduced paths in $\paths_{\mathcal{R}_n}(w,w_0)$ are in bijection with pipe dreams for $w$ with reduced co-pipe dreams.
		
		\begin{proposition}
			\label{thm:pathreducedcopdbijection}
The map 
\begin{equation}
	\label{eqn:pathreducedcobpdbijection}
\varphi:\{p \in \paths_{\mathcal{R}_n}(w,w_0): p \text{ is reduced}\} \rightarrow \{\mathcal{P} \in \kpipes(w):\check{\mathcal{P}} \text{ is reduced}\}
\end{equation}
 is a bijection.
		\end{proposition}
		\begin{proof}
			
We proceed by reverse induction on Bruhat order. 
The statement is immediate for $w=w_0$. Assume that $w\neq w_0$ and that the statement holds for all $u>w$.
			
Note that $\{p\in\paths_{\mathcal{R}_n}(w,w_0): p \text{ is reduced}\}$ is a subset of $\paths_{\overline{\mathcal{G}}_n}(w,w_0)$ (thinking of the graph $\mathcal{R}_n$ as being embedded in $\overline{\mathcal{G}}_n$). Because the map in \cref{lemma:graphbijections} is bijective, our given map here is injective. Additionally, \cref{lemma:pathwords} Part (2) tells us that the map in \cref{eqn:pathreducedcobpdbijection} is well defined. 

It remains to show the map is surjective. 
Take $\mathcal{P}\in \kpipes(w)$ with $\check{\mathcal{P}}$ reduced. 
Write $j=\addable(w)$ and $i=w^{-1}(j)$.
By \cref{lemma:kcotransitionrestrictedbijection}, we have $\mathcal Q=\mathcal{P}\cup\{(i,j)\}\in \kpipes(u)$ with $u\in \kcotransitionperms_{i}(w)$ and $u\geq ws_i$. Also, $\check{\mathcal Q}$ is reduced and has a descent at $i$.

We know that there exists $p\in \paths_{\overline{\mathcal{G}}_n}(w,w_0)$ such that $\varphi(p)=\mathcal{P}$. Furthermore, if $p'$ is the path we get by deleting the first edge of $p$, then $\varphi(p')=\mathcal{Q}$. Thus, by induction $p'\in \paths_{\mathcal{R}_n}(u,w_0)$ and $p'$ is reduced. Because $u\in \kcotransitionperms_{i}(w)$ and $u\geq ws_i$, we have $w\xrightarrow{(i,j)}u$ is an edge in $\mathcal{R}_n$, thus $p\in \paths_{\mathcal{R}_n}(w,w_0)$.

It remains to show that $p$ is reduced. Because $\varphi(p')=\mathcal{Q}$ and $\demprod(\check{\mathcal{Q}})$ has a descent at $i$, $\demprod(\word(p'))=w_0\demprod(\check{\mathcal{Q}})$ has an ascent at $i$. 
We obtain $\word(p)$ by appending $i$ to $\word(p')$, which is a reduced word.
Thus, $\word(p)$ is a reduced word.
\end{proof}
		
		We now prove the main theorem of this section.
		\begin{proof}[Proof of \cref{thm:pathformulas}]
For parts (1) and (2), first note that given $p\in \paths_{\overline{\mathcal{G}}_n}(w,w_0)$, we have $\frac{x_1^{n-1}x_2^{n-2}\cdots x_{n-1}}{\wt(p)}=\wt(\varphi(p))$. Thus these statements follow from \cref{lemma:graphbijections} and \cref{theorem:pipesschubertandgrothendieck}.

Part (3) follows from \cref{thm:pipe_groth_to_schub_no_signs}, \cref{lemma:pathwords}, and \cref{thm:pathreducedcopdbijection}.

Finally, Part (4) follows from \cref{thm:pipe_schub_to_groth}, \cref{lemma:graphbijections}, and \cref{lemma:pathwords}.
		\end{proof}

	We conclude with an example for which there is a path in $\paths_{\mathcal{R}_n}(w,w_0)$ that is not reduced.
	
	\begin{example}
		Let $w = 1243$. We list the paths in $\mathcal{R}_4$ that go from $w$ to $w_0$ below.
		
		\begin{itemize}
			\item $1243\xrightarrow{(1, 1)}2143\xrightarrow{(2, 1)}2413\xrightarrow{(3, 1)}2431\xrightarrow{(1, 2)}4231\xrightarrow{(2, 2)}4321$
			\item $1243\xrightarrow{(1, 1)}2143\xrightarrow{(2, 1)}2413\xrightarrow{(3, 1)}2431\xrightarrow{(1, 2)}4321$
			\item $1243\xrightarrow{(1, 1)}2143\xrightarrow{(2, 1)}2431\xrightarrow{(1, 2)}4231\xrightarrow{(2, 2)}4321$
			\item $1243\xrightarrow{(1, 1)}2143\xrightarrow{(2, 1)}2431\xrightarrow{(1, 2)}4321$
		\end{itemize}
		The third path has the word $2121$, which is not reduced. 
		As such, we discard it. 
		The remaining paths have reduced words. 
		By \cref{thm:pathformulas}, we have
		\begin{align*}
			\mathfrak G_{1243}^{(1)}&=\mathfrak S_{w_0\demprod((2,1,3,2,1))}+\mathfrak S_{w_0\demprod((1,3,2,1))}+\mathfrak S_{w_0\demprod((1,2,1))}\\
			&=\mathfrak S_{w_0\cdot 4312}+\mathfrak S_{w_0\cdot4213 }+\mathfrak S_{w_0\cdot 3214}\\
			&=\mathfrak S_{ 1243}+\mathfrak S_{1342 }+\mathfrak S_{ 2341}.
		\end{align*}
		We list the pipe dreams associated to the first, second, and fourth paths below.
		
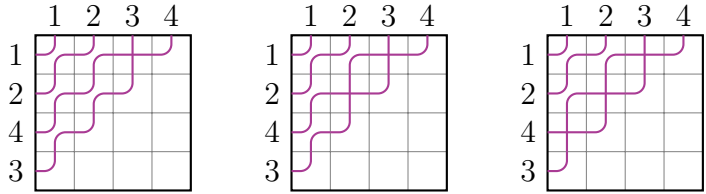
\[
		\begin{tikzpicture}[x = 1.25em,y = 1.25em]
			\draw[step = 1,gray, very thin] (0,0) grid (4, -4);
			\draw[color = black, thick] (0,0) rectangle (4, -4);
			\draw[thick,rounded corners,color = Mulberry](1/2, 0)--(1/2, -1/2)--(0, -1/2);
			\draw[thick,rounded corners,color = Mulberry](1, -1/2)--(1/2, -1/2)--(1/2, -1);
			\draw[thick,rounded corners,color = Mulberry](3/2, 0)--(3/2, -1/2)--(1, -1/2);
			\draw[thick,rounded corners,color = Mulberry](2, -1/2)--(3/2, -1/2)--(3/2, -1);
			\draw[thick,rounded corners,color = Mulberry](2, -1/2)--(3, -1/2);
			\draw[thick,rounded corners,color = Mulberry](5/2, 0)--(5/2, -1);
			\draw[thick,rounded corners,color = Mulberry](7/2, 0)--(7/2, -1/2)--(3, -1/2);
			\draw[thick,rounded corners,color = Mulberry](1/2, -1)--(1/2, -3/2)--(0, -3/2);
			\draw[thick,rounded corners,color = Mulberry](1, -3/2)--(1/2, -3/2)--(1/2, -2);
			\draw[thick,rounded corners,color = Mulberry](3/2, -1)--(3/2, -3/2)--(1, -3/2);
			\draw[thick,rounded corners,color = Mulberry](2, -3/2)--(3/2, -3/2)--(3/2, -2);
			\draw[thick,rounded corners,color = Mulberry](5/2, -1)--(5/2, -3/2)--(2, -3/2);
			\draw[thick,rounded corners,color = Mulberry](1/2, -2)--(1/2, -5/2)--(0, -5/2);
			\draw[thick,rounded corners,color = Mulberry](1, -5/2)--(1/2, -5/2)--(1/2, -3);
			\draw[thick,rounded corners,color = Mulberry](3/2, -2)--(3/2, -5/2)--(1, -5/2);
			\draw[thick,rounded corners,color = Mulberry](1/2, -3)--(1/2, -7/2)--(0, -7/2);
			\node at (1/2, 1/2){1};
			\node at (-1/2, -1/2){1};
			\node at (3/2, 1/2){2};
			\node at (-1/2, -3/2){2};
			\node at (5/2, 1/2){3};
			\node at (-1/2, -5/2){4};
			\node at (7/2, 1/2){4};
			\node at (-1/2, -7/2){3};
		\end{tikzpicture}
		\hspace{2em}
		\begin{tikzpicture}[x = 1.25em,y = 1.25em]
			\draw[step = 1,gray, very thin] (0,0) grid (4, -4);
			\draw[color = black, thick] (0,0) rectangle (4, -4);
			\draw[thick,rounded corners,color = Mulberry](1/2, 0)--(1/2, -1/2)--(0, -1/2);
			\draw[thick,rounded corners,color = Mulberry](1, -1/2)--(1/2, -1/2)--(1/2, -1);
			\draw[thick,rounded corners,color = Mulberry](3/2, 0)--(3/2, -1/2)--(1, -1/2);
			\draw[thick,rounded corners,color = Mulberry](2, -1/2)--(3/2, -1/2)--(3/2, -1);
			\draw[thick,rounded corners,color = Mulberry](2, -1/2)--(3, -1/2);
			\draw[thick,rounded corners,color = Mulberry](5/2, 0)--(5/2, -1);
			\draw[thick,rounded corners,color = Mulberry](7/2, 0)--(7/2, -1/2)--(3, -1/2);
			\draw[thick,rounded corners,color = Mulberry](1/2, -1)--(1/2, -3/2)--(0, -3/2);
			\draw[thick,rounded corners,color = Mulberry](1, -3/2)--(1/2, -3/2)--(1/2, -2);
			\draw[thick,rounded corners,color = Mulberry](1, -3/2)--(2, -3/2);
			\draw[thick,rounded corners,color = Mulberry](3/2, -1)--(3/2, -2);
			\draw[thick,rounded corners,color = Mulberry](5/2, -1)--(5/2, -3/2)--(2, -3/2);
			\draw[thick,rounded corners,color = Mulberry](1/2, -2)--(1/2, -5/2)--(0, -5/2);
			\draw[thick,rounded corners,color = Mulberry](1, -5/2)--(1/2, -5/2)--(1/2, -3);
			\draw[thick,rounded corners,color = Mulberry](3/2, -2)--(3/2, -5/2)--(1, -5/2);
			\draw[thick,rounded corners,color = Mulberry](1/2, -3)--(1/2, -7/2)--(0, -7/2);
			\node at (1/2, 1/2){1};
			\node at (-1/2, -1/2){1};
			\node at (3/2, 1/2){2};
			\node at (-1/2, -3/2){2};
			\node at (5/2, 1/2){3};
			\node at (-1/2, -5/2){4};
			\node at (7/2, 1/2){4};
			\node at (-1/2, -7/2){3};
		\end{tikzpicture}
		\hspace{2em}
		\begin{tikzpicture}[x = 1.25em,y = 1.25em]
			\draw[step = 1,gray, very thin] (0,0) grid (4, -4);
			\draw[color = black, thick] (0,0) rectangle (4, -4);
			\draw[thick,rounded corners,color = Mulberry](1/2, 0)--(1/2, -1/2)--(0, -1/2);
			\draw[thick,rounded corners,color = Mulberry](1, -1/2)--(1/2, -1/2)--(1/2, -1);
			\draw[thick,rounded corners,color = Mulberry](3/2, 0)--(3/2, -1/2)--(1, -1/2);
			\draw[thick,rounded corners,color = Mulberry](2, -1/2)--(3/2, -1/2)--(3/2, -1);
			\draw[thick,rounded corners,color = Mulberry](2, -1/2)--(3, -1/2);
			\draw[thick,rounded corners,color = Mulberry](5/2, 0)--(5/2, -1);
			\draw[thick,rounded corners,color = Mulberry](7/2, 0)--(7/2, -1/2)--(3, -1/2);
			\draw[thick,rounded corners,color = Mulberry](1/2, -1)--(1/2, -3/2)--(0, -3/2);
			\draw[thick,rounded corners,color = Mulberry](1, -3/2)--(1/2, -3/2)--(1/2, -2);
			\draw[thick,rounded corners,color = Mulberry](1, -3/2)--(2, -3/2);
			\draw[thick,rounded corners,color = Mulberry](3/2, -1)--(3/2, -2);
			\draw[thick,rounded corners,color = Mulberry](5/2, -1)--(5/2, -3/2)--(2, -3/2);
			\draw[thick,rounded corners,color = Mulberry](0, -5/2)--(1, -5/2);
			\draw[thick,rounded corners,color = Mulberry](1/2, -2)--(1/2, -3);
			\draw[thick,rounded corners,color = Mulberry](3/2, -2)--(3/2, -5/2)--(1, -5/2);
			\draw[thick,rounded corners,color = Mulberry](1/2, -3)--(1/2, -7/2)--(0, -7/2);
			\node at (1/2, 1/2){1};
			\node at (-1/2, -1/2){1};
			\node at (3/2, 1/2){2};
			\node at (-1/2, -3/2){2};
			\node at (5/2, 1/2){3};
			\node at (-1/2, -5/2){4};
			\node at (7/2, 1/2){4};
			\node at (-1/2, -7/2){3};
		\end{tikzpicture}
		\hspace{2em}
		\]
		The reader is invited to write down the corresponding co-pipe dreams and compare their permutations with the terms in the given expansion of $\mathfrak G_{1243}^{(1)}$.
	\end{example}

		\section{Co-permutations and the canonical bijection}
		\label{section:canonicalbijection}
		
		The goal of this section is to prove \cref{thm:main1} and \cref{thm:main2}, the BPD change of basis formulas. To do so, we will prove \cref{thm:coperm}, which says that co-permutations are preserved under the column-weight preserving canonical bijection of \cite{Gao.Huang}. The proof will involve analyzing the BPD insertion algorithm of \cite{Huang.2023}.
				
		\subsection{Insertion on reduced BPDs}
		
		We begin by recalling an insertion algorithm on BPDs, first given by Huang \cite{Huang.2023}. 
		Huang used this algorithm to give a bijective proof of Monk's rule. 
		We will only require the special case that arises from the co-transition recurrence.
		We refer the reader to \cite{Huang.2023} for full details and provide an overview here.
		
		An \newword{almost BPD} is a NE planar history that satisfies nearly the same conditions as those for reduced BPDs, except that the diagram contains exactly one bump tile \bumptile. 
		We now describe the \emph{droop move} of \cite{Lam.Lee.Shimozono}. 
		Suppose we are given a BPD or an almost BPD that satisfies the following conditions within a rectangle $[a,b] \times [c,d]$:
		\begin{itemize}
			\item a pipe enters the rectangle in cell $(b,c)$, travels vertically within column $c$, turns right in cell $(a,c)$, proceeds horizontally along row $a$, and exits the rectangle in cell $(a,d)$,
			\item cell $(b,c)$ is a \btile or a \drtile,
			\item other pipes do not bend within the region $[a,b] \times [c,d]-\{(a,c),(b,d)\}$.
		\end{itemize}
		In this situation, we may perform a \newword{droop move} by bending the pipe that turns right in cell $(a,c)$ so that it instead travels horizontally along row $b$, turns upward at cell $(b, d)$, and then continues vertically within column $d$. See the picture below for an example of a droop move.
		
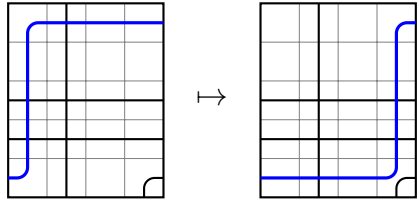
\[\begin{tikzpicture}[x = 1.25em,y = 1.25em]
			\draw[step = 1,gray, very thin] (0,0) grid (4, -5);
			\draw[color = black, thick] (0,0) rectangle (4, -5);
			\draw[color = black,thick](0,-2.5)--(4,-2.5);
			\draw[color = black,thick](0,-3.5)--(4,-3.5);
			\draw[color = black,thick](1.5,0)--(1.5,-5);
			\draw[color = black,thick,rounded corners](3.5,-5)--(3.5,-4.5)--(4,-4.5);
			\draw[very thick,rounded corners,color = blue](0, -4.5)--(1/2,-4.5)--(1/2, -1/2)--(4, -1/2);
		\end{tikzpicture}
		\hspace{1em}
		\raisebox{3em}{$\mapsto$}
		\hspace{1em}
		\begin{tikzpicture}[x = 1.25em,y = 1.25em]
			\draw[step = 1,gray, very thin] (0,0) grid (4, -5);
			\draw[color = black, thick] (0,0) rectangle (4, -5);
			\draw[color = black,thick](0,-2.5)--(4,-2.5);
			\draw[color = black,thick](0,-3.5)--(4,-3.5);
			\draw[color = black,thick](1.5,0)--(1.5,-5);
			\draw[color = black,thick, rounded corners](3.5,-5)--(3.5,-4.5)--(4,-4.5);
			\draw[very thick,rounded corners,color = blue](0, -4.5)--(1/2, -4.5)--(3.5,-4.5)--(3.5,-1/2)--(4, -1/2);
		\end{tikzpicture}\]		
		We refer to the pipe that bends as the \newword{active pipe}, and the rectangle in which the modifications happen as the \newword{active rectangle}. We also call the part of the active pipe that travels through cell $(a,c)$ the \newword{active elbow}.
		If a droop is restricted to two consecutive rows, we say that the droop is a \newword{row droop}.
		
		When doing a droop move, only the active pipe is modified, and the droop move preserves the underlying permutation. 
		Moreover, if we begin with a reduced BPD or almost BPD, the result is again a reduced BPD or almost BPD.

		A droop within an active rectangle $[a,b] \times [c,d]$ is called a \newword{mindroop} if no cell in the region $[a,b] \times [c,d]- \{(b,d)\}$ contains a \btile.
		Note that because a mindroop is a droop, this is equivalent to the conditions that $(a,c)$ is a \drtile or a \bumptile,
		\[b = \min\{i > a : (i,c) \text{ is not a } \ctile\},\] and 
		\[d = \min\{j > c : (a,j) \text{ is not a } \ctile\}.\]
		See the picture below for an example of a mindroop.
		
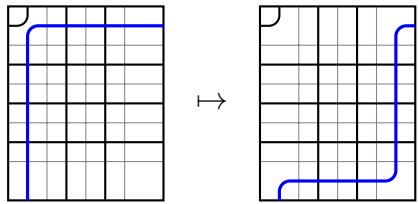
\[\begin{tikzpicture}[x = 1.25em,y = 1.25em]
			\draw[step = 1,gray, very thin] (0,0) grid (4, -5);
			\draw[color = black, thick] (0,0) rectangle (4, -5);
			\draw[color = black,thick](0,-1.5)--(4,-1.5);
			\draw[color = black,thick](0,-2.5)--(4,-2.5);
			\draw[color = black,thick](0,-3.5)--(4,-3.5);
			\draw[color = black,thick](1.5,0)--(1.5,-5);
			\draw[color = black,thick](2.5,0)--(2.5,-5);
			\draw[color = black,thick,rounded corners](0,-.5)--(.5,-.5)--(.5,0);
			\draw[very thick,rounded corners,color = blue](1/2, -5)--(1/2, -1/2)--(4, -1/2);
		\end{tikzpicture}
		\hspace{1em}
		\raisebox{3em}{$\mapsto$}
		\hspace{1em}
		\begin{tikzpicture}[x = 1.25em,y = 1.25em]
			\draw[step = 1,gray, very thin] (0,0) grid (4, -5);
			\draw[color = black, thick] (0,0) rectangle (4, -5);
			\draw[color = black,thick](0,-1.5)--(4,-1.5);
			\draw[color = black,thick](0,-2.5)--(4,-2.5);
			\draw[color = black,thick](0,-3.5)--(4,-3.5);
			\draw[color = black,thick](1.5,0)--(1.5,-5);
			\draw[color = black,thick](2.5,0)--(2.5,-5);
			\draw[color = black,thick,rounded corners](0,-.5)--(.5,-.5)--(.5,0);
			\draw[very thick,rounded corners,color = blue](1/2, -5)--(1/2, -4.5)--(3.5,-4.5)--(3.5,-1/2)--(4, -1/2);
		\end{tikzpicture}\]
		If a mindroop occupies only two consecutive rows, we call it a \newword{row mindroop}. 
		Note that a mindroop is a row mindroop if and only if the tile immediately below the active elbow is not a \ctile.
		
		Given an almost BPD, it is possible that two pipes may cross and also both appear together in a \bumptile. 
		If we exchange the positions of the \bumptile and the \ctile, this operation is called a \newword{cross-bump swap}. 
		We provide an example of such a move below. The blue pipes are involved in the cross-bump swap and the black pipe is irrelevant.
		\[\begin{tikzpicture}[x = 1.25em,y = 1.25em]
			\draw[step = 1,gray, very thin] (0,0) grid (4, -5);
			\draw[color = black, thick] (0,0) rectangle (4, -5);
			\draw[thick,rounded corners](0, -2.5)--(3/2, -2.5)--(3/2, -3/2)--(4,-1.5);
			\draw[thick,rounded corners,color = blue](0, -4.5)--(1/2, -4.5)--(1/2, -1/2)--(4,-1/2);
			\draw[thick,rounded corners,color = blue](1/2,-5)--(1/2,-4.5)--(2.5,-4.5)--(2.5,-2.5)--(3.5,-2.5)--(3.5,0);
		\end{tikzpicture}
		\hspace{1em}
		\raisebox{3em}{$\mapsto$}
		\hspace{1em}
		\begin{tikzpicture}[x = 1.25em,y = 1.25em]
			\draw[step = 1,gray, very thin] (0,0) grid (4, -5);
			\draw[color = black, thick] (0,0) rectangle (4, -5);
			\draw[thick,rounded corners](0, -2.5)--(3/2, -2.5)--(3/2, -3/2)--(4,-1.5);
			\draw[thick,rounded corners,color = blue](1/2, -5)--(1/2, -4.5)--(1/2, -1/2)--(3.5,-.5)--(3.5,0);
			\draw[thick,rounded corners,color = blue](0,-4.5)--(1/2,-4.5)--(2.5,-4.5)--(2.5,-2.5)--(3.5,-2.5)--(3.5,-1/2)--(4,-1/2);
		\end{tikzpicture}
		\]

		Having defined droops and cross-bump swaps, we are now ready to describe column insertion on reduced BPDs. 
		Let $w \in S_n$, and fix $\mathcal{B} \in \bpd(w)$. 
		We will insert a blank tile into column $j$ of $\mathcal{B}$ using the following algorithm:
		\begin{enumerate}
			\item Look for the first \drtile in column $j$ (reading within column $j$ starting from the top). 
			Set this to be the active elbow.
			\item Perform a mindroop with the active elbow.
			\label{step:activemindroop}
			\begin{enumerate}
				\item If the mindroop results in a \ultile in the bottom-right corner of the active rectangle, search upward within the same column for the first \drtile. 
				Set this to be the new active elbow. 
				Return to Step~\ref{step:activemindroop}.
				\item Otherwise, the bottom-right corner is a \bumptile. 
				\begin{enumerate}
					\item If changing the \bumptile into a \ctile results in a reduced BPD, then make this replacement, stop, and return the resulting BPD as output.
					\item Otherwise, this \bumptile is eligible for a cross-bump swap. Do the cross-bump swap. 
					In the new bump, set the bottom elbow to be active elbow. Return to Step~\ref{step:activemindroop}.
				\end{enumerate}
			\end{enumerate}
		\end{enumerate}
		This procedure is called \newword{column insertion} into a reduced BPD. If $\mathcal B\mapsto \mathcal B'$ by insertion into column $j$, then $\mathcal B'$ has one more \btile in column $j$ than $\mathcal B$ does, and the same number in all other columns.

		Note that we use column insertion instead of row insertion. 
		The results of \cite{Huang.2023} extend to column insertion immediately by symmetry. 
		In particular, \cref{theorem:huang_bijection} follows directly from Huang's arguments about \emph{decorated BPDs} if we consider each blank tile to be labeled with $-y$.
				
		\begin{remark}
			If we restrict ourselves to size $n$ BPDs, it may be that column insertion fails to be well-defined. However, there is a natural inclusion of $\bpd_n$ into $\bpd_{n+1}$, and for a given $\mathcal{B}\in \bpd_n$, if we embed it into $\bpd_{n+1}$, column insertion is well-defined. For the insertion we use in our main arguments (inserting into column $j=\addable(w)$), this issue does not arise. 
		\end{remark}		
		
		\subsection{Co-transition insertion}
		For our purposes, it is enough to analyze the result of column insertion arising from the co-transition recurrence. 
		Specifically, given $w \in S_n$ with $w\neq w_0$, we restrict ourselves to insertion into column $j = \addable(w)$. 
		
		\begin{theorem}
			\label{theorem:huang_bijection}
			Let $w \in S_n$ with $w \neq w_0$, and let $j = \addable(w)$ and $i = w^{-1}(j)$. 
			Then insertion on BPDs into column $j$ defines a bijection from $\bpd(w)$ to $\bigsqcup_{v \in \cotransitionperms_i(w)}\bpd(v)$. 
			Furthermore, if $\mathcal{B} \mapsto \mathcal{B}'$ under this bijection, then $\mathcal{B}'$ has one more blank tile in column $j$ than $\mathcal{B}$, and the same number of blank tiles in all other columns.
		\end{theorem}
		\begin{proof}
			This is a special case of \cite[Theorem~5]{Huang.2023}.  
		\end{proof}
		
		We will show in \cref{lemma:rowdroop} that because $j=\addable(w)$, all mindroops encountered along the insertion path are row mindroops.

		We define 
		\[\mute(w) = \bigcup_{(i,j) \in\rothe(w)}[1,i] \times [1,j]\] 
		and call this the \newword{mutable region} of $w$. 
		A \newword{hook} in a BPD is a pipe that bends exactly once. 
		Note that the bend in a hook must be a $\drtile$. In particular, any pipe in a BPD which does not pass through the mutable region is a hook. Note that the Rothe BPD for $w$ is the only BPD for $w$ for which all of its pipes are hooks.
		
		We now recall a lemma which states that tiles outside the mutable region are the same across all BPDs of $w$.
		\begin{lemma}
			\label{lemma:mutable}
			Let $w \in S_n$. 
			If $\mathcal{B}_w$ is the Rothe BPD for $w$ and $\mathcal{B} \in \kbpd(w)$, then for all $(i,j)\notin \mute(w)$, the tile at position $(i,j)$ is the same in both $\mathcal{B}_w$ and $\mathcal{B}$.
		\end{lemma}
		\begin{proof}
			This is part of \cite[Lemma 4.6]{Weigandt}.
		\end{proof}
				
		Next we record two additional properties of column insertion into BPDs, in the case that we are inserting into column $j=\addable(w)$.
		\begin{lemma}
			\label{lemma:setuprowlemma}
			Let $w \in S_n$ with $w \neq w_0$, and set $j = \addable(w)$. 
			Let $\mathcal{B} \in \kbpd(w)$. 
			Then the following statements hold:
			\begin{enumerate}
				\item If $i < j$, then the pipe that enters $\mathcal{B}$ in column $i$ is a hook.
				\item Suppose a pipe with label $j$ runs vertically across row $k$ in $\mathcal{B}$, and the pipe still has label $j$ when it enters row $k-1$.
				If another pipe, say with label $i$, crosses $j$ horizontally in row $k$, then $k > n - j + 1$.
				\item The tile at position $(n - j + 1, j)$ is a \vtile in $\mathcal{B}$. 
			\end{enumerate}
		\end{lemma}
		
		\begin{proof}
			\noindent (1) The statement holds immediately if $\mathcal{B}$ is the Rothe BPD for $w$. 
			More generally, the claim follows from \cref{lemma:mutable} because the hooks labeled $i<j$ in the Rothe BPD lie entirely outside of the mutable region of $w$, and thus are the same in all BPDs of $w$.
			
			\noindent (2) By \cref{eqn:NEpipecross}, because the label of the vertical pipe stays the same, it must be that $i < j$. 
			By the previous part, pipe $i$ is a hook. 
			In particular, because $j = \addable(w)$, we have $i+w^{-1}(i)=n+1$. 
			Thus, pipe $i$ exits in row $w^{-1}(i)=n-i + 1$. 
			Because pipe $i$ is a hook that travels horizontally in row $k$, it must exit in row $k$.
			Therefore, $k = n-i + 1 > n-j + 1$.
			
			\noindent (3) The tile at position $(n-j + 1,j)$ is vertical in the Rothe BPD for $w$, and $(n-j + 1,j)$ is outside of the mutable region of $w$. 
			Thus, the claim follows by \cref{lemma:mutable}.
		\end{proof}
		
		\begin{lemma}
			\label{lemma:rowdroop}
			Let $w \in S_n$ with $w \neq w_0$, and let $j = \addable(w)$. 
			Fix $\mathcal{B} \in \bpd(w)$ and let $\mathcal{B}'$ be the BPD obtained by insertion into column $j$. 
			\begin{enumerate}
				\item In every cross-bump swap encountered during the insertion, if the \bumptile is in row $r$ then the \ctile lies in some row $r' < r$.
				\item Every mindroop move used in the insertion process is a row mindroop.
			\end{enumerate}
		\end{lemma}
	
	\begin{proof}
		
		First note that at all points during the insertion process, the active elbow has the label $j$. 
		We now proceed to prove each claim.
		
		\noindent (1) Suppose we perform a mindroop, followed by a cross-bump swap. 
		Let the \bumptile occur in row $r$ and the \ctile in row $r'$. Suppose for contradiction that $r'\geq r$. 
		Then the \ctile has pipe $j$ running vertically. 
		If the horizontal pipe has label $i$, then by \cref{eqn:NEpipecross}, $i < j$. 
		By replacing the \bumptile with a \ctile, we may apply \cref{lemma:setuprowlemma}, which says that pipe $i$ is a hook. However, a hook cannot be involved in a cross-bump swap, so this is a contradiction.

		\noindent (2) To verify each mindroop is a row mindroop, we need to ensure that each active elbow we encounter during the insertion process does not have a \ctile immediately below it.
		
		First consider a row mindroop on some BPD or almost BPD for $w$, with active elbow labeled $j$. 
		Suppose the active elbow is in row $a$. 
		If the algorithm terminates, there is nothing to show. 
		If the mindroop is followed by another mindroop, because we have a row mindroop, the new active elbow will be in some row $a'$ with $a'\leq a$. 
		Otherwise, the next move is a cross-bump swap. 
		Again write $a'$ for the row of the next active elbow. 
		We have that $a'$ is the row of the \ctile that we changed to a \bumptile.
		By the previous part, $a'\leq a$.
		
		We now start with the first mindroop. 
		Observe that because $j=\addable(w)$, the only \drtile in column $j$ of $\mathcal{B}$ belongs to pipe $j$. 
		By \cref{lemma:setuprowlemma}, if this tile lies in row $k$, then $k < n - j + 1$. 
		Also by \cref{lemma:setuprowlemma}, the tile in position $(k + 1, j)$ is not a \ctile, as the only horizontal crossings come from hooks, which occur in rows $k'$ with $k'>n-j+1$. 
		Thus, the first mindroop is a row mindroop.
		
		Furthermore, because we have established that the row indices of all of the active elbows encountered during the insertion process are weakly decreasing and have label $j$, none of them will have a \ctile immediately below, so all of these mindroops are also row mindroops.
	\end{proof}
		
		Often, along an insertion path, we encounter multiple row mindroops within the same row in succession. 
		We may compose these moves into a single row droop.
		We call such a move a \newword{maximal row droop} if it is not immediately preceded or followed by a row mindroop contained within the same two rows during the insertion procedure. In the next lemma, we summarize the types of corner configurations that may arise in maximal row droops.
		
		\begin{lemma}
			\label{lemma:maximalrowdroop}
			Within the active rectangle of the starting configuration of a maximal row droop, the possible corner configurations are as follows:
			\begin{enumerate}
				\item The top-left corner is either a \bumptile or a \drtile.
				\item The bottom-left corner must be a \vtile or a \ultile. If it is a \ultile, then the top-left corner must be a \bumptile.
				\item The top-right corner is either a \htile or a \ultile.
				\item The bottom-right corner must be a \drtile or a \btile. If it is a \btile, then the top-right corner must be a \ultile.
			\end{enumerate}
		\end{lemma}
		
		\begin{proof}
			\noindent (1) The condition on the top-left corner is necessary for there to be an active elbow to droop.
			
			\noindent (2) There must be a pipe exiting upward from the bottom-left corner so that it can connect to the active elbow in the top-left corner. 
			If this tile were a \ctile or a \bumptile, we would not be able to complete a row droop with the active elbow in the top-left corner.
			
			For the second statement: if the bottom-left corner is a \ultile, then the active elbow in the top-left corner must have arisen from a cross-bump swap; otherwise, we could have extended a row droop. 
			Thus, in this case, the top-left corner must be a \bumptile.
			
			\noindent (3) The top-right corner must have an incoming pipe from the tile to its left to be part of a droop. 
			However, the top-right corner must not be a \ctile or a \bumptile, as either would prevent the active pipe from drooping into the cell below.
			
			\noindent (4) The only two allowable tiles for drooping the active pipe into the bottom-right corner are a \btile or a \drtile.
			If the bottom-right corner is a \btile and the top-right corner is a \htile, then after drooping, the active elbow would remain in the same row, allowing the droop to be extended. 
			This would contradict maximality. Therefore, by the previous part, the top-right corner must be a \ultile.
		\end{proof}
				
		It is sometimes convenient to associate a co-BPD to an almost BPD. 
		Suppose $\mathcal{B}$ is an almost BPD, and $\mathcal{B}'$ is the BPD obtained by replacing the \bumptile with a \ctile. 
		Then define $\check{\mathcal{B}} = \check{\mathcal{B}'}$. 
		We are now ready to prove the key lemma.
				
		\begin{lemma}
			\label{lemma:columninsertionbpdcopermutations}
			Let $w \in S_n$ with $w \neq w_0$, and let $j = \addable(w)$ and $i = w^{-1}(j)$.
			Suppose that $\mathcal{B} \mapsto \mathcal{B}'$ via insertion into column $j$. 
			Then $\demprod(\check{\mathcal{B}}) = \demprod(\check{\mathcal{B}'}) \square \tau_i$.
		\end{lemma}
		
		\begin{proof}
			By \cref{lemma:rowdroop}, each mindroop is a row droop. 
			Moreover, by \cref{lemma:rowdroop}, in every cross-bump swap, if the initial bump occurs in row $r$, then the corresponding cross is in row $r'$ with $r' < r$. 
			Thus, the insertion algorithm begins by modifying a pair of consecutive rows, and works its way upward. 
			This gives rise to a sequence of (almost) BPDs:
			\[\mathcal{B} = \mathcal{B}_0 \mapsto \mathcal{B}_1 \mapsto \cdots \mapsto \mathcal{B}_h = \mathcal{B}'\]
			where $\mathcal{B}_1$ is the resulting (almost) BPD after modifying the first pair of rows, $\mathcal{B}_2$ is the (almost) BPD after modifying the second pair of rows, and so on.
			
			We will show that for all $1 \leq k < h$, we have $\demprod(\check{\mathcal{B}}_{k-1}) = \demprod(\check{\mathcal{B}}_{k})$. 
			Additionally, we will show that $\demprod(\check{\mathcal{B}}_{h-1}) = \demprod(\check{\mathcal{B}}_{h}) \square \tau_i$. 
			In particular, it follows that $\demprod(\check{\mathcal{B}}) = \demprod(\check{\mathcal{B}'}) \square \tau_i$.
			
			Fix some $k \in[h]$. 
			Suppose the move $\mathcal{B}_{k-1} \mapsto \mathcal{B}_k$ modifies tiles in rows $r$ and $r + 1$, and that the first tiles modified in these rows are in column $c$ and the last are in column $d$. 
			If the cell $(r,c)$ contains a \ctile, then the pipes must be involved in a cross-bump swap, and so the first move is to replace the \ctile with a \bumptile. 
			If not, then the tile at $(r, c)$ must be a \drtile. In either case, we carry out a maximal row droop, with active elbow in tile $(r,c)$.
			At this point, there could be a cross-bump swap in rows $r$ and $r + 1$, followed by another maximal row droop. 
			This alternating pattern of cross-bump swaps and maximal row droops continues until the final droop in rows $r$ and $r + 1$ occurs.
			If this process results in a \bumptile in cell $(r + 1,d)$, then it is replaced with a \ctile.
			Otherwise, we end with a \ultile in cell $(r+1,d)$. 
			At this point, we have obtained $\mathcal{B}_k$.
			
			We will proceed with a case analysis based on the tiles in columns $c$ and $d$ of rows $r$ and $r+1$ in $\mathcal{B}_{k-1}$. Note that the cases we encounter are exactly those from \cref{lemma:maximalrowdroop}, after replacing each \bumptile with a \ctile.

			Before analyzing these boundary tiles, we will show that the modification of cells in the interior region $[r, r+1] \times [c+1, d-1]$ of $\mathcal{B}_{k-1}$ preserves the co-permutation.

			First note that we may encounter the following replacements as we perform row droops.
			\begin{equation}
				\label{eqn:irreleventcolumns}
\begin{array}{c}
				\begin{tikzpicture}[x = 1.25em,y = 1.25em]
					\draw[step = 1,gray, very thin] (0,0) grid (1, -2);
					\draw[color = black, thick] (0,0) rectangle (1, -2);
					\draw[thick,rounded corners,color = blue](0, -1/2)--(1, -1/2);
				\end{tikzpicture}
				\quad
				\raisebox{1em}{$\mapsto$}
				\quad
				\begin{tikzpicture}[x = 1.25em,y = 1.25em]
					\draw[step = 1,gray, very thin] (0,0) grid (1, -2);
					\draw[color = black, thick] (0,0) rectangle (1, -2);
					\draw[thick,rounded corners,color = blue](0, -3/2)--(1, -3/2);
				\end{tikzpicture}
				\\[1ex]
				\begin{tikzpicture}[x = 1.25em,y = 1.25em]
					\draw[step = 1,gray, very thin] (0,0) grid (1, -2);
					\draw[color = black, thick] (0,0) rectangle (1, -2);
					\draw[thick,rounded corners,color = ForestGreen](0, -1/2)--(1, -1/2);
					\draw[thick,rounded corners,color = ForestGreen](1/2, 0)--(1/2, -1);
					\draw[thick,rounded corners,color = ForestGreen](1/2, -1)--(1/2, -2);
				\end{tikzpicture}
				\quad
				\raisebox{1em}{$\mapsto$}
				\quad
				\begin{tikzpicture}[x = 1.25em,y = 1.25em]
					\draw[step = 1,gray, very thin] (0,0) grid (1, -2);
					\draw[color = black, thick] (0,0) rectangle (1, -2);
					\draw[thick,rounded corners,color = ForestGreen](1/2, 0)--(1/2, -1);
					\draw[thick,rounded corners,color = ForestGreen](0, -3/2)--(1, -3/2);
					\draw[thick,rounded corners,color = ForestGreen](1/2, -1)--(1/2, -2);
				\end{tikzpicture}
			\end{array}
			\hspace{3em}
			\begin{array}{c}
				\begin{tikzpicture}[x = 1.25em,y = 1.25em]
					\draw[step = 1,gray, very thin] (0,0) grid (1, -2);
					\draw[color = black, thick] (0,0) rectangle (1, -2);
					\draw[thick,rounded corners,color = blue](0, -1/2)--(1, -1/2);
					\draw[thick,rounded corners,color = blue](1/2, 0)--(1/2, -1);
					\draw[thick,rounded corners,color = blue](1/2, -1)--(1/2, -2);
				\end{tikzpicture}
				\quad
				\raisebox{1em}{$\mapsto$}
				\quad
				\begin{tikzpicture}[x = 1.25em,y = 1.25em]
					\draw[step = 1,gray, very thin] (0,0) grid (1, -2);
					\draw[color = black, thick] (0,0) rectangle (1, -2);
					\draw[thick,rounded corners,color = blue](1/2, 0)--(1/2, -1);
					\draw[thick,rounded corners,color = blue](0, -3/2)--(1, -3/2);
					\draw[thick,rounded corners,color = blue](1/2, -1)--(1/2, -2);
				\end{tikzpicture}
				\\[1ex]
				\begin{tikzpicture}[x = 1.25em,y = 1.25em]
					\draw[step = 1,gray, very thin] (0,0) grid (1, -2);
					\draw[color = black, thick] (0,0) rectangle (1, -2);
					\draw[thick,rounded corners,color = ForestGreen](0, -1/2)--(1, -1/2);
				\end{tikzpicture}
				\quad
				\raisebox{1em}{$\mapsto$}
				\quad
				\begin{tikzpicture}[x = 1.25em,y = 1.25em]
					\draw[step = 1,gray, very thin] (0,0) grid (1, -2);
					\draw[color = black, thick] (0,0) rectangle (1, -2);
					\draw[thick,rounded corners,color = ForestGreen](0, -3/2)--(1, -3/2);
				\end{tikzpicture}
			\end{array}
		\end{equation}
			In each of the images, the modification on BPDs is pictured above and the corresponding change in the co-BPDs is pictured below. 
						
			After ignoring such columns as in \cref{eqn:irreleventcolumns}, the result of finishing a row droop, then doing a cross-bump swap, and then starting another row droop looks like the following:
			\[
			%
			\]
			Even if there are unpictured columns with vertical pipes in the co-BPDs, the composition of these replacements is co-permutation preserving, by \cref{lemma:planarmoves}. 
			Furthermore, the modification of the tiles within the rectangle $[r, r+1] \times [c+1, d-1]$ consists of repeated applications of these replacements, and thus also preserves the co-permutation. See the image below for an example.
			\[
			%
			\]

			We now analyze the boundary configurations.

			\noindent {\bf Configurations for $[r, r + 1] \times \{c\}$:}

			The first change within a pair of rows either comes from changing a cross to a bump, or from drooping an active elbow. 
			We summarize the possibilities below, showing the initial cells to be modified in $\mathcal{B}_{k-1}$ and the resulting cells in $\mathcal{B}_{k}$, along with the corresponding changes to the co-BPDs.

		\[	\begin{tabular}{|c|c|c|c|} 
			\hline
			Case 1&Case 2&Case 3& Case 4\\
\hline
			\end{tabular}\]
			
			Case 1 and Case 2 are co-permutation preserving. 
			By \cref{lemma:maximalrowdroop}, the configuration in Case 3 cannot be the start of a maximal row droop, and so it will not arise as the first modification within a pair of rows. 
			Thus, we may ignore this configuration.
			
			Case 4 requires additional analysis. 
			Looking leftward, each of the pipes in rows $r$ and $r+1$ of $\mathcal{B}_{k-1}$ must bend downward at some point. 
			If the top pipe bends first, our starting BPD or almost BPD would not have been reduced. 
			So the bottom pipe must bend first, leading to a configuration of the form:
			\[
			%
.
			\]
			This move preserves the co-permutation. 
			We note that there may be more, unpictured vertical pipes in the co-BPD, as in \cref{eqn:irreleventcolumns}, but the move is still co-permutation preserving by \cref{lemma:planarmoves} and \cref{lemma:planarhistoryverticalisok}.

			\noindent {\bf Configurations for $[r,r + 1] \times \{d\}$:}
			
			We now analyze the different configurations which can arise when the rightmost two tiles within $[r,r+1]\times[c,d]$ of $\mathcal{B}_{k-1}$ are modified.
			
			\[	\begin{tabular}{|c|c|c|c|} 
				\hline
				Case 1&Case 2&Case 3& Case 4\\
\hline
			\end{tabular}
		\]

Case 2 and Case 4 preserve co-permutations. Case 3 is not a valid configuration as next active elbow would still be in the first row of the active rectangle and we could have extended the row droop.

Case 1 requires additional analysis. There are three subcases to consider.
			
			\noindent {\bf Subcase 1:} Suppose when looking to the right within rows $r$ and $r + 1$ of $\mathcal{B}_{k-1}$, neither pipe bends. 
			In this case, there are no further cross-bump swaps and the algorithm terminates, so we must have $k = h$. 
			In particular, because $w(i)=j$ and the active pipe has label $j$, it must exit in row $i$, so $r=i$ and the active rectangle must sit in rows $i$ and $i + 1$. 
			Thus, $\mathcal{B}_h = \mathcal{B}' \in \bpd(ws_i)$.

			Because the pipes that exit in rows $i$ and $i + 1$ of $\check{\mathcal{B}}_{h-1}$ cross, from \cref{lemma:planarhistoryinversions}, we have $\demprod(\check{\mathcal{B}}_{h-1})$ has an ascent at $i$. 
			Removing this crossing could result in two cases for $\demprod(\check{\mathcal{B}}_{h})$: either it is $\demprod(\check{\mathcal{B}}_{h-1})$ or $\demprod(\check{\mathcal{B}}_{h-1})s_i$. In either case, we conclude that $\demprod(\check{\mathcal{B}'})\square \tau_i =\demprod(\check{\mathcal{B}_h})\square \tau_i= \demprod(\check{\mathcal{B}}_{h-1})$.
			
			\noindent {\bf Subcase 2:} If the bottom pipe bends upward first, we would have to do another cross-bump swap in the same two active rows, so this configuration does not occur.
			
			\noindent {\bf Subcase 3:} It could be that the top pipe bends upward first. In that case, we are in the situation below.
			
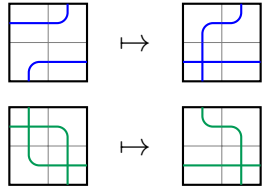
\[
			\begin{tikzpicture}[x = 1.25em,y = 1.25em]
				\draw[step = 1,gray, very thin] (0,0) grid (2, -2);
				\draw[color = black, thick] (0,0) rectangle (2, -2);
				\draw[thick,rounded corners,color = blue](0, -1/2)--(1, -1/2);
				\draw[thick,rounded corners,color = blue](3/2, 0)--(3/2, -1/2)--(1, -1/2);
				\draw[thick,rounded corners,color = blue](1/2, -2)--(1/2, -3/2)--(1, -3/2);
				\draw[thick,rounded corners,color = blue](1, -3/2)--(2, -3/2);
			\end{tikzpicture}
			\quad
			\raisebox{1em}{$\mapsto$}
			\quad
			\begin{tikzpicture}[x = 1.25em,y = 1.25em]
				\draw[step = 1,gray, very thin] (0,0) grid (2, -2);
				\draw[color = black, thick] (0,0) rectangle (2, -2);
				\draw[thick,rounded corners,color = blue](1/2, -1)--(1/2, -1/2)--(1, -1/2);
				\draw[thick,rounded corners,color = blue](3/2, 0)--(3/2, -1/2)--(1, -1/2);
				\draw[thick,rounded corners,color = blue](0, -3/2)--(1, -3/2);
				\draw[thick,rounded corners,color = blue](1/2, -1)--(1/2, -2);
				\draw[thick,rounded corners,color = blue](1, -3/2)--(2, -3/2);
			\end{tikzpicture}
			\]
			\[
			\begin{tikzpicture}[x = 1.25em,y = 1.25em]
				\draw[step = 1,gray, very thin] (0,0) grid (2, -2);
				\draw[color = black, thick] (0,0) rectangle (2, -2);
				\draw[thick,rounded corners,color = ForestGreen](0, -1/2)--(1, -1/2);
				\draw[thick,rounded corners,color = ForestGreen](1/2, 0)--(1/2, -1);
				\draw[thick,rounded corners,color = ForestGreen](3/2, -1)--(3/2, -1/2)--(1, -1/2);
				\draw[thick,rounded corners,color = ForestGreen](1/2, -1)--(1/2, -3/2)--(1, -3/2);
				\draw[thick,rounded corners,color = ForestGreen](1, -3/2)--(2, -3/2);
				\draw[thick,rounded corners,color = ForestGreen](3/2, -1)--(3/2, -2);
			\end{tikzpicture}
			\quad
			\raisebox{1em}{$\mapsto$}
			\quad
			\begin{tikzpicture}[x = 1.25em,y = 1.25em]
				\draw[step = 1,gray, very thin] (0,0) grid (2, -2);
				\draw[color = black, thick] (0,0) rectangle (2, -2);
				\draw[thick,rounded corners,color = ForestGreen](1/2, 0)--(1/2, -1/2)--(1, -1/2);
				\draw[thick,rounded corners,color = ForestGreen](3/2, -1)--(3/2, -1/2)--(1, -1/2);
				\draw[thick,rounded corners,color = ForestGreen](0, -3/2)--(1, -3/2);
				\draw[thick,rounded corners,color = ForestGreen](1, -3/2)--(2, -3/2);
				\draw[thick,rounded corners,color = ForestGreen](3/2, -1)--(3/2, -2);
			\end{tikzpicture}
			\]
			This preserves co-permutations.
			Note that there may be more, unpictured vertical pipes in the co-BPD, as in \cref{eqn:irreleventcolumns}, but the move is still co-permutation preserving by \cref{lemma:planarmoves} and \cref{lemma:planarhistoryverticalisok}.
			
			Now we have seen that only Subcase 1 has the potential to modify co-permutations. 
			Furthermore, if this subcase occurs, we must have $k=h$. 
			Thus we conclude that \[\demprod(\check{\mathcal{B}'})\square \tau_i = \demprod(\check{\mathcal{B}}_{h-1}) = \cdots = \demprod(\check{\mathcal{B}_0})=\demprod(\check{\mathcal{B}}). \qedhere\]
		\end{proof}

		\subsection{The Canonical bijection}

		Gao and Huang \cite{Gao.Huang} introduced a bijection between reduced pipe dreams and reduced bumpless pipe dreams. 
		Their bijection is canonical in the sense that it intertwines with certain natural bijective maps on pipe dreams and bumpless pipe dreams that provide combinatorial explanations for Monk's rule.  Furthermore, Gao and Huang showed that any bijection with this property is the canonical bijection.
		We will use the column-weight preserving version of their map, which differs from the original only by transposition of diagrams. 
		Let $\kappa:\pipes_n\rightarrow \bpd_n$ denote the column-weight preserving canonical bijection of Gao--Huang. Column-weight preserving means that
		\[\#\{i:(i,j) \in \mathcal{P}\} = \#\{i:(i,j) \in \rothe(\kappa(\mathcal{P}))\}\] for all $j \in [n]$.
		The map $\kappa$ is also permutation preserving, that is, $\delta(\mathcal{P}) = \delta(\kappa(\mathcal{P}))$.
		For a detailed definition of this map, we refer the reader to \cite{Gao.Huang}.

		For our purposes, the following lemma is enough.
		
		\begin{lemma}
			\label{lemma:monkcompatible}
			Let $w \in S_n$ with $w \neq w_0$. Set $j = \addable(w)$ and $i=w^{-1}(j)$. 
			Fix $\mathcal{B} \in \bpd(w)$. 
			Suppose $\mathcal{B} \mapsto \mathcal{B}'$ by insertion into column $j$ of $\mathcal{B}$. 
			Suppose $\mathcal{P}' \in \pipes_n$ satisfies $\kappa(\mathcal{P}') = \mathcal{B}'$. Then $\kappa({\mathcal{P}'}-\{(i,j)\}) = \mathcal{B}$.
		\end{lemma}
		\begin{proof}
			This follows from a special case of \cite[Theorem 4.5]{Gao.Huang}.
		\end{proof}
				
		We now restate and prove \cref{thm:coperm}.
		\coperm*
		\begin{proof}
			We will proceed by reverse induction on Bruhat order. Suppose $w=w_0$. In this case, there is a unique pipe dream and BPD for $w_0$ and the co-permutation of each of these is $w_0$.
			
			Now fix $w \in S_n$ with $w \neq w_0$. 
			Assume the statement holds for all permutations $v \in S_n$ so that $v > w$ in Bruhat order. 
			Fix $\mathcal{P} \in \pipes(w)$, and let $\mathcal{B} = \kappa(\mathcal{P})$. Let $j = \addable(w)$ and $i = w^{-1}(j)$. 
			
			Let $\mathcal{P}' = \mathcal{P} \cup \{(i, j)\}$, and let $\mathcal{B}' = \kappa(\mathcal{P}')$. Then, by \cref{lemma:monkcompatible}, $\mathcal{B}'$ is obtained from $\mathcal{B}$ by column insertion.
			We have $\demprod(\mathcal{P}') = \demprod(\mathcal{B}') > w$ so by the inductive hypothesis, $\demprod(\check{\mathcal{P}'}) = \demprod(\check{\mathcal{B}'})$.
			Thus,
			\begin{align*}
				\demprod(\check{\mathcal{P}})
				& = \demprod(\check{\mathcal{P'}}) \square \tau_i &\text{(by \cref{lemma:cotransitionflippermutation})}\\
				& = \demprod(\check{\mathcal{B'}}) \square \tau_i &\text{(by the inductive hypothesis)}\\
				& = \demprod(\check{\mathcal{B}}) & \text{(by \cref{lemma:columninsertionbpdcopermutations})}
			\end{align*}
			as desired.
		\end{proof}
		
		\begin{example}
			Let $w = 21543$. There are two reduced pipe dreams for $w$ that have the column-weight $(2,1,1,0,0)$. We list them, and their associated co-pipe dreams below.
			\[\raisebox{3.75em}{$\mathcal{P}=$}\,
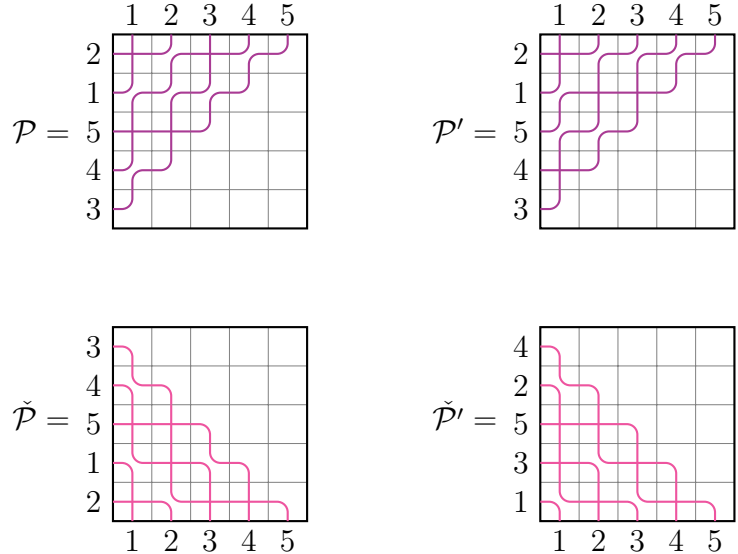
			\begin{tikzpicture}[x = 1.25em,y = 1.25em]
			\draw[step = 1,gray, very thin] (0,0) grid (5, -5);
			\draw[color = black, thick] (0,0) rectangle (5, -5);
			\draw[thick,rounded corners,color = Mulberry](0, -1/2)--(1, -1/2);
			\draw[thick,rounded corners,color = Mulberry](1/2, 0)--(1/2, -1);
			\draw[thick,rounded corners,color = Mulberry](3/2, 0)--(3/2, -1/2)--(1, -1/2);
			\draw[thick,rounded corners,color = Mulberry](2, -1/2)--(3/2, -1/2)--(3/2, -1);
			\draw[thick,rounded corners,color = Mulberry](2, -1/2)--(3, -1/2);
			\draw[thick,rounded corners,color = Mulberry](5/2, 0)--(5/2, -1);
			\draw[thick,rounded corners,color = Mulberry](7/2, 0)--(7/2, -1/2)--(3, -1/2);
			\draw[thick,rounded corners,color = Mulberry](4, -1/2)--(7/2, -1/2)--(7/2, -1);
			\draw[thick,rounded corners,color = Mulberry](9/2, 0)--(9/2, -1/2)--(4, -1/2);
			\draw[thick,rounded corners,color = Mulberry](1/2, -1)--(1/2, -3/2)--(0, -3/2);
			\draw[thick,rounded corners,color = Mulberry](1, -3/2)--(1/2, -3/2)--(1/2, -2);
			\draw[thick,rounded corners,color = Mulberry](3/2, -1)--(3/2, -3/2)--(1, -3/2);
			\draw[thick,rounded corners,color = Mulberry](2, -3/2)--(3/2, -3/2)--(3/2, -2);
			\draw[thick,rounded corners,color = Mulberry](5/2, -1)--(5/2, -3/2)--(2, -3/2);
			\draw[thick,rounded corners,color = Mulberry](3, -3/2)--(5/2, -3/2)--(5/2, -2);
			\draw[thick,rounded corners,color = Mulberry](7/2, -1)--(7/2, -3/2)--(3, -3/2);
			\draw[thick,rounded corners,color = Mulberry](0, -5/2)--(1, -5/2);
			\draw[thick,rounded corners,color = Mulberry](1/2, -2)--(1/2, -3);
			\draw[thick,rounded corners,color = Mulberry](1, -5/2)--(2, -5/2);
			\draw[thick,rounded corners,color = Mulberry](3/2, -2)--(3/2, -3);
			\draw[thick,rounded corners,color = Mulberry](5/2, -2)--(5/2, -5/2)--(2, -5/2);
			\draw[thick,rounded corners,color = Mulberry](1/2, -3)--(1/2, -7/2)--(0, -7/2);
			\draw[thick,rounded corners,color = Mulberry](1, -7/2)--(1/2, -7/2)--(1/2, -4);
			\draw[thick,rounded corners,color = Mulberry](3/2, -3)--(3/2, -7/2)--(1, -7/2);
			\draw[thick,rounded corners,color = Mulberry](1/2, -4)--(1/2, -9/2)--(0, -9/2);
			\node at (1/2, 1/2){1};
			\node at (-1/2, -1/2){2};
			\node at (3/2, 1/2){2};
			\node at (-1/2, -3/2){1};
			\node at (5/2, 1/2){3};
			\node at (-1/2, -5/2){5};
			\node at (7/2, 1/2){4};
			\node at (-1/2, -7/2){4};
			\node at (9/2, 1/2){5};
			\node at (-1/2, -9/2){3};
			\node at (1/2, -11/2){\null};
			\end{tikzpicture}
			\hspace{4em}
			\raisebox{3.75em}{$\mathcal{P}'=$}\,
			\begin{tikzpicture}[x = 1.25em,y = 1.25em]
			\draw[step = 1,gray, very thin] (0,0) grid (5, -5);
			\draw[color = black, thick] (0,0) rectangle (5, -5);
			\draw[thick,rounded corners,color = Mulberry](0, -1/2)--(1, -1/2);
			\draw[thick,rounded corners,color = Mulberry](1/2, 0)--(1/2, -1);
			\draw[thick,rounded corners,color = Mulberry](3/2, 0)--(3/2, -1/2)--(1, -1/2);
			\draw[thick,rounded corners,color = Mulberry](2, -1/2)--(3/2, -1/2)--(3/2, -1);
			\draw[thick,rounded corners,color = Mulberry](5/2, 0)--(5/2, -1/2)--(2, -1/2);
			\draw[thick,rounded corners,color = Mulberry](3, -1/2)--(5/2, -1/2)--(5/2, -1);
			\draw[thick,rounded corners,color = Mulberry](7/2, 0)--(7/2, -1/2)--(3, -1/2);
			\draw[thick,rounded corners,color = Mulberry](4, -1/2)--(7/2, -1/2)--(7/2, -1);
			\draw[thick,rounded corners,color = Mulberry](9/2, 0)--(9/2, -1/2)--(4, -1/2);
			\draw[thick,rounded corners,color = Mulberry](1/2, -1)--(1/2, -3/2)--(0, -3/2);
			\draw[thick,rounded corners,color = Mulberry](1, -3/2)--(1/2, -3/2)--(1/2, -2);
			\draw[thick,rounded corners,color = Mulberry](1, -3/2)--(2, -3/2);
			\draw[thick,rounded corners,color = Mulberry](3/2, -1)--(3/2, -2);
			\draw[thick,rounded corners,color = Mulberry](2, -3/2)--(3, -3/2);
			\draw[thick,rounded corners,color = Mulberry](5/2, -1)--(5/2, -2);
			\draw[thick,rounded corners,color = Mulberry](7/2, -1)--(7/2, -3/2)--(3, -3/2);
			\draw[thick,rounded corners,color = Mulberry](1/2, -2)--(1/2, -5/2)--(0, -5/2);
			\draw[thick,rounded corners,color = Mulberry](1, -5/2)--(1/2, -5/2)--(1/2, -3);
			\draw[thick,rounded corners,color = Mulberry](3/2, -2)--(3/2, -5/2)--(1, -5/2);
			\draw[thick,rounded corners,color = Mulberry](2, -5/2)--(3/2, -5/2)--(3/2, -3);
			\draw[thick,rounded corners,color = Mulberry](5/2, -2)--(5/2, -5/2)--(2, -5/2);
			\draw[thick,rounded corners,color = Mulberry](0, -7/2)--(1, -7/2);
			\draw[thick,rounded corners,color = Mulberry](1/2, -3)--(1/2, -4);
			\draw[thick,rounded corners,color = Mulberry](3/2, -3)--(3/2, -7/2)--(1, -7/2);
			\draw[thick,rounded corners,color = Mulberry](1/2, -4)--(1/2, -9/2)--(0, -9/2);
			\node at (1/2, -11/2){\null};
			\node at (1/2, 1/2){1};
			\node at (-1/2, -1/2){2};
			\node at (3/2, 1/2){2};
			\node at (-1/2, -3/2){1};
			\node at (5/2, 1/2){3};
			\node at (-1/2, -5/2){5};
			\node at (7/2, 1/2){4};
			\node at (-1/2, -7/2){4};
			\node at (9/2, 1/2){5};
			\node at (-1/2, -9/2){3};
			\end{tikzpicture}
			\]
			\[		
			\raisebox{4.25em}{$\check{\mathcal{P}}=$}\,		
			\begin{tikzpicture}[x = 1.25em,y = 1.25em]
			\draw[step = 1,gray, very thin] (0,0) grid (5, -5);
			\draw[color = black, thick] (0,0) rectangle (5, -5);
			\draw[thick,rounded corners,color = Rhodamine](1/2, -1)--(1/2, -1/2)--(0, -1/2);
			\draw[thick,rounded corners,color = Rhodamine](1/2, -1)--(1/2, -3/2)--(1, -3/2);
			\draw[thick,rounded corners,color = Rhodamine](0, -3/2)--(1/2, -3/2)--(1/2, -2);
			\draw[thick,rounded corners,color = Rhodamine](3/2, -2)--(3/2, -3/2)--(1, -3/2);
			\draw[thick,rounded corners,color = Rhodamine](0, -5/2)--(1, -5/2);
			\draw[thick,rounded corners,color = Rhodamine](1/2, -2)--(1/2, -3);
			\draw[thick,rounded corners,color = Rhodamine](1, -5/2)--(2, -5/2);
			\draw[thick,rounded corners,color = Rhodamine](3/2, -2)--(3/2, -3);
			\draw[thick,rounded corners,color = Rhodamine](5/2, -3)--(5/2, -5/2)--(2, -5/2);
			\draw[thick,rounded corners,color = Rhodamine](1/2, -3)--(1/2, -7/2)--(1, -7/2);
			\draw[thick,rounded corners,color = Rhodamine](0, -7/2)--(1/2, -7/2)--(1/2, -4);
			\draw[thick,rounded corners,color = Rhodamine](1, -7/2)--(2, -7/2);
			\draw[thick,rounded corners,color = Rhodamine](3/2, -3)--(3/2, -4);
			\draw[thick,rounded corners,color = Rhodamine](5/2, -3)--(5/2, -7/2)--(3, -7/2);
			\draw[thick,rounded corners,color = Rhodamine](2, -7/2)--(5/2, -7/2)--(5/2, -4);
			\draw[thick,rounded corners,color = Rhodamine](7/2, -4)--(7/2, -7/2)--(3, -7/2);
			\draw[thick,rounded corners,color = Rhodamine](0, -9/2)--(1, -9/2);
			\draw[thick,rounded corners,color = Rhodamine](1/2, -4)--(1/2, -5);
			\draw[thick,rounded corners,color = Rhodamine](3/2, -4)--(3/2, -9/2)--(2, -9/2);
			\draw[thick,rounded corners,color = Rhodamine](1, -9/2)--(3/2, -9/2)--(3/2, -5);
			\draw[thick,rounded corners,color = Rhodamine](2, -9/2)--(3, -9/2);
			\draw[thick,rounded corners,color = Rhodamine](5/2, -4)--(5/2, -5);
			\draw[thick,rounded corners,color = Rhodamine](3, -9/2)--(4, -9/2);
			\draw[thick,rounded corners,color = Rhodamine](7/2, -4)--(7/2, -5);
			\draw[thick,rounded corners,color = Rhodamine](9/2, -5)--(9/2, -9/2)--(4, -9/2);
			\node at (1/2, 1/2){\null};
			\node at (1/2, -11/2){1};
			\node at (-1/2, -1/2){3};
			\node at (3/2, -11/2){2};
			\node at (-1/2, -3/2){4};
			\node at (5/2, -11/2){3};
			\node at (-1/2, -5/2){5};
			\node at (7/2, -11/2){4};
			\node at (-1/2, -7/2){1};
			\node at (9/2, -11/2){5};
			\node at (-1/2, -9/2){2};
			\end{tikzpicture}
			\hspace{4em}
			\raisebox{4.25em}{$\check{\mathcal{P}'}=$}\,	
			\begin{tikzpicture}[x = 1.25em,y = 1.25em]
			\draw[step = 1,gray, very thin] (0,0) grid (5, -5);
			\draw[color = black, thick] (0,0) rectangle (5, -5);
			\draw[thick,rounded corners,color = Rhodamine](1/2, -1)--(1/2, -1/2)--(0, -1/2);
			\draw[thick,rounded corners,color = Rhodamine](1/2, -1)--(1/2, -3/2)--(1, -3/2);
			\draw[thick,rounded corners,color = Rhodamine](0, -3/2)--(1/2, -3/2)--(1/2, -2);
			\draw[thick,rounded corners,color = Rhodamine](3/2, -2)--(3/2, -3/2)--(1, -3/2);
			\draw[thick,rounded corners,color = Rhodamine](0, -5/2)--(1, -5/2);
			\draw[thick,rounded corners,color = Rhodamine](1/2, -2)--(1/2, -3);
			\draw[thick,rounded corners,color = Rhodamine](1, -5/2)--(2, -5/2);
			\draw[thick,rounded corners,color = Rhodamine](3/2, -2)--(3/2, -3);
			\draw[thick,rounded corners,color = Rhodamine](5/2, -3)--(5/2, -5/2)--(2, -5/2);
			\draw[thick,rounded corners,color = Rhodamine](0, -7/2)--(1, -7/2);
			\draw[thick,rounded corners,color = Rhodamine](1/2, -3)--(1/2, -4);
			\draw[thick,rounded corners,color = Rhodamine](3/2, -3)--(3/2, -7/2)--(2, -7/2);
			\draw[thick,rounded corners,color = Rhodamine](1, -7/2)--(3/2, -7/2)--(3/2, -4);
			\draw[thick,rounded corners,color = Rhodamine](2, -7/2)--(3, -7/2);
			\draw[thick,rounded corners,color = Rhodamine](5/2, -3)--(5/2, -4);
			\draw[thick,rounded corners,color = Rhodamine](7/2, -4)--(7/2, -7/2)--(3, -7/2);
			\draw[thick,rounded corners,color = Rhodamine](1/2, -4)--(1/2, -9/2)--(1, -9/2);
			\draw[thick,rounded corners,color = Rhodamine](0, -9/2)--(1/2, -9/2)--(1/2, -5);
			\draw[thick,rounded corners,color = Rhodamine](1, -9/2)--(2, -9/2);
			\draw[thick,rounded corners,color = Rhodamine](3/2, -4)--(3/2, -5);
			\draw[thick,rounded corners,color = Rhodamine](5/2, -4)--(5/2, -9/2)--(3, -9/2);
			\draw[thick,rounded corners,color = Rhodamine](2, -9/2)--(5/2, -9/2)--(5/2, -5);
			\draw[thick,rounded corners,color = Rhodamine](3, -9/2)--(4, -9/2);
			\draw[thick,rounded corners,color = Rhodamine](7/2, -4)--(7/2, -5);
			\draw[thick,rounded corners,color = Rhodamine](9/2, -5)--(9/2, -9/2)--(4, -9/2);
			\node at (1/2, 1/2){\null};
			\node at (1/2, -11/2){1};
			\node at (-1/2, -1/2){4};
			\node at (3/2, -11/2){2};
			\node at (-1/2, -3/2){2};
			\node at (5/2, -11/2){3};
			\node at (-1/2, -5/2){5};
			\node at (7/2, -11/2){4};
			\node at (-1/2, -7/2){3};
			\node at (9/2, -11/2){5};
			\node at (-1/2, -9/2){1};
			\end{tikzpicture}
			\]
			We now list the reduced BPDs of $w$ with the same column weight, and their associated co-BPDs.
			\[
			\raisebox{4.25em}{$\mathcal{B}=$}\,
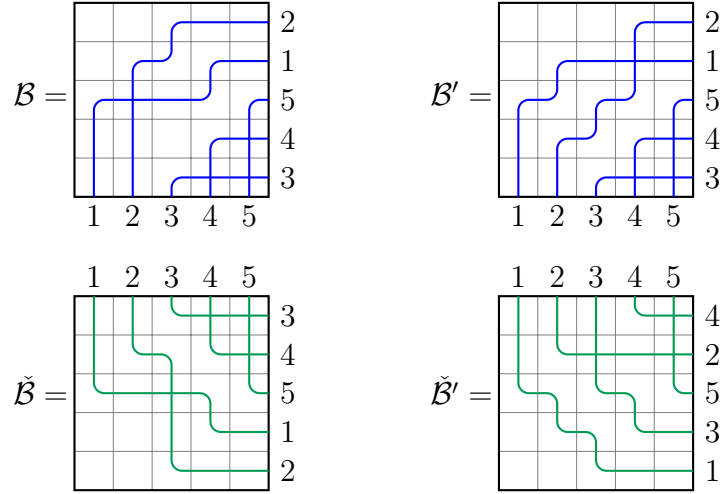
			\begin{tikzpicture}[x = 1.25em,y = 1.25em]
			\draw[step = 1,gray, very thin] (0,0) grid (5, -5);
			\draw[color = black, thick] (0,0) rectangle (5, -5);
			\draw[thick,rounded corners,color = blue](5/2, -1)--(5/2, -1/2)--(3, -1/2);
			\draw[thick,rounded corners,color = blue](3, -1/2)--(4, -1/2);
			\draw[thick,rounded corners,color = blue](4, -1/2)--(5, -1/2);
			\draw[thick,rounded corners,color = blue](3/2, -2)--(3/2, -3/2)--(2, -3/2);
			\draw[thick,rounded corners,color = blue](5/2, -1)--(5/2, -3/2)--(2, -3/2);
			\draw[thick,rounded corners,color = blue](7/2, -2)--(7/2, -3/2)--(4, -3/2);
			\draw[thick,rounded corners,color = blue](4, -3/2)--(5, -3/2);
			\draw[thick,rounded corners,color = blue](1/2, -3)--(1/2, -5/2)--(1, -5/2);
			\draw[thick,rounded corners,color = blue](1, -5/2)--(2, -5/2);
			\draw[thick,rounded corners,color = blue](3/2, -2)--(3/2, -3);
			\draw[thick,rounded corners,color = blue](2, -5/2)--(3, -5/2);
			\draw[thick,rounded corners,color = blue](7/2, -2)--(7/2, -5/2)--(3, -5/2);
			\draw[thick,rounded corners,color = blue](9/2, -3)--(9/2, -5/2)--(5, -5/2);
			\draw[thick,rounded corners,color = blue](1/2, -3)--(1/2, -4);
			\draw[thick,rounded corners,color = blue](3/2, -3)--(3/2, -4);
			\draw[thick,rounded corners,color = blue](7/2, -4)--(7/2, -7/2)--(4, -7/2);
			\draw[thick,rounded corners,color = blue](4, -7/2)--(5, -7/2);
			\draw[thick,rounded corners,color = blue](9/2, -3)--(9/2, -4);
			\draw[thick,rounded corners,color = blue](1/2, -4)--(1/2, -5);
			\draw[thick,rounded corners,color = blue](3/2, -4)--(3/2, -5);
			\draw[thick,rounded corners,color = blue](5/2, -5)--(5/2, -9/2)--(3, -9/2);
			\draw[thick,rounded corners,color = blue](3, -9/2)--(4, -9/2);
			\draw[thick,rounded corners,color = blue](7/2, -4)--(7/2, -5);
			\draw[thick,rounded corners,color = blue](4, -9/2)--(5, -9/2);
			\draw[thick,rounded corners,color = blue](9/2, -4)--(9/2, -5);
			\node at (1/2, 1/2){\null};
			\node at (1/2, -11/2){1};
			\node at (11/2, -1/2){2};
			\node at (3/2, -11/2){2};
			\node at (11/2, -3/2){1};
			\node at (5/2, -11/2){3};
			\node at (11/2, -5/2){5};
			\node at (7/2, -11/2){4};
			\node at (11/2, -7/2){4};
			\node at (9/2, -11/2){5};
			\node at (11/2, -9/2){3};
			\end{tikzpicture}
			\hspace{4em}
			\raisebox{4.25em}{$\mathcal{B'}=$}\,
			\begin{tikzpicture}[x = 1.25em,y = 1.25em]
			\draw[step = 1,gray, very thin] (0,0) grid (5, -5);
			\draw[color = black, thick] (0,0) rectangle (5, -5);
			\draw[thick,rounded corners,color = blue](7/2, -1)--(7/2, -1/2)--(4, -1/2);
			\draw[thick,rounded corners,color = blue](4, -1/2)--(5, -1/2);
			\draw[thick,rounded corners,color = blue](3/2, -2)--(3/2, -3/2)--(2, -3/2);
			\draw[thick,rounded corners,color = blue](2, -3/2)--(3, -3/2);
			\draw[thick,rounded corners,color = blue](3, -3/2)--(4, -3/2);
			\draw[thick,rounded corners,color = blue](7/2, -1)--(7/2, -2);
			\draw[thick,rounded corners,color = blue](4, -3/2)--(5, -3/2);
			\draw[thick,rounded corners,color = blue](1/2, -3)--(1/2, -5/2)--(1, -5/2);
			\draw[thick,rounded corners,color = blue](3/2, -2)--(3/2, -5/2)--(1, -5/2);
			\draw[thick,rounded corners,color = blue](5/2, -3)--(5/2, -5/2)--(3, -5/2);
			\draw[thick,rounded corners,color = blue](7/2, -2)--(7/2, -5/2)--(3, -5/2);
			\draw[thick,rounded corners,color = blue](9/2, -3)--(9/2, -5/2)--(5, -5/2);
			\draw[thick,rounded corners,color = blue](1/2, -3)--(1/2, -4);
			\draw[thick,rounded corners,color = blue](3/2, -4)--(3/2, -7/2)--(2, -7/2);
			\draw[thick,rounded corners,color = blue](5/2, -3)--(5/2, -7/2)--(2, -7/2);
			\draw[thick,rounded corners,color = blue](7/2, -4)--(7/2, -7/2)--(4, -7/2);
			\draw[thick,rounded corners,color = blue](4, -7/2)--(5, -7/2);
			\draw[thick,rounded corners,color = blue](9/2, -3)--(9/2, -4);
			\draw[thick,rounded corners,color = blue](1/2, -4)--(1/2, -5);
			\draw[thick,rounded corners,color = blue](3/2, -4)--(3/2, -5);
			\draw[thick,rounded corners,color = blue](5/2, -5)--(5/2, -9/2)--(3, -9/2);
			\draw[thick,rounded corners,color = blue](3, -9/2)--(4, -9/2);
			\draw[thick,rounded corners,color = blue](7/2, -4)--(7/2, -5);
			\draw[thick,rounded corners,color = blue](4, -9/2)--(5, -9/2);
			\draw[thick,rounded corners,color = blue](9/2, -4)--(9/2, -5);
			\node at (1/2, 1/2){\null};
			\node at (1/2, -11/2){1};
			\node at (11/2, -1/2){2};
			\node at (3/2, -11/2){2};
			\node at (11/2, -3/2){1};
			\node at (5/2, -11/2){3};
			\node at (11/2, -5/2){5};
			\node at (7/2, -11/2){4};
			\node at (11/2, -7/2){4};
			\node at (9/2, -11/2){5};
			\node at (11/2, -9/2){3};
			\end{tikzpicture}
			\]
			\[
			\raisebox{3.75em}{$\check{\mathcal{B}}=$}\,	
			\begin{tikzpicture}[x = 1.25em,y = 1.25em]
			\draw[step = 1,gray, very thin] (0,0) grid (5, -5);
			\draw[color = black, thick] (0,0) rectangle (5, -5);
			\draw[thick,rounded corners,color = ForestGreen](1/2, 0)--(1/2, -1);
			\draw[thick,rounded corners,color = ForestGreen](3/2, 0)--(3/2, -1);
			\draw[thick,rounded corners,color = ForestGreen](5/2, 0)--(5/2, -1/2)--(3, -1/2);
			\draw[thick,rounded corners,color = ForestGreen](3, -1/2)--(4, -1/2);
			\draw[thick,rounded corners,color = ForestGreen](7/2, 0)--(7/2, -1);
			\draw[thick,rounded corners,color = ForestGreen](4, -1/2)--(5, -1/2);
			\draw[thick,rounded corners,color = ForestGreen](9/2, 0)--(9/2, -1);
			\draw[thick,rounded corners,color = ForestGreen](1/2, -1)--(1/2, -2);
			\draw[thick,rounded corners,color = ForestGreen](3/2, -1)--(3/2, -3/2)--(2, -3/2);
			\draw[thick,rounded corners,color = ForestGreen](5/2, -2)--(5/2, -3/2)--(2, -3/2);
			\draw[thick,rounded corners,color = ForestGreen](7/2, -1)--(7/2, -3/2)--(4, -3/2);
			\draw[thick,rounded corners,color = ForestGreen](4, -3/2)--(5, -3/2);
			\draw[thick,rounded corners,color = ForestGreen](9/2, -1)--(9/2, -2);
			\draw[thick,rounded corners,color = ForestGreen](1/2, -2)--(1/2, -5/2)--(1, -5/2);
			\draw[thick,rounded corners,color = ForestGreen](1, -5/2)--(2, -5/2);
			\draw[thick,rounded corners,color = ForestGreen](2, -5/2)--(3, -5/2);
			\draw[thick,rounded corners,color = ForestGreen](5/2, -2)--(5/2, -3);
			\draw[thick,rounded corners,color = ForestGreen](7/2, -3)--(7/2, -5/2)--(3, -5/2);
			\draw[thick,rounded corners,color = ForestGreen](9/2, -2)--(9/2, -5/2)--(5, -5/2);
			\draw[thick,rounded corners,color = ForestGreen](5/2, -3)--(5/2, -4);
			\draw[thick,rounded corners,color = ForestGreen](7/2, -3)--(7/2, -7/2)--(4, -7/2);
			\draw[thick,rounded corners,color = ForestGreen](4, -7/2)--(5, -7/2);
			\draw[thick,rounded corners,color = ForestGreen](5/2, -4)--(5/2, -9/2)--(3, -9/2);
			\draw[thick,rounded corners,color = ForestGreen](3, -9/2)--(4, -9/2);
			\draw[thick,rounded corners,color = ForestGreen](4, -9/2)--(5, -9/2);
			\node at (1/2, -11/2){\null};
			\node at (1/2, 1/2){1};
			\node at (11/2, -1/2){3};
			\node at (3/2, 1/2){2};
			\node at (11/2, -3/2){4};
			\node at (5/2, 1/2){3};
			\node at (11/2, -5/2){5};
			\node at (7/2, 1/2){4};
			\node at (11/2, -7/2){1};
			\node at (9/2, 1/2){5};
			\node at (11/2, -9/2){2};
			\end{tikzpicture}
			\hspace{4em}
			\raisebox{3.75em}{$\check{\mathcal{B}'}=$}\,	
			\begin{tikzpicture}[x = 1.25em,y = 1.25em]
			\draw[step = 1,gray, very thin] (0,0) grid (5, -5);
			\draw[color = black, thick] (0,0) rectangle (5, -5);
			\draw[thick,rounded corners,color = ForestGreen](1/2, 0)--(1/2, -1);
			\draw[thick,rounded corners,color = ForestGreen](3/2, 0)--(3/2, -1);
			\draw[thick,rounded corners,color = ForestGreen](5/2, 0)--(5/2, -1);
			\draw[thick,rounded corners,color = ForestGreen](7/2, 0)--(7/2, -1/2)--(4, -1/2);
			\draw[thick,rounded corners,color = ForestGreen](4, -1/2)--(5, -1/2);
			\draw[thick,rounded corners,color = ForestGreen](9/2, 0)--(9/2, -1);
			\draw[thick,rounded corners,color = ForestGreen](1/2, -1)--(1/2, -2);
			\draw[thick,rounded corners,color = ForestGreen](3/2, -1)--(3/2, -3/2)--(2, -3/2);
			\draw[thick,rounded corners,color = ForestGreen](2, -3/2)--(3, -3/2);
			\draw[thick,rounded corners,color = ForestGreen](5/2, -1)--(5/2, -2);
			\draw[thick,rounded corners,color = ForestGreen](3, -3/2)--(4, -3/2);
			\draw[thick,rounded corners,color = ForestGreen](4, -3/2)--(5, -3/2);
			\draw[thick,rounded corners,color = ForestGreen](9/2, -1)--(9/2, -2);
			\draw[thick,rounded corners,color = ForestGreen](1/2, -2)--(1/2, -5/2)--(1, -5/2);
			\draw[thick,rounded corners,color = ForestGreen](3/2, -3)--(3/2, -5/2)--(1, -5/2);
			\draw[thick,rounded corners,color = ForestGreen](5/2, -2)--(5/2, -5/2)--(3, -5/2);
			\draw[thick,rounded corners,color = ForestGreen](7/2, -3)--(7/2, -5/2)--(3, -5/2);
			\draw[thick,rounded corners,color = ForestGreen](9/2, -2)--(9/2, -5/2)--(5, -5/2);
			\draw[thick,rounded corners,color = ForestGreen](3/2, -3)--(3/2, -7/2)--(2, -7/2);
			\draw[thick,rounded corners,color = ForestGreen](5/2, -4)--(5/2, -7/2)--(2, -7/2);
			\draw[thick,rounded corners,color = ForestGreen](7/2, -3)--(7/2, -7/2)--(4, -7/2);
			\draw[thick,rounded corners,color = ForestGreen](4, -7/2)--(5, -7/2);
			\draw[thick,rounded corners,color = ForestGreen](5/2, -4)--(5/2, -9/2)--(3, -9/2);
			\draw[thick,rounded corners,color = ForestGreen](3, -9/2)--(4, -9/2);
			\draw[thick,rounded corners,color = ForestGreen](4, -9/2)--(5, -9/2);
			\node at (1/2, -11/2){\null};
			\node at (1/2, 1/2){1};
			\node at (11/2, -1/2){4};
			\node at (3/2, 1/2){2};
			\node at (11/2, -3/2){2};
			\node at (5/2, 1/2){3};
			\node at (11/2, -5/2){5};
			\node at (7/2, 1/2){4};
			\node at (11/2, -7/2){3};
			\node at (9/2, 1/2){5};
			\node at (11/2, -9/2){1};
			\end{tikzpicture}\]
			By looking at co-permutations and applying \cref{thm:coperm}, we conclude that $\kappa(\mathcal{P})=\mathcal{B}$ and $\kappa(\mathcal{P}')=\mathcal{B}'$.
		\end{example}
		
		\begin{remark}
		Huang, Shimozono, and Yu \cite{Huang.Shimozono.Yu} generalized \cite{Gao.Huang} to give a weight preserving bijection between marked bumpless pipe dreams and pipe dreams.
		It is likely that considering co-objects would give further insight into this K-theoretic bijection, but we do not pursue this direction here.
		\end{remark}
				
		\subsection{Proofs of the bumpless pipe dream change of basis formulas}
		
		We are now ready to prove our main theorems. 
		We begin with the BPD formula for expanding Grothendieck polynomials in the Schubert basis.
		
		\thmmainA*
		
		\begin{proof}
			Let $w \in S_n$. Define $a_{w,v}$ to be the coefficients in the expansion
			\[
			\mathfrak{G}_w = \sum_{v \in S_n} (-1)^{\ell(v) - \ell(w)} a_{w,v} \mathfrak{S}_v.
			\]
			We have 
			\begin{align*}
				a_{w,v} 
				&= \#\{ \mathcal{P} \in \kpipes(w) : \check{\mathcal{P}} \text{ is reduced and } \demprod(\check{\mathcal{P}}) = v \} & \text{(by \cref{thm:pipe_groth_to_schub})} \\
				&= \#\{ \mathcal{P} \in \kpipes(w) : \check{\mathcal{P}} \in \copipes(v) \} &\text{(by definition)} \\
				&= \#\{ \mathcal{B} \in \kbpd(w) : \check{\mathcal{B}} \in \cobpd(v) \} & \text{(by \cref{thm:coperm} and symmetry)} \\
				&=\#\{\mathcal{B} \in \kbpd(w) : \check{\mathcal{B}} \text{ is reduced and } \demprod(\check{\mathcal{B}}) = v\} &\text{(by definition).}
			\end{align*}
		Thus, the result follows.
		\end{proof}

		We now prove the BPD formula for Schubert to Grothendieck expansion.
		\thmmainB*
		\begin{proof}
			Let $w \in S_n$. Define $b_{w,v}$ to be the coefficients in the expansion
			\[
			\mathfrak{S}_w = \sum_{v \in S_n} b_{w,v} \mathfrak{G}_v.
			\]
			We have
			\begin{align*}
				b_{w,v}
				&= \#\{ \mathcal{P} \in \pipes(w) : \demprod(\check{\mathcal{P}}) = v \} &\text{(by \cref{thm:pipe_schub_to_groth})} \\
				&=\#\{ \mathcal{B} \in \bpd(w) : \demprod(\check{\mathcal{B}}) = v \} & \text{(by \cref{thm:coperm}).} 
			\end{align*}
		Thus, the result follows.
		\end{proof}
		
		We conclude this section with examples of the BPD change of basis formulas.

		\begin{example}\label{example:bpd21534}
			
			Let $w = 21534$. The BPDs of $w$ are pictured below.
			\[

			\]
			
			The corresponding co-BPDs, listed in the same order as their BPDs, are below.
			\[
			%
			\]
			
			Of the co-BPDs, the first, second, third, and fourth are reduced. 
			Thus, \cref{thm:main1} says that 
			\[\mathfrak{G}_{21534} = \mathfrak{S}_{21534}-\mathfrak{S}_{31524}-\mathfrak{S}_{23514} + \mathfrak{S}_{32514}.\]
			Of the BPDs, the fourth and eighth are not reduced. 
			Therefore, by \cref{thm:main2},
			\[\mathfrak{S}_{21534} = \mathfrak{G}_{21534} + \mathfrak{G}_{31524} + \mathfrak{G}_{23514} + \mathfrak{G}_{41523} + \mathfrak{G}_{24513} + \mathfrak{G}_{34512}. \qedhere\]
		\end{example}	
	
	We give one additional example, to show that the expansions need not be multiplicity free.
	
	\begin{example}
		\label{ex:13452}
		Let $w=13452$. The BPDs of $w$ are pictured below.
		
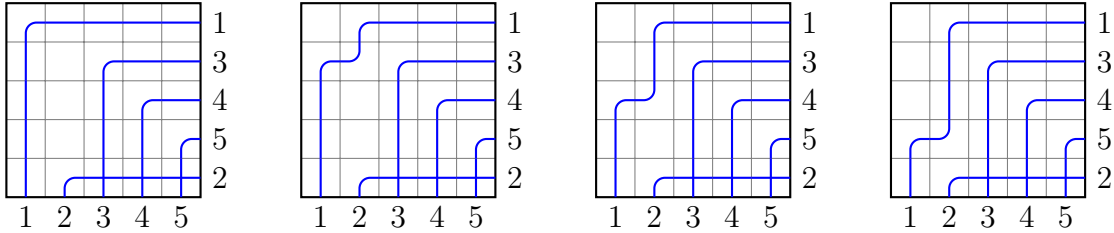
\[\begin{tikzpicture}[x=1.25em,y=1.25em]
			\draw[step=1,gray, thin] (0,0) grid (5, -5);
			\draw[color=black, thick] (0,0) rectangle (5, -5);
			\draw[thick,rounded corners,color=blue](1/2, -1)--(1/2, -1/2)--(1, -1/2);
			\draw[thick,rounded corners,color=blue](1, -1/2)--(2, -1/2);
			\draw[thick,rounded corners,color=blue](2, -1/2)--(3, -1/2);
			\draw[thick,rounded corners,color=blue](3, -1/2)--(4, -1/2);
			\draw[thick,rounded corners,color=blue](4, -1/2)--(5, -1/2);
			\draw[thick,rounded corners,color=blue](1/2, -1)--(1/2, -2);
			\draw[thick,rounded corners,color=blue](5/2, -2)--(5/2, -3/2)--(3, -3/2);
			\draw[thick,rounded corners,color=blue](3, -3/2)--(4, -3/2);
			\draw[thick,rounded corners,color=blue](4, -3/2)--(5, -3/2);
			\draw[thick,rounded corners,color=blue](1/2, -2)--(1/2, -3);
			\draw[thick,rounded corners,color=blue](5/2, -2)--(5/2, -3);
			\draw[thick,rounded corners,color=blue](7/2, -3)--(7/2, -5/2)--(4, -5/2);
			\draw[thick,rounded corners,color=blue](4, -5/2)--(5, -5/2);
			\draw[thick,rounded corners,color=blue](1/2, -3)--(1/2, -4);
			\draw[thick,rounded corners,color=blue](5/2, -3)--(5/2, -4);
			\draw[thick,rounded corners,color=blue](7/2, -3)--(7/2, -4);
			\draw[thick,rounded corners,color=blue](9/2, -4)--(9/2, -7/2)--(5, -7/2);
			\draw[thick,rounded corners,color=blue](1/2, -4)--(1/2, -5);
			\draw[thick,rounded corners,color=blue](3/2, -5)--(3/2, -9/2)--(2, -9/2);
			\draw[thick,rounded corners,color=blue](2, -9/2)--(3, -9/2);
			\draw[thick,rounded corners,color=blue](5/2, -4)--(5/2, -5);
			\draw[thick,rounded corners,color=blue](3, -9/2)--(4, -9/2);
			\draw[thick,rounded corners,color=blue](7/2, -4)--(7/2, -5);
			\draw[thick,rounded corners,color=blue](4, -9/2)--(5, -9/2);
			\draw[thick,rounded corners,color=blue](9/2, -4)--(9/2, -5);
			\node at (1/2, -11/2){1};
			\node at (11/2, -1/2){1};
			\node at (3/2, -11/2){2};
			\node at (11/2, -3/2){3};
			\node at (5/2, -11/2){3};
			\node at (11/2, -5/2){4};
			\node at (7/2, -11/2){4};
			\node at (11/2, -7/2){5};
			\node at (9/2, -11/2){5};
			\node at (11/2, -9/2){2};
		\end{tikzpicture}
		\hspace{2em}
		\begin{tikzpicture}[x=1.25em,y=1.25em]
			\draw[step=1,gray, thin] (0,0) grid (5, -5);
			\draw[color=black, thick] (0,0) rectangle (5, -5);
			\draw[thick,rounded corners,color=blue](3/2, -1)--(3/2, -1/2)--(2, -1/2);
			\draw[thick,rounded corners,color=blue](2, -1/2)--(3, -1/2);
			\draw[thick,rounded corners,color=blue](3, -1/2)--(4, -1/2);
			\draw[thick,rounded corners,color=blue](4, -1/2)--(5, -1/2);
			\draw[thick,rounded corners,color=blue](1/2, -2)--(1/2, -3/2)--(1, -3/2);
			\draw[thick,rounded corners,color=blue](3/2, -1)--(3/2, -3/2)--(1, -3/2);
			\draw[thick,rounded corners,color=blue](5/2, -2)--(5/2, -3/2)--(3, -3/2);
			\draw[thick,rounded corners,color=blue](3, -3/2)--(4, -3/2);
			\draw[thick,rounded corners,color=blue](4, -3/2)--(5, -3/2);
			\draw[thick,rounded corners,color=blue](1/2, -2)--(1/2, -3);
			\draw[thick,rounded corners,color=blue](5/2, -2)--(5/2, -3);
			\draw[thick,rounded corners,color=blue](7/2, -3)--(7/2, -5/2)--(4, -5/2);
			\draw[thick,rounded corners,color=blue](4, -5/2)--(5, -5/2);
			\draw[thick,rounded corners,color=blue](1/2, -3)--(1/2, -4);
			\draw[thick,rounded corners,color=blue](5/2, -3)--(5/2, -4);
			\draw[thick,rounded corners,color=blue](7/2, -3)--(7/2, -4);
			\draw[thick,rounded corners,color=blue](9/2, -4)--(9/2, -7/2)--(5, -7/2);
			\draw[thick,rounded corners,color=blue](1/2, -4)--(1/2, -5);
			\draw[thick,rounded corners,color=blue](3/2, -5)--(3/2, -9/2)--(2, -9/2);
			\draw[thick,rounded corners,color=blue](2, -9/2)--(3, -9/2);
			\draw[thick,rounded corners,color=blue](5/2, -4)--(5/2, -5);
			\draw[thick,rounded corners,color=blue](3, -9/2)--(4, -9/2);
			\draw[thick,rounded corners,color=blue](7/2, -4)--(7/2, -5);
			\draw[thick,rounded corners,color=blue](4, -9/2)--(5, -9/2);
			\draw[thick,rounded corners,color=blue](9/2, -4)--(9/2, -5);
			\node at (1/2, -11/2){1};
			\node at (11/2, -1/2){1};
			\node at (3/2, -11/2){2};
			\node at (11/2, -3/2){3};
			\node at (5/2, -11/2){3};
			\node at (11/2, -5/2){4};
			\node at (7/2, -11/2){4};
			\node at (11/2, -7/2){5};
			\node at (9/2, -11/2){5};
			\node at (11/2, -9/2){2};
		\end{tikzpicture}
		\hspace{2em}
		\begin{tikzpicture}[x=1.25em,y=1.25em]
			\draw[step=1,gray, thin] (0,0) grid (5, -5);
			\draw[color=black, thick] (0,0) rectangle (5, -5);
			\draw[thick,rounded corners,color=blue](3/2, -1)--(3/2, -1/2)--(2, -1/2);
			\draw[thick,rounded corners,color=blue](2, -1/2)--(3, -1/2);
			\draw[thick,rounded corners,color=blue](3, -1/2)--(4, -1/2);
			\draw[thick,rounded corners,color=blue](4, -1/2)--(5, -1/2);
			\draw[thick,rounded corners,color=blue](3/2, -1)--(3/2, -2);
			\draw[thick,rounded corners,color=blue](5/2, -2)--(5/2, -3/2)--(3, -3/2);
			\draw[thick,rounded corners,color=blue](3, -3/2)--(4, -3/2);
			\draw[thick,rounded corners,color=blue](4, -3/2)--(5, -3/2);
			\draw[thick,rounded corners,color=blue](1/2, -3)--(1/2, -5/2)--(1, -5/2);
			\draw[thick,rounded corners,color=blue](3/2, -2)--(3/2, -5/2)--(1, -5/2);
			\draw[thick,rounded corners,color=blue](5/2, -2)--(5/2, -3);
			\draw[thick,rounded corners,color=blue](7/2, -3)--(7/2, -5/2)--(4, -5/2);
			\draw[thick,rounded corners,color=blue](4, -5/2)--(5, -5/2);
			\draw[thick,rounded corners,color=blue](1/2, -3)--(1/2, -4);
			\draw[thick,rounded corners,color=blue](5/2, -3)--(5/2, -4);
			\draw[thick,rounded corners,color=blue](7/2, -3)--(7/2, -4);
			\draw[thick,rounded corners,color=blue](9/2, -4)--(9/2, -7/2)--(5, -7/2);
			\draw[thick,rounded corners,color=blue](1/2, -4)--(1/2, -5);
			\draw[thick,rounded corners,color=blue](3/2, -5)--(3/2, -9/2)--(2, -9/2);
			\draw[thick,rounded corners,color=blue](2, -9/2)--(3, -9/2);
			\draw[thick,rounded corners,color=blue](5/2, -4)--(5/2, -5);
			\draw[thick,rounded corners,color=blue](3, -9/2)--(4, -9/2);
			\draw[thick,rounded corners,color=blue](7/2, -4)--(7/2, -5);
			\draw[thick,rounded corners,color=blue](4, -9/2)--(5, -9/2);
			\draw[thick,rounded corners,color=blue](9/2, -4)--(9/2, -5);
			\node at (1/2, -11/2){1};
			\node at (11/2, -1/2){1};
			\node at (3/2, -11/2){2};
			\node at (11/2, -3/2){3};
			\node at (5/2, -11/2){3};
			\node at (11/2, -5/2){4};
			\node at (7/2, -11/2){4};
			\node at (11/2, -7/2){5};
			\node at (9/2, -11/2){5};
			\node at (11/2, -9/2){2};
		\end{tikzpicture}
		\hspace{2em}
		\begin{tikzpicture}[x=1.25em,y=1.25em]
			\draw[step=1,gray, thin] (0,0) grid (5, -5);
			\draw[color=black, thick] (0,0) rectangle (5, -5);
			\draw[thick,rounded corners,color=blue](3/2, -1)--(3/2, -1/2)--(2, -1/2);
			\draw[thick,rounded corners,color=blue](2, -1/2)--(3, -1/2);
			\draw[thick,rounded corners,color=blue](3, -1/2)--(4, -1/2);
			\draw[thick,rounded corners,color=blue](4, -1/2)--(5, -1/2);
			\draw[thick,rounded corners,color=blue](3/2, -1)--(3/2, -2);
			\draw[thick,rounded corners,color=blue](5/2, -2)--(5/2, -3/2)--(3, -3/2);
			\draw[thick,rounded corners,color=blue](3, -3/2)--(4, -3/2);
			\draw[thick,rounded corners,color=blue](4, -3/2)--(5, -3/2);
			\draw[thick,rounded corners,color=blue](3/2, -2)--(3/2, -3);
			\draw[thick,rounded corners,color=blue](5/2, -2)--(5/2, -3);
			\draw[thick,rounded corners,color=blue](7/2, -3)--(7/2, -5/2)--(4, -5/2);
			\draw[thick,rounded corners,color=blue](4, -5/2)--(5, -5/2);
			\draw[thick,rounded corners,color=blue](1/2, -4)--(1/2, -7/2)--(1, -7/2);
			\draw[thick,rounded corners,color=blue](3/2, -3)--(3/2, -7/2)--(1, -7/2);
			\draw[thick,rounded corners,color=blue](5/2, -3)--(5/2, -4);
			\draw[thick,rounded corners,color=blue](7/2, -3)--(7/2, -4);
			\draw[thick,rounded corners,color=blue](9/2, -4)--(9/2, -7/2)--(5, -7/2);
			\draw[thick,rounded corners,color=blue](1/2, -4)--(1/2, -5);
			\draw[thick,rounded corners,color=blue](3/2, -5)--(3/2, -9/2)--(2, -9/2);
			\draw[thick,rounded corners,color=blue](2, -9/2)--(3, -9/2);
			\draw[thick,rounded corners,color=blue](5/2, -4)--(5/2, -5);
			\draw[thick,rounded corners,color=blue](3, -9/2)--(4, -9/2);
			\draw[thick,rounded corners,color=blue](7/2, -4)--(7/2, -5);
			\draw[thick,rounded corners,color=blue](4, -9/2)--(5, -9/2);
			\draw[thick,rounded corners,color=blue](9/2, -4)--(9/2, -5);
			\node at (1/2, -11/2){1};
			\node at (11/2, -1/2){1};
			\node at (3/2, -11/2){2};
			\node at (11/2, -3/2){3};
			\node at (5/2, -11/2){3};
			\node at (11/2, -5/2){4};
			\node at (7/2, -11/2){4};
			\node at (11/2, -7/2){5};
			\node at (9/2, -11/2){5};
			\node at (11/2, -9/2){2};
		\end{tikzpicture}
		\]
		The corresponding co-BPDs, listed in the same relative order as their BPDs, are below.
		\[
		\begin{tikzpicture}[x=1.25em,y=1.25em]
			\draw[step=1,gray, thin] (0,0) grid (5, -5);
			\draw[color=black, thick] (0,0) rectangle (5, -5);
			\draw[thick,rounded corners,color=ForestGreen](1/2, 0)--(1/2, -1/2)--(1, -1/2);
			\draw[thick,rounded corners,color=ForestGreen](1, -1/2)--(2, -1/2);
			\draw[thick,rounded corners,color=ForestGreen](3/2, 0)--(3/2, -1);
			\draw[thick,rounded corners,color=ForestGreen](2, -1/2)--(3, -1/2);
			\draw[thick,rounded corners,color=ForestGreen](5/2, 0)--(5/2, -1);
			\draw[thick,rounded corners,color=ForestGreen](3, -1/2)--(4, -1/2);
			\draw[thick,rounded corners,color=ForestGreen](7/2, 0)--(7/2, -1);
			\draw[thick,rounded corners,color=ForestGreen](4, -1/2)--(5, -1/2);
			\draw[thick,rounded corners,color=ForestGreen](9/2, 0)--(9/2, -1);
			\draw[thick,rounded corners,color=ForestGreen](3/2, -1)--(3/2, -2);
			\draw[thick,rounded corners,color=ForestGreen](5/2, -1)--(5/2, -3/2)--(3, -3/2);
			\draw[thick,rounded corners,color=ForestGreen](3, -3/2)--(4, -3/2);
			\draw[thick,rounded corners,color=ForestGreen](7/2, -1)--(7/2, -2);
			\draw[thick,rounded corners,color=ForestGreen](4, -3/2)--(5, -3/2);
			\draw[thick,rounded corners,color=ForestGreen](9/2, -1)--(9/2, -2);
			\draw[thick,rounded corners,color=ForestGreen](3/2, -2)--(3/2, -3);
			\draw[thick,rounded corners,color=ForestGreen](7/2, -2)--(7/2, -5/2)--(4, -5/2);
			\draw[thick,rounded corners,color=ForestGreen](4, -5/2)--(5, -5/2);
			\draw[thick,rounded corners,color=ForestGreen](9/2, -2)--(9/2, -3);
			\draw[thick,rounded corners,color=ForestGreen](3/2, -3)--(3/2, -4);
			\draw[thick,rounded corners,color=ForestGreen](9/2, -3)--(9/2, -7/2)--(5, -7/2);
			\draw[thick,rounded corners,color=ForestGreen](3/2, -4)--(3/2, -9/2)--(2, -9/2);
			\draw[thick,rounded corners,color=ForestGreen](2, -9/2)--(3, -9/2);
			\draw[thick,rounded corners,color=ForestGreen](3, -9/2)--(4, -9/2);
			\draw[thick,rounded corners,color=ForestGreen](4, -9/2)--(5, -9/2);
			\node at (1/2, 1/2){1};
			\node at (11/2, -1/2){1};
			\node at (3/2, 1/2){2};
			\node at (11/2, -3/2){3};
			\node at (5/2, 1/2){3};
			\node at (11/2, -5/2){4};
			\node at (7/2, 1/2){4};
			\node at (11/2, -7/2){5};
			\node at (9/2, 1/2){5};
			\node at (11/2, -9/2){2};
		\end{tikzpicture}
		\hspace{2em}
		\begin{tikzpicture}[x=1.25em,y=1.25em]
			\draw[step=1,gray, thin] (0,0) grid (5, -5);
			\draw[color=black, thick] (0,0) rectangle (5, -5);
			\draw[thick,rounded corners,color=ForestGreen](1/2, 0)--(1/2, -1);
			\draw[thick,rounded corners,color=ForestGreen](3/2, 0)--(3/2, -1/2)--(2, -1/2);
			\draw[thick,rounded corners,color=ForestGreen](2, -1/2)--(3, -1/2);
			\draw[thick,rounded corners,color=ForestGreen](5/2, 0)--(5/2, -1);
			\draw[thick,rounded corners,color=ForestGreen](3, -1/2)--(4, -1/2);
			\draw[thick,rounded corners,color=ForestGreen](7/2, 0)--(7/2, -1);
			\draw[thick,rounded corners,color=ForestGreen](4, -1/2)--(5, -1/2);
			\draw[thick,rounded corners,color=ForestGreen](9/2, 0)--(9/2, -1);
			\draw[thick,rounded corners,color=ForestGreen](1/2, -1)--(1/2, -3/2)--(1, -3/2);
			\draw[thick,rounded corners,color=ForestGreen](3/2, -2)--(3/2, -3/2)--(1, -3/2);
			\draw[thick,rounded corners,color=ForestGreen](5/2, -1)--(5/2, -3/2)--(3, -3/2);
			\draw[thick,rounded corners,color=ForestGreen](3, -3/2)--(4, -3/2);
			\draw[thick,rounded corners,color=ForestGreen](7/2, -1)--(7/2, -2);
			\draw[thick,rounded corners,color=ForestGreen](4, -3/2)--(5, -3/2);
			\draw[thick,rounded corners,color=ForestGreen](9/2, -1)--(9/2, -2);
			\draw[thick,rounded corners,color=ForestGreen](3/2, -2)--(3/2, -3);
			\draw[thick,rounded corners,color=ForestGreen](7/2, -2)--(7/2, -5/2)--(4, -5/2);
			\draw[thick,rounded corners,color=ForestGreen](4, -5/2)--(5, -5/2);
			\draw[thick,rounded corners,color=ForestGreen](9/2, -2)--(9/2, -3);
			\draw[thick,rounded corners,color=ForestGreen](3/2, -3)--(3/2, -4);
			\draw[thick,rounded corners,color=ForestGreen](9/2, -3)--(9/2, -7/2)--(5, -7/2);
			\draw[thick,rounded corners,color=ForestGreen](3/2, -4)--(3/2, -9/2)--(2, -9/2);
			\draw[thick,rounded corners,color=ForestGreen](2, -9/2)--(3, -9/2);
			\draw[thick,rounded corners,color=ForestGreen](3, -9/2)--(4, -9/2);
			\draw[thick,rounded corners,color=ForestGreen](4, -9/2)--(5, -9/2);
			\node at (1/2, 1/2){1};
			\node at (11/2, -1/2){2};
			\node at (3/2, 1/2){2};
			\node at (11/2, -3/2){3};
			\node at (5/2, 1/2){3};
			\node at (11/2, -5/2){4};
			\node at (7/2, 1/2){4};
			\node at (11/2, -7/2){5};
			\node at (9/2, 1/2){5};
			\node at (11/2, -9/2){1};
		\end{tikzpicture}
		\hspace{2em}
		\begin{tikzpicture}[x=1.25em,y=1.25em]
			\draw[step=1,gray, thin] (0,0) grid (5, -5);
			\draw[color=black, thick] (0,0) rectangle (5, -5);
			\draw[thick,rounded corners,color=ForestGreen](1/2, 0)--(1/2, -1);
			\draw[thick,rounded corners,color=ForestGreen](3/2, 0)--(3/2, -1/2)--(2, -1/2);
			\draw[thick,rounded corners,color=ForestGreen](2, -1/2)--(3, -1/2);
			\draw[thick,rounded corners,color=ForestGreen](5/2, 0)--(5/2, -1);
			\draw[thick,rounded corners,color=ForestGreen](3, -1/2)--(4, -1/2);
			\draw[thick,rounded corners,color=ForestGreen](7/2, 0)--(7/2, -1);
			\draw[thick,rounded corners,color=ForestGreen](4, -1/2)--(5, -1/2);
			\draw[thick,rounded corners,color=ForestGreen](9/2, 0)--(9/2, -1);
			\draw[thick,rounded corners,color=ForestGreen](1/2, -1)--(1/2, -2);
			\draw[thick,rounded corners,color=ForestGreen](5/2, -1)--(5/2, -3/2)--(3, -3/2);
			\draw[thick,rounded corners,color=ForestGreen](3, -3/2)--(4, -3/2);
			\draw[thick,rounded corners,color=ForestGreen](7/2, -1)--(7/2, -2);
			\draw[thick,rounded corners,color=ForestGreen](4, -3/2)--(5, -3/2);
			\draw[thick,rounded corners,color=ForestGreen](9/2, -1)--(9/2, -2);
			\draw[thick,rounded corners,color=ForestGreen](1/2, -2)--(1/2, -5/2)--(1, -5/2);
			\draw[thick,rounded corners,color=ForestGreen](3/2, -3)--(3/2, -5/2)--(1, -5/2);
			\draw[thick,rounded corners,color=ForestGreen](7/2, -2)--(7/2, -5/2)--(4, -5/2);
			\draw[thick,rounded corners,color=ForestGreen](4, -5/2)--(5, -5/2);
			\draw[thick,rounded corners,color=ForestGreen](9/2, -2)--(9/2, -3);
			\draw[thick,rounded corners,color=ForestGreen](3/2, -3)--(3/2, -4);
			\draw[thick,rounded corners,color=ForestGreen](9/2, -3)--(9/2, -7/2)--(5, -7/2);
			\draw[thick,rounded corners,color=ForestGreen](3/2, -4)--(3/2, -9/2)--(2, -9/2);
			\draw[thick,rounded corners,color=ForestGreen](2, -9/2)--(3, -9/2);
			\draw[thick,rounded corners,color=ForestGreen](3, -9/2)--(4, -9/2);
			\draw[thick,rounded corners,color=ForestGreen](4, -9/2)--(5, -9/2);
			\node at (1/2, 1/2){1};
			\node at (11/2, -1/2){2};
			\node at (3/2, 1/2){2};
			\node at (11/2, -3/2){3};
			\node at (5/2, 1/2){3};
			\node at (11/2, -5/2){4};
			\node at (7/2, 1/2){4};
			\node at (11/2, -7/2){5};
			\node at (9/2, 1/2){5};
			\node at (11/2, -9/2){1};
		\end{tikzpicture}
		\hspace{2em}
		\begin{tikzpicture}[x=1.25em,y=1.25em]
			\draw[step=1,gray, thin] (0,0) grid (5, -5);
			\draw[color=black, thick] (0,0) rectangle (5, -5);
			\draw[thick,rounded corners,color=ForestGreen](1/2, 0)--(1/2, -1);
			\draw[thick,rounded corners,color=ForestGreen](3/2, 0)--(3/2, -1/2)--(2, -1/2);
			\draw[thick,rounded corners,color=ForestGreen](2, -1/2)--(3, -1/2);
			\draw[thick,rounded corners,color=ForestGreen](5/2, 0)--(5/2, -1);
			\draw[thick,rounded corners,color=ForestGreen](3, -1/2)--(4, -1/2);
			\draw[thick,rounded corners,color=ForestGreen](7/2, 0)--(7/2, -1);
			\draw[thick,rounded corners,color=ForestGreen](4, -1/2)--(5, -1/2);
			\draw[thick,rounded corners,color=ForestGreen](9/2, 0)--(9/2, -1);
			\draw[thick,rounded corners,color=ForestGreen](1/2, -1)--(1/2, -2);
			\draw[thick,rounded corners,color=ForestGreen](5/2, -1)--(5/2, -3/2)--(3, -3/2);
			\draw[thick,rounded corners,color=ForestGreen](3, -3/2)--(4, -3/2);
			\draw[thick,rounded corners,color=ForestGreen](7/2, -1)--(7/2, -2);
			\draw[thick,rounded corners,color=ForestGreen](4, -3/2)--(5, -3/2);
			\draw[thick,rounded corners,color=ForestGreen](9/2, -1)--(9/2, -2);
			\draw[thick,rounded corners,color=ForestGreen](1/2, -2)--(1/2, -3);
			\draw[thick,rounded corners,color=ForestGreen](7/2, -2)--(7/2, -5/2)--(4, -5/2);
			\draw[thick,rounded corners,color=ForestGreen](4, -5/2)--(5, -5/2);
			\draw[thick,rounded corners,color=ForestGreen](9/2, -2)--(9/2, -3);
			\draw[thick,rounded corners,color=ForestGreen](1/2, -3)--(1/2, -7/2)--(1, -7/2);
			\draw[thick,rounded corners,color=ForestGreen](3/2, -4)--(3/2, -7/2)--(1, -7/2);
			\draw[thick,rounded corners,color=ForestGreen](9/2, -3)--(9/2, -7/2)--(5, -7/2);
			\draw[thick,rounded corners,color=ForestGreen](3/2, -4)--(3/2, -9/2)--(2, -9/2);
			\draw[thick,rounded corners,color=ForestGreen](2, -9/2)--(3, -9/2);
			\draw[thick,rounded corners,color=ForestGreen](3, -9/2)--(4, -9/2);
			\draw[thick,rounded corners,color=ForestGreen](4, -9/2)--(5, -9/2);
			\node at (1/2, 1/2){1};
			\node at (11/2, -1/2){2};
			\node at (3/2, 1/2){2};
			\node at (11/2, -3/2){3};
			\node at (5/2, 1/2){3};
			\node at (11/2, -5/2){4};
			\node at (7/2, 1/2){4};
			\node at (11/2, -7/2){5};
			\node at (9/2, 1/2){5};
			\node at (11/2, -9/2){1};
		\end{tikzpicture}
		\]
		By \cref{thm:main1},
		\[\mathfrak G_{13452}=\mathfrak S_{13452}-3\mathfrak S_{23451}.\]
		By \cref{thm:main2},
		\[\mathfrak S_{13452}=\mathfrak G_{13452}+3\mathfrak G_{23451}.\]
		Notice that these expressions are not multiplicity free. 
	\end{example}
		
		\section{Expansion of back stable Grothendieck polynomials into back stable Schubert polynomials}
		\label{section:backstable}
		
		\cref{thm:main1} extends naturally to the setting of back stable Schubert calculus.
		We refer the reader to \cite{Lam.Lee.Shimozono,Lam.Lee.Shimozono-K-theory} for background.
		
		\subsection{The ring of back symmetric formal power series}
		
		Write $S_{\mathbb{Z}}$ for the set of permutations of $\mathbb{Z}$ that fix all but finitely many elements. 
		The set $S_{\mathbb{Z}}$ is a group under composition of functions and is generated by the set of simple reflections $\{s_i= (i \, i+1) : i \in \mathbb{Z}\}$.
		
		Given $p, q \in \mathbb{Z}$ with $p \leq q$, we say that $w\in S_{\mathbb{Z}}$ is \newword{supported on $[p,q]$} if $w(i) = i$ for all $i\notin [p,q]$. 
		Whenever $w \in S_{\mathbb{Z}}$ is supported on $[p,q]$, we may regard $w$ as an element of $S_{[p,q]}$, the group of permutations of the set $[p,q] = \{p,p + 1,\ldots,q\}$. 
		In particular, if $w$ is supported on $[1,n]$, we may identify it with its restriction to the symmetric group $S_n$.
		
		Our setting involves certain formal power series in the variables $\{x_i : i \in \mathbb{Z}\}$. 
		We say
		such a power series $f$ has \newword{bounded degree} if there exists $M \in \mathbb{Z}_{\geq 0}$ such that the total degree of each monomial in $f$ is at most $M$. 
		Also, $f$ has \newword{bounded support} if there exists $N \in \mathbb{Z}$ such that $x_i$ does not appear in $f$ for all $i > N$.
		Let $R$ denote the set of formal power series that have bounded degree and bounded support.

		An element $f \in R$ is \newword{back symmetric} if there exists some integer $c$ so that $s_i \cdot f = f$ for all $i < c$. 
		We write $\overleftarrow{R} \subsetneq R$ for the subset of \newword{back symmetric formal power series}.
		Lam, Lee, and Shimozono showed that 
		\[\overleftarrow{R} = \Lambda \otimes \mathbb Q[\ldots, x_{-1}, x_0, x_1,\ldots],\] where $\Lambda$ denotes the ring of symmetric functions in the variables $\{\ldots, x_{-2}, x_{-1}, x_0\}$ \cite[Proposition 3.1]{Lam.Lee.Shimozono}. 
				
		\subsection{Back stable Schubert and Grothendieck polynomials}
		
		Let $w \in S_{\mathbb{Z}}$. 
		Write $\redword(w)$ for the set of reduced words for $w$, and let 
		\[\heckeword(w) = \{\mathbf{a} = (a_1,a_2,\ldots,a_L):\demprod(\mathbf{a}) = w\}\] 
		denote the set of \newword{Hecke words} for $w$. 
		By \cite[Theorem 3.2]{Lam.Lee.Shimozono}, the \newword{back stable Schubert polynomial} has the monomial expansion
		\begin{equation}
			\label{eqn:backschubertcompat}
			\overleftarrow{\mathfrak{S}}_w = \sum_{(a_1,a_2,\ldots,a_{\ell(w)}) \in\redword(w)}
			\sum_{\substack{b_1\leq b_2\leq \cdots \leq b_{\ell(w)}\\
					a_i\leq a_{i + 1} \Rightarrow b_i < b_{i + 1}\\
					b_i\leq a_i}
			} 
		x_{b_1} x_{b_2} \cdots x_{b_{\ell(w)}}.
		\end{equation}
		We take this expansion as a definition here, although $\overleftarrow{\mathfrak{S}}_w$ is equivalently the limit of certain shifts of Schubert polynomials. 
		Lam, Lee, and Shimozono showed that back stable Schubert polynomials form a $\mathbb{Q}$-linear basis for $\overleftarrow{R}$ \cite[Theorem~3.5]{Lam.Lee.Shimozono}.
		
		Similarly, by \cite[Proposition~4.4]{Lam.Lee.Shimozono-K-theory}, the \newword{back stable Grothendieck polynomial} has the monomial expansion
		\begin{equation}
			\label{eqn:backgrothcompat}
			\overleftarrow{\mathfrak{G}}_w = \sum_{(a_1,a_2,\ldots,a_L) \in\heckeword(w)}(-1)^{L-\ell(w)}
			\sum_{\substack{
					b_1\leq b_2\leq \cdots \leq b_L \\
					a_i\leq a_{i + 1}\Rightarrow b_i < b_{i + 1}\\
					b_i\leq a_i}
				} 
			x_{b_1} x_{b_2} \cdots x_{b_L}.
		\end{equation}
		Again, $\overleftarrow{\mathfrak{G}}_w$ is the limit of certain shifts of Grothendieck polynomials.
		Note that for $w \neq \id$, $\overleftarrow{\mathfrak{G}}_w$ is not an element of $R$, since it is not of bounded degree. 
		However, the degree $d$ component of $\overleftarrow{\mathfrak{G}}_w$ is a back symmetric formal power series, and it expands as a finite sum of back stable Schubert polynomials. 
		See \cite[Section~4.1]{Lam.Lee.Shimozono-K-theory} for a description of the back stable ring in which $\overleftarrow{\mathfrak{G}}_w$ resides.
		
		Given integers $p \leq q$, we define
		\[\mathfrak{S}^{[p,q]}_w = \sum_{(a_1,a_2,\ldots,a_{\ell(w)}) \in\redword(w)}
		\sum_{\substack{
				p\leq b_1\leq b_2\leq \cdots \leq b_{\ell(w)}\leq q\\
				a_i\leq a_{i + 1}\Rightarrow b_i < b_{i + 1}\\
				b_i\leq a_i}
		} 
		x_{b_1}x_{b_2}\cdots x_{b_{\ell(w)}}\]
		and
		\[\mathfrak{G}^{[p,q]}_w = \sum_{(a_1,a_2,\ldots,a_L) \in\heckeword(w)}(-1)^{L-\ell(w)}
		\sum_{\substack{p\leq b_1\leq b_2\leq \cdots \leq b_L \leq q\\
				a_i\leq a_{i + 1}\Rightarrow b_i < b_{i + 1}\\
				b_i\leq a_i}
			} 
		x_{b_1}x_{b_2}\cdots x_{b_L}.\] 
				
		Given $w\in \mathbb S_{\mathbb Z}$, fix $q$ so that $w(i) = i$ for all $i > q$. 
		Then $\mathfrak{S}^{[p,q]}_w$ and $\mathfrak{G}^{[p,q]}_w$ are nonzero if and only if $w$ is supported on $[p,q]$.
		Furthermore, in this case, these polynomials are simply evaluations of ordinary Schubert and Grothendieck polynomials on the alphabet $x_p, x_{p + 1},\ldots, x_{q}$, for an appropriate shift of $w$. 
		Namely, if $w'\in S_{[p,q]}$ is defined by $w'(i)=w(i+p-1)-p+1$, then
		\[\mathfrak S_w^{[p,q]}=\mathfrak S_{w'}(x_p,x_{p+1},\ldots,x_q)\]
		and
		\[\mathfrak G_w^{[p,q]}=\mathfrak G_{w'}(x_p,x_{p+1},\ldots,x_q).\]
		In particular, when $p = 1$ and $q=n$, we recover the usual compatible sequence definitions of the ordinary Schubert and Grothendieck polynomials, as given in \cite{BJS93} and \cite{Fomin.Kirillov}, respectively.
				
		\subsection{Back stable BPDs}
		
		We now recall the back stable version of BPDs from \cite{Lam.Lee.Shimozono,Lam.Lee.Shimozono-K-theory}. 
		An \newword{$S_{\mathbb{Z}}$ BPD} is a tiling of $\mathbb{Z} \times \mathbb{Z}$ with the six tiles from \cref{eqn:BPDtiles}, forming a network of pipes labeled by $\mathbb{Z}$ such that:
		\begin{enumerate}
			\item pipes do not start or end within the grid,
			\item for each column $c$, there exists a row $r_c$ such that the tile at position $(r,c)$ is a \vtile for all $r > r_c$,
			\item for each row $r$, there exists a column $c_r$ such that the tile at position $(r,c)$ is a \htile for all $c > c_r$, and
			\item there exist integers $p \leq q$ such that for all $i \notin [p,q]$, the tile at position $(i,i)$ is a \drtile, and the pipe that passes through this tile is a hook.
		\end{enumerate}
		If the fourth condition holds for a particular choice of $p \leq q$, we say that the $S_{\mathbb{Z}}$ BPD is \newword{supported} on $[p,q]$.
		An $S_{\mathbb{Z}}$ BPD is \newword{reduced} if each pair of pipes crosses at most one time.
		
		We may assign a permutation to an $S_{\mathbb{Z}}$ BPD $\mathcal{B}$ by labeling each pipe with the index of the column in which it is eventually vertical, and propagating labels across crossings using the local rules for NE planar histories, as in \cref{eqn:NEpipecross}. 
		We obtain an associated permutation $w \in S_{\mathbb{Z}}$ by setting $w(i) = j$ if the pipe labeled $j$ is eventually horizontal in row $i$. 
		We write $\demprod(\mathcal{B}) = w$. 
		
		Given $w \in S_{\mathbb{Z}}$, write $\overleftarrow{\kbpd}(w)$ for the set of $S_{\mathbb{Z}}$ BPDs of $w$, and $\overleftarrow{\bpd}(w)$ for the set of reduced $S_{\mathbb{Z}}$ BPDs of $w$. 
		Note that if $\mathcal{B} \in \overleftarrow{\kbpd}(w)$ is supported on $[p, q]$, then $w$ is also supported on $[p,q]$. 
		We write $\kbpd^{[p,q]}(w) \subseteq \overleftarrow{\kbpd}(w)$ and $\bpd^{[p,q]}(w) \subseteq \overleftarrow{\bpd}(w)$ for the subsets of $S_{\mathbb{Z}}$ BPDs supported on $[p,q]$. 
		
		It is immediate from the definitions that restricting an element of $\kbpd^{[p,q]}(w)$ to the rectangle $[p,q]\times[p,q]$ defines a bijection from $\kbpd^{[p,q]}(w)$ to $\kbpd(w')$ where $w'\in S_{q-p+1}$ and $w'(i)=w(i+p-1)-p+1$. 
		
		\subsection{Back stable co-BPDs}
		
		In this section, we introduce a back stable version of co-BPDs. 
		An \newword{$S_{\mathbb{Z}}$ co-BPD} is a tiling of the $\mathbb{Z} \times \mathbb{Z}$ grid with the six tiles from \cref{eqn:coBPDtiles}, forming a network of pipes labeled by $\mathbb{Z}$ so that:
		\begin{enumerate}
			\item pipes do not start or end within the grid,
			\item for each column $c$, there exists a row $r_c$ such that the tile at position $(r,c)$ is a \covtile for all $r < r_c$,
			\item for each row $r$, there exists a column $c_r$ such that the tile at position $(r,c)$ is a \cohtile for all $c > c_r$, and
			\item there exist integers $p \leq q$ such that for all $i \notin [p,q]$, the tile at position $(i,i)$ is a \courtile, and the pipe that passes through this tile is a hook.
		\end{enumerate}
		If the fourth condition holds for a particular choice of $p \leq q$, we say that the $S_{\mathbb{Z}}$ co-BPD is \newword{supported} on $[p,q]$.
		An $S_{\mathbb{Z}}$ co-BPD is \newword{reduced} if each pair of pipes crosses at most one time.
		
		Note that in contrast to $S_{\mathbb{Z}}$ BPDs, each $S_{\mathbb{Z}}$ co-BPD contains infinitely many crossings. 
		Nevertheless, we can still read off a permutation $w\in S_{\mathbb{Z}}$ from an $S_{\mathbb{Z}}$ co-BPD $\mathcal B$ as follows. 
		First, choose integers $p \leq q$ such that $\mathcal B$ is supported on $[p,q]$. 
		If $i\notin [p,q]$, then the pipe passing through $(i,i)$ is a hook, and we set $w(i)=i$. 
		Otherwise, within the rectangle $[p,q]\times [p,q]$, we label each pipe by the column in which it enters the rectangle, and propagate labels across crossings using the local rules for SE planar histories, as in \cref{eqn:SEpipecross}. 
		If a pipe labeled $j$ exits the rectangle in row $i$, then we set $w(i)=j$. 
		We define $\demprod(\mathcal{B}) = w$. 
		Note that $\demprod(\mathcal{B})$ is independent of the choice of $p$ and $q$.

		Making the same tile-by-tile replacements as in \cref{lemma:copipedreambijection} induces a bijection from $S_{\mathbb{Z}}$ BPDs to $S_{\mathbb{Z}}$ co-BPDs. 
		Given $\mathcal{B}\in\overleftarrow{\kbpd}(w)$, we write $\check{\mathcal{B}}$ for the image of $\mathcal{B}$ under this map. 
		Note that $\mathcal{B}$ is supported on $[p,q]$ if and only if $\check{\mathcal{B}}$ is supported on $[p,q]$.
		
		In contrast to the situation with finite BPDs, we no longer have a well-defined map from $S_{\mathbb{Z}}$ BPDs to $S_{\mathbb{Z}}$ co-BPDs by reflection.

		\subsection{The back stable Grothendieck to back stable Schubert expansion}
		
		We now state and prove the main theorem of this section.
		
		\begin{theorem}
			\label{thm:backstable}
			Given $w \in S_{\mathbb{Z}}$, we have 
			\[\overleftarrow{\mathfrak{G}}_w = \sum_{\substack{\mathcal{B} \in \overleftarrow{\kbpd}(w)\\\check{\mathcal{B}} \text{ is reduced}}} (-1)^{\ell(\demprod(\check{\mathcal{B}}))-\ell(w)} \overleftarrow{\mathfrak{S}}_{\demprod(\check{\mathcal{B}})}.\]
		\end{theorem}
		
		\begin{proof}
			Because each homogeneous component of $\overleftarrow{\mathfrak{G}}_w$ is back symmetric, by \cite[Theorem~3.5]{Lam.Lee.Shimozono}, we may write 
			\begin{equation}
				\label{eqn:backstablelinearcomb}
				\overleftarrow{\mathfrak{G}}_w = \sum_{v \in S_{\mathbb{Z}}} (-1)^{\ell(v)-\ell(w)} d_{w,v}\overleftarrow{\mathfrak{S}}_v,
			\end{equation} 
			so that in each degree, we have a finite linear combination of back stable Schubert polynomials. 
			In other words, for each $k$, the set 
			\[\{v\in S_{\mathbb Z}:\ell(v) = k \text{ and } d_{w,v} \neq 0\}\]
			is finite. 
			
			Given $k\in \mathbb Z_{\geq 0}$, choose integers $p \leq q$ so that $w$ is supported on $[p,q]$, and also $v$ is supported on $[p,q]$ for all $v\in S_{\mathbb Z}$ such that $\ell(v)=k$ and $d_{w,v}\neq 0$. We can do this because the set of such $v$ is finite.
			Now specialize both sides of \cref{eqn:backstablelinearcomb}, setting $x_i = 0$ for all $i\notin [p,q]$, which yields
			\[\mathfrak{G}_w^{[p,q]} = \sum_{v \in S_{\mathbb{Z}}} (-1)^{\ell(v)-\ell(w)}d_{w,v} \mathfrak{S}^{[p,q]}_v.\]
			On the other hand, as an immediate consequence of \cref{thm:main1}, we also have
			\[\mathfrak{G}_w^{[p,q]} = \sum_{\substack{\mathcal{B} \in \kbpd^{[p,q]}(w)\\ \check{\mathcal{B}} \text{ is reduced}}} (-1)^{\ell(\demprod(\check{\mathcal{B}}))-\ell(w)} \mathfrak{S}^{[p,q]}_{\demprod(\check{\mathcal{B}})}.\]
			Since the set $\{\mathfrak{S}_v^{[p,q]}: v\in S_{\mathbb Z} \text{ is supported on } [p,q]\}$ is linearly independent, it follows that if $\ell(v)=k$, then 
			\[d_{w,v} = \#\{\mathcal{B} \in \kbpd^{[p,q]}(w) : \check{\mathcal{B}} \text{ is reduced and } \demprod(\check{\mathcal{B}}) = v\}.\]
			In particular, if $\mathfrak{S}^{[p,q]}_v \neq 0$, then $\mathfrak{S}^{[p',q]}_v \neq 0$ for all $p' < p$. Thus,
			\[d_{w,v} = \#\{\mathcal{B} \in \kbpd^{[p',q]}(w) : \check{\mathcal{B}} \text{ is reduced and } \demprod(\check{\mathcal{B}}) = v\}\]
			for all $p'<p$ and so
			\[d_{w,v} = \#\{\mathcal{B} \in \overleftarrow{\kbpd}(w) : \check{\mathcal{B}} \text{ is reduced and } \demprod(\check{\mathcal{B}}) = v\}.\]
			Since the choice of $k$ was arbitrary, the theorem follows.
		\end{proof}
		
		\begin{example}
			Let $w = s_2s_1 \in S_{\mathbb{Z}}$. 
			The Rothe $S_{\mathbb Z}$ BPD of $w$ and its associated $S_{\mathbb Z}$ co-BPD, both restricted to $[-1,3] \times [-1,3]$, are shown below.
			
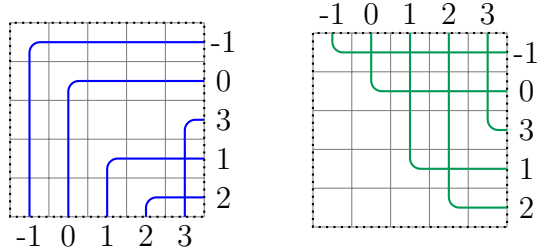
\[\begin{tikzpicture}[x = 1.25em,y = 1.25em]
				\draw[step = 1,gray, thin] (0,0) grid (5, -5);
				\draw[color = black, thick,dotted] (0,0) rectangle (5, -5);
				\draw[thick,rounded corners,color = blue](1/2, -1)--(1/2, -1/2)--(1, -1/2);
				\draw[thick,rounded corners,color = blue](1, -1/2)--(2, -1/2);
				\draw[thick,rounded corners,color = blue](2, -1/2)--(3, -1/2);
				\draw[thick,rounded corners,color = blue](3, -1/2)--(4, -1/2);
				\draw[thick,rounded corners,color = blue](4, -1/2)--(5, -1/2);
				\draw[thick,rounded corners,color = blue](1/2, -1)--(1/2, -2);
				\draw[thick,rounded corners,color = blue](3/2, -2)--(3/2, -3/2)--(2, -3/2);
				\draw[thick,rounded corners,color = blue](2, -3/2)--(3, -3/2);
				\draw[thick,rounded corners,color = blue](3, -3/2)--(4, -3/2);
				\draw[thick,rounded corners,color = blue](4, -3/2)--(5, -3/2);
				\draw[thick,rounded corners,color = blue](1/2, -2)--(1/2, -3);
				\draw[thick,rounded corners,color = blue](3/2, -2)--(3/2, -3);
				\draw[thick,rounded corners,color = blue](9/2, -3)--(9/2, -5/2)--(5, -5/2);
				\draw[thick,rounded corners,color = blue](1/2, -3)--(1/2, -4);
				\draw[thick,rounded corners,color = blue](3/2, -3)--(3/2, -4);
				\draw[thick,rounded corners,color = blue](5/2, -4)--(5/2, -7/2)--(3, -7/2);
				\draw[thick,rounded corners,color = blue](3, -7/2)--(4, -7/2);
				\draw[thick,rounded corners,color = blue](4, -7/2)--(5, -7/2);
				\draw[thick,rounded corners,color = blue](9/2, -3)--(9/2, -4);
				\draw[thick,rounded corners,color = blue](1/2, -4)--(1/2, -5);
				\draw[thick,rounded corners,color = blue](3/2, -4)--(3/2, -5);
				\draw[thick,rounded corners,color = blue](5/2, -4)--(5/2, -5);
				\draw[thick,rounded corners,color = blue](7/2, -5)--(7/2, -9/2)--(4, -9/2);
				\draw[thick,rounded corners,color = blue](4, -9/2)--(5, -9/2);
				\draw[thick,rounded corners,color = blue](9/2, -4)--(9/2, -5);
				\node at (1/2, 1/2){\null};
				\node at (1/2, -11/2){-1};
				\node at (11/2, -1/2){-1};
				\node at (3/2, -11/2){0};
				\node at (11/2, -3/2){0};
				\node at (5/2, -11/2){1};
				\node at (11/2, -5/2){3};
				\node at (7/2, -11/2){2};
				\node at (11/2, -7/2){1};
				\node at (9/2, -11/2){3};
				\node at (11/2, -9/2){2};
			\end{tikzpicture}
			\hspace{2em}
			\begin{tikzpicture}[x = 1.25em,y = 1.25em]
				\draw[step = 1,gray, thin] (0,0) grid (5, -5);
				\draw[color = black, thick, dotted] (0,0) rectangle (5, -5);
				\draw[thick,rounded corners,color = ForestGreen](1/2, 0)--(1/2, -1/2)--(1, -1/2);
				\draw[thick,rounded corners,color = ForestGreen](1, -1/2)--(2, -1/2);
				\draw[thick,rounded corners,color = ForestGreen](3/2, 0)--(3/2, -1);
				\draw[thick,rounded corners,color = ForestGreen](2, -1/2)--(3, -1/2);
				\draw[thick,rounded corners,color = ForestGreen](5/2, 0)--(5/2, -1);
				\draw[thick,rounded corners,color = ForestGreen](3, -1/2)--(4, -1/2);
				\draw[thick,rounded corners,color = ForestGreen](7/2, 0)--(7/2, -1);
				\draw[thick,rounded corners,color = ForestGreen](4, -1/2)--(5, -1/2);
				\draw[thick,rounded corners,color = ForestGreen](9/2, 0)--(9/2, -1);
				\draw[thick,rounded corners,color = ForestGreen](3/2, -1)--(3/2, -3/2)--(2, -3/2);
				\draw[thick,rounded corners,color = ForestGreen](2, -3/2)--(3, -3/2);
				\draw[thick,rounded corners,color = ForestGreen](5/2, -1)--(5/2, -2);
				\draw[thick,rounded corners,color = ForestGreen](3, -3/2)--(4, -3/2);
				\draw[thick,rounded corners,color = ForestGreen](7/2, -1)--(7/2, -2);
				\draw[thick,rounded corners,color = ForestGreen](4, -3/2)--(5, -3/2);
				\draw[thick,rounded corners,color = ForestGreen](9/2, -1)--(9/2, -2);
				\draw[thick,rounded corners,color = ForestGreen](5/2, -2)--(5/2, -3);
				\draw[thick,rounded corners,color = ForestGreen](7/2, -2)--(7/2, -3);
				\draw[thick,rounded corners,color = ForestGreen](9/2, -2)--(9/2, -5/2)--(5, -5/2);
				\draw[thick,rounded corners,color = ForestGreen](5/2, -3)--(5/2, -7/2)--(3, -7/2);
				\draw[thick,rounded corners,color = ForestGreen](3, -7/2)--(4, -7/2);
				\draw[thick,rounded corners,color = ForestGreen](7/2, -3)--(7/2, -4);
				\draw[thick,rounded corners,color = ForestGreen](4, -7/2)--(5, -7/2);
				\draw[thick,rounded corners,color = ForestGreen](7/2, -4)--(7/2, -9/2)--(4, -9/2);
				\draw[thick,rounded corners,color = ForestGreen](4, -9/2)--(5, -9/2);
				\node at (1/2, -11/2){\null};
				\node at (1/2, 1/2){-1};
				\node at (11/2, -1/2){-1};
				\node at (3/2, 1/2){0};
				\node at (11/2, -3/2){0};
				\node at (5/2, 1/2){1};
				\node at (11/2, -5/2){3};
				\node at (7/2, 1/2){2};
				\node at (11/2, -7/2){1};
				\node at (9/2, 1/2){3};
				\node at (11/2, -9/2){2};
			\end{tikzpicture}\]
			Notice that the following $S_{\mathbb Z}$ BPD for $w$ (pictured on the left) has a nonreduced $S_{\mathbb Z}$ co-BPD (pictured on the right).
			
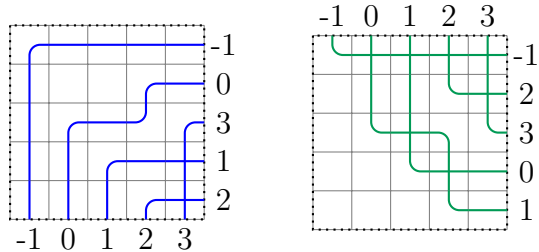
\[
			\begin{tikzpicture}[x = 1.25em,y = 1.25em]
				\draw[step = 1,gray, thin] (0,0) grid (5, -5);
				\draw[color = black, thick,dotted] (0,0) rectangle (5, -5);
				\draw[thick,rounded corners,color = blue](1/2, -1)--(1/2, -1/2)--(1, -1/2);
				\draw[thick,rounded corners,color = blue](1, -1/2)--(2, -1/2);
				\draw[thick,rounded corners,color = blue](2, -1/2)--(3, -1/2);
				\draw[thick,rounded corners,color = blue](3, -1/2)--(4, -1/2);
				\draw[thick,rounded corners,color = blue](4, -1/2)--(5, -1/2);
				\draw[thick,rounded corners,color = blue](1/2, -1)--(1/2, -2);
				\draw[thick,rounded corners,color = blue](7/2, -2)--(7/2, -3/2)--(4, -3/2);
				\draw[thick,rounded corners,color = blue](4, -3/2)--(5, -3/2);
				\draw[thick,rounded corners,color = blue](1/2, -2)--(1/2, -3);
				\draw[thick,rounded corners,color = blue](3/2, -3)--(3/2, -5/2)--(2, -5/2);
				\draw[thick,rounded corners,color = blue](2, -5/2)--(3, -5/2);
				\draw[thick,rounded corners,color = blue](7/2, -2)--(7/2, -5/2)--(3, -5/2);
				\draw[thick,rounded corners,color = blue](9/2, -3)--(9/2, -5/2)--(5, -5/2);
				\draw[thick,rounded corners,color = blue](1/2, -3)--(1/2, -4);
				\draw[thick,rounded corners,color = blue](3/2, -3)--(3/2, -4);
				\draw[thick,rounded corners,color = blue](5/2, -4)--(5/2, -7/2)--(3, -7/2);
				\draw[thick,rounded corners,color = blue](3, -7/2)--(4, -7/2);
				\draw[thick,rounded corners,color = blue](4, -7/2)--(5, -7/2);
				\draw[thick,rounded corners,color = blue](9/2, -3)--(9/2, -4);
				\draw[thick,rounded corners,color = blue](1/2, -4)--(1/2, -5);
				\draw[thick,rounded corners,color = blue](3/2, -4)--(3/2, -5);
				\draw[thick,rounded corners,color = blue](5/2, -4)--(5/2, -5);
				\draw[thick,rounded corners,color = blue](7/2, -5)--(7/2, -9/2)--(4, -9/2);
				\draw[thick,rounded corners,color = blue](4, -9/2)--(5, -9/2);
				\draw[thick,rounded corners,color = blue](9/2, -4)--(9/2, -5);
				\node at (1/2, 1/2){\null};
				\node at (1/2, -11/2){-1};
				\node at (11/2, -1/2){-1};
				\node at (3/2, -11/2){0};
				\node at (11/2, -3/2){0};
				\node at (5/2, -11/2){1};
				\node at (11/2, -5/2){3};
				\node at (7/2, -11/2){2};
				\node at (11/2, -7/2){1};
				\node at (9/2, -11/2){3};
				\node at (11/2, -9/2){2};
			\end{tikzpicture}
			\hspace{2em}
			\begin{tikzpicture}[x = 1.25em,y = 1.25em]
				\draw[step = 1,gray, thin] (0,0) grid (5, -5);
				\draw[color = black, thick,dotted] (0,0) rectangle (5, -5);
				\draw[thick,rounded corners,color = ForestGreen](1/2, 0)--(1/2, -1/2)--(1, -1/2);
				\draw[thick,rounded corners,color = ForestGreen](1, -1/2)--(2, -1/2);
				\draw[thick,rounded corners,color = ForestGreen](3/2, 0)--(3/2, -1);
				\draw[thick,rounded corners,color = ForestGreen](2, -1/2)--(3, -1/2);
				\draw[thick,rounded corners,color = ForestGreen](5/2, 0)--(5/2, -1);
				\draw[thick,rounded corners,color = ForestGreen](3, -1/2)--(4, -1/2);
				\draw[thick,rounded corners,color = ForestGreen](7/2, 0)--(7/2, -1);
				\draw[thick,rounded corners,color = ForestGreen](4, -1/2)--(5, -1/2);
				\draw[thick,rounded corners,color = ForestGreen](9/2, 0)--(9/2, -1);
				\draw[thick,rounded corners,color = ForestGreen](3/2, -1)--(3/2, -2);
				\draw[thick,rounded corners,color = ForestGreen](5/2, -1)--(5/2, -2);
				\draw[thick,rounded corners,color = ForestGreen](7/2, -1)--(7/2, -3/2)--(4, -3/2);
				\draw[thick,rounded corners,color = ForestGreen](4, -3/2)--(5, -3/2);
				\draw[thick,rounded corners,color = ForestGreen](9/2, -1)--(9/2, -2);
				\draw[thick,rounded corners,color = ForestGreen](3/2, -2)--(3/2, -5/2)--(2, -5/2);
				\draw[thick,rounded corners,color = ForestGreen](2, -5/2)--(3, -5/2);
				\draw[thick,rounded corners,color = ForestGreen](5/2, -2)--(5/2, -3);
				\draw[thick,rounded corners,color = ForestGreen](7/2, -3)--(7/2, -5/2)--(3, -5/2);
				\draw[thick,rounded corners,color = ForestGreen](9/2, -2)--(9/2, -5/2)--(5, -5/2);
				\draw[thick,rounded corners,color = ForestGreen](5/2, -3)--(5/2, -7/2)--(3, -7/2);
				\draw[thick,rounded corners,color = ForestGreen](3, -7/2)--(4, -7/2);
				\draw[thick,rounded corners,color = ForestGreen](7/2, -3)--(7/2, -4);
				\draw[thick,rounded corners,color = ForestGreen](4, -7/2)--(5, -7/2);
				\draw[thick,rounded corners,color = ForestGreen](7/2, -4)--(7/2, -9/2)--(4, -9/2);
				\draw[thick,rounded corners,color = ForestGreen](4, -9/2)--(5, -9/2);
				\node at (1/2, -11/2){\null};
				\node at (1/2, 1/2){-1};
				\node at (11/2, -1/2){-1};
				\node at (3/2, 1/2){0};
				\node at (11/2, -3/2){2};
				\node at (5/2, 1/2){1};
				\node at (11/2, -5/2){3};
				\node at (7/2, 1/2){2};
				\node at (11/2, -7/2){0};
				\node at (9/2, 1/2){3};
				\node at (11/2, -9/2){1};
			\end{tikzpicture}\]
			Furthermore, any $S_{\mathbb Z}$ BPD obtained from it by a sequence of droop moves will still have a nonreduced $S_{\mathbb Z}$ co-BPD.
			
			One can check that the $S_{\mathbb Z}$ BPDs contributing to the expansion of $\overleftarrow{\mathfrak{G}}_w$ are the Rothe $S_{\mathbb Z}$ BPD for $w$, plus $S_{\mathbb Z}$ BPDs obtained from a sequence of droop moves applied to the $S_{\mathbb Z}$ BPD below, each of which leaves the pipe with label $0$ fixed.
			
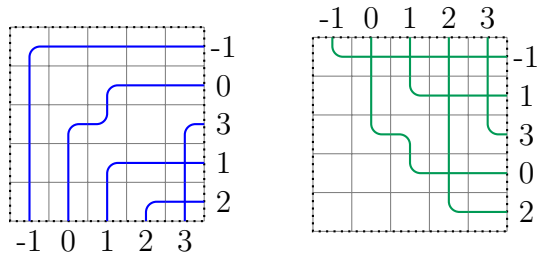
\[\begin{tikzpicture}[x = 1.25em,y = 1.25em]
				\draw[step = 1,gray, thin] (0,0) grid (5, -5);
				\draw[color = black, thick,dotted] (0,0) rectangle (5, -5);
				\draw[thick,rounded corners,color = blue](1/2, -1)--(1/2, -1/2)--(1, -1/2);
				\draw[thick,rounded corners,color = blue](1, -1/2)--(2, -1/2);
				\draw[thick,rounded corners,color = blue](2, -1/2)--(3, -1/2);
				\draw[thick,rounded corners,color = blue](3, -1/2)--(4, -1/2);
				\draw[thick,rounded corners,color = blue](4, -1/2)--(5, -1/2);
				\draw[thick,rounded corners,color = blue](1/2, -1)--(1/2, -2);
				\draw[thick,rounded corners,color = blue](5/2, -2)--(5/2, -3/2)--(3, -3/2);
				\draw[thick,rounded corners,color = blue](3, -3/2)--(4, -3/2);
				\draw[thick,rounded corners,color = blue](4, -3/2)--(5, -3/2);
				\draw[thick,rounded corners,color = blue](1/2, -2)--(1/2, -3);
				\draw[thick,rounded corners,color = blue](3/2, -3)--(3/2, -5/2)--(2, -5/2);
				\draw[thick,rounded corners,color = blue](5/2, -2)--(5/2, -5/2)--(2, -5/2);
				\draw[thick,rounded corners,color = blue](9/2, -3)--(9/2, -5/2)--(5, -5/2);
				\draw[thick,rounded corners,color = blue](1/2, -3)--(1/2, -4);
				\draw[thick,rounded corners,color = blue](3/2, -3)--(3/2, -4);
				\draw[thick,rounded corners,color = blue](5/2, -4)--(5/2, -7/2)--(3, -7/2);
				\draw[thick,rounded corners,color = blue](3, -7/2)--(4, -7/2);
				\draw[thick,rounded corners,color = blue](4, -7/2)--(5, -7/2);
				\draw[thick,rounded corners,color = blue](9/2, -3)--(9/2, -4);
				\draw[thick,rounded corners,color = blue](1/2, -4)--(1/2, -5);
				\draw[thick,rounded corners,color = blue](3/2, -4)--(3/2, -5);
				\draw[thick,rounded corners,color = blue](5/2, -4)--(5/2, -5);
				\draw[thick,rounded corners,color = blue](7/2, -5)--(7/2, -9/2)--(4, -9/2);
				\draw[thick,rounded corners,color = blue](4, -9/2)--(5, -9/2);
				\draw[thick,rounded corners,color = blue](9/2, -4)--(9/2, -5);
				\node at (1/2, 1/2){\null};
				\node at (1/2, -11/2){-1};
				\node at (11/2, -1/2){-1};
				\node at (3/2, -11/2){0};
				\node at (11/2, -3/2){0};
				\node at (5/2, -11/2){1};
				\node at (11/2, -5/2){3};
				\node at (7/2, -11/2){2};
				\node at (11/2, -7/2){1};
				\node at (9/2, -11/2){3};
				\node at (11/2, -9/2){2};
			\end{tikzpicture}
			\hspace{2em}
			\begin{tikzpicture}[x = 1.25em,y = 1.25em]
				\draw[step = 1,gray, thin] (0,0) grid (5, -5);
				\draw[color = black, thick,dotted] (0,0) rectangle (5, -5);
				\draw[thick,rounded corners,color = ForestGreen](1/2, 0)--(1/2, -1/2)--(1, -1/2);
				\draw[thick,rounded corners,color = ForestGreen](1, -1/2)--(2, -1/2);
				\draw[thick,rounded corners,color = ForestGreen](3/2, 0)--(3/2, -1);
				\draw[thick,rounded corners,color = ForestGreen](2, -1/2)--(3, -1/2);
				\draw[thick,rounded corners,color = ForestGreen](5/2, 0)--(5/2, -1);
				\draw[thick,rounded corners,color = ForestGreen](3, -1/2)--(4, -1/2);
				\draw[thick,rounded corners,color = ForestGreen](7/2, 0)--(7/2, -1);
				\draw[thick,rounded corners,color = ForestGreen](4, -1/2)--(5, -1/2);
				\draw[thick,rounded corners,color = ForestGreen](9/2, 0)--(9/2, -1);
				\draw[thick,rounded corners,color = ForestGreen](3/2, -1)--(3/2, -2);
				\draw[thick,rounded corners,color = ForestGreen](5/2, -1)--(5/2, -3/2)--(3, -3/2);
				\draw[thick,rounded corners,color = ForestGreen](3, -3/2)--(4, -3/2);
				\draw[thick,rounded corners,color = ForestGreen](7/2, -1)--(7/2, -2);
				\draw[thick,rounded corners,color = ForestGreen](4, -3/2)--(5, -3/2);
				\draw[thick,rounded corners,color = ForestGreen](9/2, -1)--(9/2, -2);
				\draw[thick,rounded corners,color = ForestGreen](3/2, -2)--(3/2, -5/2)--(2, -5/2);
				\draw[thick,rounded corners,color = ForestGreen](5/2, -3)--(5/2, -5/2)--(2, -5/2);
				\draw[thick,rounded corners,color = ForestGreen](7/2, -2)--(7/2, -3);
				\draw[thick,rounded corners,color = ForestGreen](9/2, -2)--(9/2, -5/2)--(5, -5/2);
				\draw[thick,rounded corners,color = ForestGreen](5/2, -3)--(5/2, -7/2)--(3, -7/2);
				\draw[thick,rounded corners,color = ForestGreen](3, -7/2)--(4, -7/2);
				\draw[thick,rounded corners,color = ForestGreen](7/2, -3)--(7/2, -4);
				\draw[thick,rounded corners,color = ForestGreen](4, -7/2)--(5, -7/2);
				\draw[thick,rounded corners,color = ForestGreen](7/2, -4)--(7/2, -9/2)--(4, -9/2);
				\draw[thick,rounded corners,color = ForestGreen](4, -9/2)--(5, -9/2);
				\node at (1/2, -11/2){\null};
				\node at (1/2, 1/2){-1};
				\node at (11/2, -1/2){-1};
				\node at (3/2, 1/2){0};
				\node at (11/2, -3/2){1};
				\node at (5/2, 1/2){1};
				\node at (11/2, -5/2){3};
				\node at (7/2, 1/2){2};
				\node at (11/2, -7/2){0};
				\node at (9/2, 1/2){3};
				\node at (11/2, -9/2){2};
			\end{tikzpicture}\]
			For instance, we first droop the pipe with label $-1$ and obtain the diagrams below.
			
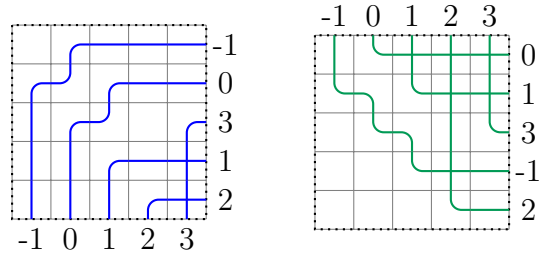
\[\begin{tikzpicture}[x = 1.25em,y = 1.25em]
				\draw[step = 1,gray, thin] (0,0) grid (5, -5);
				\draw[color = black, thick,dotted] (0,0) rectangle (5, -5);
				\draw[thick,rounded corners,color = blue](3/2, -1)--(3/2, -1/2)--(2, -1/2);
				\draw[thick,rounded corners,color = blue](2, -1/2)--(3, -1/2);
				\draw[thick,rounded corners,color = blue](3, -1/2)--(4, -1/2);
				\draw[thick,rounded corners,color = blue](4, -1/2)--(5, -1/2);
				\draw[thick,rounded corners,color = blue](1/2, -2)--(1/2, -3/2)--(1, -3/2);
				\draw[thick,rounded corners,color = blue](3/2, -1)--(3/2, -3/2)--(1, -3/2);
				\draw[thick,rounded corners,color = blue](5/2, -2)--(5/2, -3/2)--(3, -3/2);
				\draw[thick,rounded corners,color = blue](3, -3/2)--(4, -3/2);
				\draw[thick,rounded corners,color = blue](4, -3/2)--(5, -3/2);
				\draw[thick,rounded corners,color = blue](1/2, -2)--(1/2, -3);
				\draw[thick,rounded corners,color = blue](3/2, -3)--(3/2, -5/2)--(2, -5/2);
				\draw[thick,rounded corners,color = blue](5/2, -2)--(5/2, -5/2)--(2, -5/2);
				\draw[thick,rounded corners,color = blue](9/2, -3)--(9/2, -5/2)--(5, -5/2);
				\draw[thick,rounded corners,color = blue](1/2, -3)--(1/2, -4);
				\draw[thick,rounded corners,color = blue](3/2, -3)--(3/2, -4);
				\draw[thick,rounded corners,color = blue](5/2, -4)--(5/2, -7/2)--(3, -7/2);
				\draw[thick,rounded corners,color = blue](3, -7/2)--(4, -7/2);
				\draw[thick,rounded corners,color = blue](4, -7/2)--(5, -7/2);
				\draw[thick,rounded corners,color = blue](9/2, -3)--(9/2, -4);
				\draw[thick,rounded corners,color = blue](1/2, -4)--(1/2, -5);
				\draw[thick,rounded corners,color = blue](3/2, -4)--(3/2, -5);
				\draw[thick,rounded corners,color = blue](5/2, -4)--(5/2, -5);
				\draw[thick,rounded corners,color = blue](7/2, -5)--(7/2, -9/2)--(4, -9/2);
				\draw[thick,rounded corners,color = blue](4, -9/2)--(5, -9/2);
				\draw[thick,rounded corners,color = blue](9/2, -4)--(9/2, -5);
				\node at (1/2, 1/2){\null};
				\node at (1/2, -11/2){-1};
				\node at (11/2, -1/2){-1};
				\node at (3/2, -11/2){0};
				\node at (11/2, -3/2){0};
				\node at (5/2, -11/2){1};
				\node at (11/2, -5/2){3};
				\node at (7/2, -11/2){2};
				\node at (11/2, -7/2){1};
				\node at (9/2, -11/2){3};
				\node at (11/2, -9/2){2};
			\end{tikzpicture}
			\hspace{2em}
			\begin{tikzpicture}[x = 1.25em,y = 1.25em]
				\draw[step = 1,gray, thin] (0,0) grid (5, -5);
				\draw[color = black, thick,dotted] (0,0) rectangle (5, -5);
				\draw[thick,rounded corners,color = ForestGreen](1/2, 0)--(1/2, -1);
				\draw[thick,rounded corners,color = ForestGreen](3/2, 0)--(3/2, -1/2)--(2, -1/2);
				\draw[thick,rounded corners,color = ForestGreen](2, -1/2)--(3, -1/2);
				\draw[thick,rounded corners,color = ForestGreen](5/2, 0)--(5/2, -1);
				\draw[thick,rounded corners,color = ForestGreen](3, -1/2)--(4, -1/2);
				\draw[thick,rounded corners,color = ForestGreen](7/2, 0)--(7/2, -1);
				\draw[thick,rounded corners,color = ForestGreen](4, -1/2)--(5, -1/2);
				\draw[thick,rounded corners,color = ForestGreen](9/2, 0)--(9/2, -1);
				\draw[thick,rounded corners,color = ForestGreen](1/2, -1)--(1/2, -3/2)--(1, -3/2);
				\draw[thick,rounded corners,color = ForestGreen](3/2, -2)--(3/2, -3/2)--(1, -3/2);
				\draw[thick,rounded corners,color = ForestGreen](5/2, -1)--(5/2, -3/2)--(3, -3/2);
				\draw[thick,rounded corners,color = ForestGreen](3, -3/2)--(4, -3/2);
				\draw[thick,rounded corners,color = ForestGreen](7/2, -1)--(7/2, -2);
				\draw[thick,rounded corners,color = ForestGreen](4, -3/2)--(5, -3/2);
				\draw[thick,rounded corners,color = ForestGreen](9/2, -1)--(9/2, -2);
				\draw[thick,rounded corners,color = ForestGreen](3/2, -2)--(3/2, -5/2)--(2, -5/2);
				\draw[thick,rounded corners,color = ForestGreen](5/2, -3)--(5/2, -5/2)--(2, -5/2);
				\draw[thick,rounded corners,color = ForestGreen](7/2, -2)--(7/2, -3);
				\draw[thick,rounded corners,color = ForestGreen](9/2, -2)--(9/2, -5/2)--(5, -5/2);
				\draw[thick,rounded corners,color = ForestGreen](5/2, -3)--(5/2, -7/2)--(3, -7/2);
				\draw[thick,rounded corners,color = ForestGreen](3, -7/2)--(4, -7/2);
				\draw[thick,rounded corners,color = ForestGreen](7/2, -3)--(7/2, -4);
				\draw[thick,rounded corners,color = ForestGreen](4, -7/2)--(5, -7/2);
				\draw[thick,rounded corners,color = ForestGreen](7/2, -4)--(7/2, -9/2)--(4, -9/2);
				\draw[thick,rounded corners,color = ForestGreen](4, -9/2)--(5, -9/2);
				\node at (1/2, -11/2){\null};
				\node at (1/2, 1/2){-1};
				\node at (11/2, -1/2){0};
				\node at (3/2, 1/2){0};
				\node at (11/2, -3/2){1};
				\node at (5/2, 1/2){1};
				\node at (11/2, -5/2){3};
				\node at (7/2, 1/2){2};
				\node at (11/2, -7/2){-1};
				\node at (9/2, 1/2){3};
				\node at (11/2, -9/2){2};
			\end{tikzpicture}\]
			Then we droop the pipe with label $-2$, then $-3$, and so on. From this and \cref{thm:backstable}, we deduce that 
			\[
			\overleftarrow{\mathfrak{G}}_{s_2s_1} = \overleftarrow{\mathfrak{S}}_{s_2s_1}-\overleftarrow{\mathfrak{S}}_{s_{0}s_2s_1} + \overleftarrow{\mathfrak{S}}_{s_{-1}s_{0}s_2s_1}-\overleftarrow{\mathfrak{S}}_{s_{-2}s_{-1}s_{0}s_2s_1} + \cdots. \]
			The reader is invited to compare the expansion of $\overleftarrow{\mathfrak{G}}_{s_2s_1}$ into back stable Schubert polynomials directly with the monomial expansions in \cref{eqn:backschubertcompat} and \cref{eqn:backgrothcompat}.
		\end{example}

		\section*{Acknowledgments}
		
		The author thanks Daoji Huang for helpful conversations and Benjamin Young for providing code that aided this work. 
		This research was partially supported by NSF DMS \#2344764.  
				
		\bibliographystyle{amsalphavar} 
		\bibliography{WLR.bib}
	\end{document}